\documentclass[a4paper,12pt]{article}

\usepackage[english,french]{babel}
\RequirePackage[utf8]{inputenc} 
\RequirePackage[section]{placeins}
\RequirePackage[T1]{fontenc} 
\RequirePackage{mathtools} 
\RequirePackage{siunitx} 
\RequirePackage{float} 
\RequirePackage{graphicx} 
\RequirePackage[justification=centering]{caption} 
\RequirePackage{subcaption}
\RequirePackage{wallpaper}
\RequirePackage{nomencl}
\RequirePackage{fancyhdr}
\RequirePackage{url}
\RequirePackage[hidelinks]{hyperref}
\newcommand{\blind}{0}

\usepackage{float}
\usepackage[T1]{fontenc}
\usepackage[utf8]{inputenc}
\usepackage{authblk}
\usepackage{lipsum}
\usepackage{multirow}
\usepackage{textcomp}
\usepackage{colortbl}
\usepackage{amsmath}
\usepackage{amssymb}
\usepackage{csquotes}
\usepackage{amsthm}
\usepackage{gensymb}
\usepackage{pdfpages}
\usepackage{mathrsfs}
\usepackage{stmaryrd}
\usepackage{subcaption}
\usepackage{version}

\usepackage{tcolorbox}
\usepackage{booktabs}
\usepackage{stmaryrd}
\usepackage{multicol}
\usepackage{pdfpages}
\usepackage{mathtools}
\usepackage{algorithm}
\usepackage{listings}
\usepackage{algorithmic}
\usepackage{tikz}
\usepackage{thmtools}%
\usepackage{bm}
\usepackage{setspace}
\usepackage{soul}
\usepackage{titlesec}
\DeclareMathAlphabet{\mathbbold}{U}{bbold}{m}{n}
\theoremstyle{plain}
\declaretheorem[name=Theorem, numberwithin=section]{theorem}

\declaretheorem[name=Proposition, numberwithin=section]{prop}

\declaretheorem[name=Definition, numberwithin=section]{definition}

\titleformat{\subsection}{\normalfont\large\bfseries\itshape}{\thesubsection}{1em}{}

\newcommand{\Real}{\mathbb{R}}
\newcommand{\Proba}{\mathbb{P}}
\newcommand{\Xdom}{\mathcal{X}}
\newcommand{\Udom}{\mathcal{U}}
\newcommand{\Ydom}{\mathcal{Y}}
\newcommand{\Hdom}{\mathcal{H}}

\newcommand{\Bdom}{\mathcal{B}}
\newcommand{\Zdom}{\mathcal{Z}}

\newcommand{\Adom}{\mathcal{A}}
\newcommand{\Fdom}{\mathcal{F}}
\newcommand{\Esp}{\mathbb{E}}

\definecolor{forestgreen}{rgb}{0.1333, 0.5451, 0.1333} 
\newcommand{\UE}[1]{\renewcommand{\UE}{#1}}
\newcommand{\sujet}[1]{\renewcommand{\sujet}{#1}}
\newcommand{\titre}[1]{\renewcommand{\titre}{#1}}
\newcommand{\enseignant}[1]{\renewcommand{\enseignant}{#1}}
\newcommand{\eleves}[1]{\renewcommand{\eleves}{#1}}

\usepackage[normalem]{ulem}

\makeatletter
\newskip\@bigflushglue \@bigflushglue = -100pt plus 1fil

\def\bigcentering{\let\\\@centercr\rightskip\@bigflushglue%
\leftskip\@bigflushglue
\parindent\z@\parfillskip\z@skip}

\makeatother

\title{Kernel-based Sensitivity Analysis for (Excursion) Sets} 
\usepackage[style=authoryear,maxbibnames=9,maxcitenames=2,uniquelist=false,backend=biber,url=false,giveninits=true,doi=false]{biblatex}
\DeclareNameAlias{sortname}{last-first}

\renewbibmacro*{in:}{%
  \ifentrytype{article}{}{\clearfield{in:}}%
}
\DeclareFieldFormat{pages}{#1}
\newrobustcmd*{\parentexttrack}[1]{%
  \begingroup
  \blx@blxinit
  \blx@setsfcodes
  \blx@bibopenparen#1\blx@bibcloseparen
  \endgroup}

\AtEveryCite{%
  }

\makeatother

\title{Kernel-based Sensitivity Analysis for (excursion) sets}
\author[1,2]{N. Fellmann}
\author[1]{C. Blanchet-Scalliet}
\author[1]{C. Helbert}
\author[2]{A. Spagnol}
\author[2]{D. Sinoquet}
\affil[1]{École Centrale de Lyon, CNRS UMR 5208, Institut Camille Jordan, 36 Avenue Guy de Collongue,
69134 Écully, France}
\affil[2]{IFP Energies Nouvelles}
\affil[ ]{\textit{\{noe.fellmann,christophette.blanchet,celine.helbert\}@ec-lyon.fr}}
\affil[ ]{\textit{\{noe.fellmann,delphine.sinoquet,adrien.spagnol\}@ifpen.fr}}

\begin{document}
\def\spacingset#1{\renewcommand{\baselinestretch}%
{#1}\small\normalsize} \spacingset{1}
\renewcommand{\labelitemi}{$\bullet$} 
\renewcommand{\labelitemii}{$\star$} 
\selectlanguage{english}
\if0\blind{
\maketitle
}\fi

\if1\blind
{
\bigskip
  \bigskip
  \bigskip
  \begin{center}
    {\LARGE\bf Kernel-based Sensitivity Analysis for (excursion) sets}
\end{center}
  \medskip
} \fi

\bigskip

\begin{abstract}

 In this paper, we aim to perform sensitivity analysis of set-valued models and, in particular, to quantify the impact of uncertain inputs on feasible sets, which are key elements in solving a robust optimization problem under constraints. While most sensitivity analysis methods deal with scalar outputs, this paper introduces a novel approach for performing sensitivity analysis with set-valued outputs. Our innovative methodology is designed for excursion sets, but is versatile enough to be applied to set-valued simulators, including those found in viability fields, or when working with maps like pollutant concentration maps or flood zone maps.

We propose to use the Hilbert-Schmidt Independence Criterion (HSIC) with a kernel designed for set-valued outputs. After proposing a probabilistic framework for random sets, a first contribution is the proof that this kernel is \textit{characteristic}, an essential property in a kernel-based sensitivity analysis context. To measure the contribution of each input, we then propose to use HSIC-ANOVA indices. With these indices, we can identify which inputs should be neglected (\textit{screening}) and we can rank the others according to their influence (\textit{ranking}). The estimation of these indices is also adapted to the set-valued outputs. Finally, we test the proposed method on three test cases of excursion sets.
\end{abstract}
\noindent%
\vfill

\newpage


\section{Introduction }

In many fields, it is essential to understand the input/output relationships of models that simulate complex physical systems. This knowledge can be used to simplify and optimize the model and provide valuable insights to experts. Sensitivity analysis (SA) is one response to this challenge. It quantifies how variations in inputs translate into variations in outputs, precisely measuring the impact of each input on the output. However, in certain domains, models may have highly complex outputs that traditional SA methods, originally designed for scalar outputs, may not be suitable for. Specifically, we are interested in models with set-valued outputs, where each evaluation of the model produces a subset of a larger space. Our interest lies in the need to quantify the influence of parameters on the excursion sets of optimization problems, with the goal of simplifying constrained robust optimization problems. Set-valued output models are also prevalent in other domains, including image processing, map modeling, and viability fields. Therefore, we aim to measure the impact of each input, on a set-valued output through an adapted SA approach.

Global sensitivity analysis (GSA) is a methodology used to assess the impact of input variations on the output of a system or model across the entire parameter space. A detailed review of GSA methods can be found in the book \textcite{livre_SA} or in \textcite{ReviewSA}. These methods can be distinguished into two types: 
\begin{itemize}
    \item \textit{screening}-oriented SA techniques are devoted to identifying influential and non-influential inputs;
    \item \textit{ranking}-oriented SA aims to compute sensitivity indices or importance measures, which are scalars representing the effect of an input or group of inputs on the output. These indices are then used to rank the inputs by their influence.
\end{itemize} One well-known screening technique is the Morris method (\textcite{morris}), which is based on one-at-a-time designs, i.e., where each input varies while the others are held fixed. Regarding ranking methods, the most commonly used indices are the Sobol' indices, which are variance-based sensitivity measures (\textcite{sobol_global_2001}). These indices quantify the portion of the output variance that can be attributed to one or a group of inputs. This variance decomposition is called the ANOVA (ANalysis Of VAriance) decomposition.  However, these indices have several drawbacks.  One main issue is that their estimation cost in terms of the number of model evaluations is very high. For instance, the simulation budget needed for the 'pick and freeze' method (\textcite{Sobol1993SensitivityEF}') increases linearly with the input dimension. Rank-based estimators have recently been used to circumvent this problem (\textcite{gamboa_global_2021}), but at the cost of not being able to estimate the total-order Sobol' indices. Another major drawback is that Sobol' indices quantify the input contributions to the output variance, but not to the entire output distribution. That is why other types of indices examine how the entire output distribution is affected by the input parameters. This is done by comparing the input and output probability distributions. For instance, Borgonovo indices compare the density functions (\textcite{BORGONOVO2007771}), Cramér von Mises indices look at the cumulative distribution function (\textcite{cramervonmisesindex}), and kernel-based sensitivity indices compare the embedding of the distributions (\textcite{daveiga:hal-03108628}).  The former method relies on embedding distributions in Reproducing Kernel Hilbert Spaces (RKHS) using kernel functions, which makes it easier to compute distances. By embedding the input and output distributions in the RKHS, indices based on the Hilbert-Schmidt Independence Criterion (HSIC) (\textcite{gretton_hsic2}) can be defined. These indices quantify the dependence between an input and the output and can be used for screening. They also have a low cost because they can be estimated using a single sample. Under certain assumptions, an ANOVA-like decomposition of HSIC exists, as shown in \textcite{daveiga:hal-03108628} which makes them usable for ranking.

Several works in the literature have addressed SA adapted to complex outputs. However, to the best of our knowledge, the specific case of set-valued outputs has not yet been studied. For instance, \textcite{higdon2008computer}, \textcite{marrel2011global}, and \textcite{perrin2021functional} examined spatial outputs, while \textcite{de2017sensitivity} and \textcite{marrel2015development} explored spatiotemporal outputs. These methods generate sensitivity index maps that can be interpreted by connecting them to the underlying physical phenomena. However, there is a strong motivation, whether due to interpretation challenges or synthesis concerns, to find a single scalar index per input that captures its influence on the variability of the entire complex output. This is explored and presented in \textcite{Sobol_multivariate}, which defines aggregated Sobol' indices for vectorial and functional outputs. In \textcite{Sa_gen_met_space} the authors also propose an universal index that can be used in any metric space. Kernel-based indices, in particular the HSIC, can also be used with complex outputs, as proposed in \textcite{veiga2015} and in \textcite{amri_morel} with functional outputs. Adaptation of HSIC to complex output is in fact straightforward, kernel methods being known to be flexible with respect to the type of data (see for example \textcite{shawe-taylor_basic_2004}, where kernel methods are used with several types of data such as vectors, texts, trees...). The flexibility of the latter lies in the fact that they mostly rely only on defining a kernel in the relevant space (where random elements can be defined), making them particularly permissive in accommodating different input and output types. For this flexibility, low cost, ability to quantify changes across the entire output distribution, and their utility for both screening and ranking purposes,  we propose to use HSIC-ANOVA indices to conduct a SA of set-valued models. 

To perform kernel-based SA on a set-valued model, there are two main requirements. First, a framework of random sets is required. Indeed, a space of sets is not an easy space to work in, and describing randomness in such a space can be challenging. For example, \textcite{molchanov_random_2017} and \textcite{nguyen_introduction_2006} propose a complete theory of random sets. The second key to performing kernel-based SA on set-valued outputs is to have a kernel defined on sets. Given these two conditions, we can apply the kernel-based SA methodology to define sensitivity indices for set-valued outputs. This is the focus of this paper.

To this end, we first recall in Section \ref{sec:generic} the general methodology for defining kernel-based sensitivity indices, and in particular the Hilbert-Schmidt Independence Criterion (HSIC), in a generic measurable output space. We then present our contributions in Section \ref{sec:kset}. First, we define the probabilistic framework for random sets. Then, we introduce a new kernel defined on a space of sets and study it in detail. In particular, we show that our kernel is \textit{characteristic}, a crucial property for screening purposes. Finally, HSIC-ANOVA indices for set-valued outputs are derived, but their estimation is complicated by the presence of sets. Therefore, a nested Monte Carlo estimator is proposed to estimate the indices and its statistical properties are studied. Finally, numerical results obtained from two test cases related to excursion sets and from an industrial test case for electrical machine design are given in Section \ref{numerical}. Proofs of the results of section \ref{sec:kset} and some additional numerical results are given in the appendix.
 

\section{Kernel-based SA in a generic space $\Zdom$}
\label{sec:generic}
In this section, we explain how HSIC-ANOVA indices can be derived from kernel theory in a generic space. For a complete kernel theory, the reader can refer to the book \textit{Support Vector Machine} of \textcite{kernels_steinwart}.
\subsection{Kernel Embedding, MMD and HSIC}
\label{subsec:Kernel embedding and MMD}
Let $(\Zdom, \Adom)$ be a measurable space on which $k_{\Zdom}$ is a kernel, i.e., a symmetric and positive definite function $k_{\Zdom} : \Zdom \times \Zdom \rightarrow \Real $. As stated in Theorem 4.21 of \textcite{kernels_steinwart}, there exists a unique Reproducing Kernel Hilbert Space (RKHS) $\Hdom_{k_{\Zdom}}$ of reproducing kernel $k_{\Zdom}$, i.e a kernel such that
\begin{itemize}
    \item $\forall z \in \Zdom, ~~ k_{\Zdom}(\cdot,z) \in \Hdom_{k_{\Zdom}}$
    \item $\forall z \in \Zdom, ~~ \forall f \in \Hdom_{k_{\Zdom}}, ~~\langle f, k_{\Zdom}(\cdot,z) \rangle_{\Hdom_{k_{\Zdom}}} = f(z) $.
\end{itemize}
Let $\mathcal M_+^1 (\Zdom)$ be the space of probability measures on $\Zdom$. By Lemma 3.1 of \textcite{review_kernel_mean_embedding}, a \textit{bounded} and \textit{measurable} kernel $k_{\Zdom}$ is a sufficient condition for $\mathcal M_+^1 (\Zdom)$ to be embedded in $\Hdom_{k_{\Zdom}}$ by the kernel mean embedding defined in the following definition.

\begin{definition}
\label{mean_embedding}
The kernel mean embedding of $\mathcal M_+^1 (\Zdom)$ in $\Hdom_{k_{\Zdom}}$ is defined as
\begin{equation*}
\begin{matrix}
 & \mathcal M_+^1 (\Zdom)  & \rightarrow & \Hdom_{k_{\Zdom}} \\ 
\mu_{k_{\Zdom}} ~ : & \Proba & \mapsto & \mu_{k_{\Zdom}}(\Proba) = \int_{\Zdom}  k_{\Zdom}(z, \cdot) d\Proba(z).
\end{matrix}
\end{equation*}
\end{definition}

The distance between the kernel mean embedding of two distributions $\Proba$ and $\mathbb Q$ is called the Maximum Mean Discrepancy (MMD) and is denoted by $\operatorname{MMD}_{k_{\Zdom}}(\Proba, \mathbb Q) = ||\mu_{k_{\Zdom}}(\Proba)-\mu_{k_{\Zdom}}(\mathbb Q)||_{\Hdom_{k_{\Zdom}}}$ (see Figure \ref{fig:mean_embedding}). To be a distance between the distributions $\Proba$ and $\mathbb Q$, the kernel mean embedding $\mu_{k_{\Zdom}}$ must be injective. In this case, the kernel $\mu_{k_{\Zdom}}$ is said to be \textit{characteristic}.
\begin{figure}[h]
    \centering
    \begin{tikzpicture}[scale=2]

\draw[->] (-1.2,0) -- (1.7,0) node[pos=0.5,below] {$\mathcal{Z}$};

\draw[violet, thick, domain=-1.1:1.6, samples=100] plot (\x, {0.8*exp(-\x*\x*2)});
\draw[forestgreen, thick, domain=-1.1:1.6, samples=100] plot (\x, {0.6*exp(-(\x-0.5)*(\x-0.5)*3)});

\node[violet] at (0,0.9) {$\mathbb P$};
\node[forestgreen] at (0.5,0.7) {$\mathbb Q$};

 \draw[fill=gray!20, opacity=0.7] (1.7,1.2)
        to[out=0,in=180] (2.5,1.3)
        to[out=0,in=45] (3.2,1)
        to[out=-135,in=0] (2.8,0)
        to[out=180,in=-90] (2,0)
        to[out=90,in=-135] (1.5,0.5)
        to[out=45,in=180] (1.7,1.2);

\node at (2.8,1) {$\mathcal H_{k_{\mathcal Z}}$};

\draw[->, thick, violet] (0.15,0.9) to[out=30, in=150] (1.7,1) ;
\draw[->, thick, forestgreen] (0.65,0.7) to[out=30, in=200] (2.3,0.2) ;


\node[draw, circle, fill=violet, inner sep=0pt, minimum size=3pt] at (1.75,0.95) {};
\node[violet] at (2,1.05) { $\scriptstyle \mu_{k_{\mathcal Z}}(\mathbb P)$};
\node[draw, circle, fill=forestgreen, inner sep=0pt, minimum size=3pt] at (2.35,0.25) {};
\node[forestgreen] at (2.6,0.35) { $\scriptstyle \mu_{k_{\mathcal Z}}(\mathbb Q)$};

\draw[>=stealth,<->] (1.75,0.95) -- (2.35,0.25) node[pos=0.45, right] { $\scriptstyle \operatorname{MMD}_{k_{\mathcal Z}}(\textcolor{violet}{\mathbb P}, \textcolor{forestgreen}{\mathbb Q})$};

\end{tikzpicture}
    \caption{Kernel mean embedding}
    \label{fig:mean_embedding}
\end{figure}
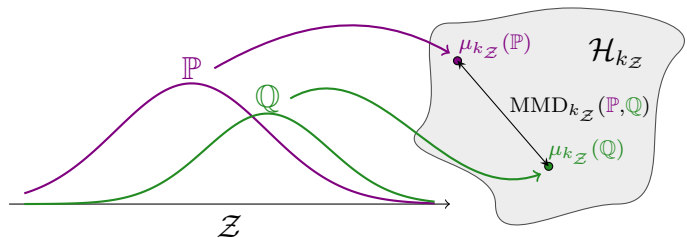
The MMD can then be used to define HSIC for the sensitivity analysis context, as proposed in \textcite{veiga2015}.

Let $\eta : \Udom \rightarrow \Zdom$ be a measurable model with respect to the Borel $\sigma$ algebras $\Bdom_{\bm U}$ and $\Bdom_{\Zdom}$, where $\Udom = \Udom_1 \times 
... \times \Udom_p \subset \Real^p$. Let $(\Omega, \Fdom, \Proba)$ be a probability space on which each input $U_i$ is a random variable of probability distribution $\Proba_{U_i}$. For any subset of indices $A \subset \{1,...,p\}$, $\bm U_A$ will denote the random vector $(U_i)_{i \in A}$. We are interested in knowing how the uncertainty of $Z=\eta(\bm U)$ can be attributed to the different inputs (or group of inputs). To do this, kernel-based sensitivity analysis relies on measuring the dependence between $\bm U_A$ and $Z$ by computing the MMD between the joint distribution $\Proba_{\bm U_A,Z}$ and the product of the marginal distributions $\Proba_{\bm U_A} \otimes \Proba_Z$. If this distance is zero, then $\bm U_A$ and $Z$ are independent, i.e. $\bm U_A$ has no effect on $Z$. This distance is called the Hilbert-Schmidt Independence Criterion (HSIC), first defined in \textcite{Gretton_hsic1}.

\begin{definition}
\label{hsic_def}
For $A \subset \{1,...,p\}$, let $k=k_{A} \otimes k_{\Zdom}$ be a kernel inducing the RKHS $\Hdom_k$, with $k_{A}$ a kernel on the space $\Udom_A=\otimes_{i\in A} \Udom_i$. The Hilbert-Schmidt Independence Criterion between $\bm U_A$ and $Z$, denoted by $\operatorname{HSIC}(\bm U_A,Z)$, is defined as
\begin{equation*}
\operatorname{HSIC}(\bm U_A,Z):=\operatorname{MMD}(\Proba_{\bm U_A,Z},\Proba_{\bm U_A}\otimes\Proba_{Z})^2=||\mu_k(\bm U_A,Z)-\mu_{k_A}(\bm U_A) \otimes \mu_{k_{\Zdom}}(Z)||^2_{\Hdom_k}.
\end{equation*}
\end{definition}

HSIC can then be used to test the independence between $\bm U_A$ and $Z$. For the test to be consistent, i.e. $\operatorname{HSIC}(\bm U_A,Z)=0$ if and only if $\bm U_A$ and $Z$ are independent, it is sufficient that both $k_{A}$ and $k_{\Zdom}$ are characteristic (which doesn't necessarily mean that $k$ is characteristic). This result can be obtained by combining Theorem 3.11 in \textcite{Lyons_2013} and Proposition 29 in \textcite{Sejdinovic_2013} as explained in the introduction of \textcite{szabo_characteristic_2018}.
To conduct screening, an independence test is performed for each input $U_i$ : $(\mathcal H^i_0):~ \operatorname{HSIC}(U_i,Z)=0 ~~~~\text{versus}~~~~  (\mathcal H^i_1):~ \operatorname{HSIC}(U_i,Z)>0.$

One of the reasons HSIC has become popular is that it can be expressed very simply in terms of kernel functions. From Lemma 1 of \textcite{Gretton_hsic1} we have
\begin{align}
\label{hsic_sum_expect1}
\operatorname{HSIC}(\bm U_A,Z)= &~\Esp[k_{A}(\bm U_A,{\bm U_A}' )k_{\Zdom}(Z,Z')]\\
&+\Esp[k_{A}(\bm U_A,{\bm U_A}')]\Esp[k_{\Zdom}(Z,Z')] \nonumber\\
&-2 \Esp[ \Esp[k_{A}(\bm U_A,{\bm U_A}')|\bm U_A]\Esp[k_{\Zdom}(Z,Z')|Z]], \nonumber
\end{align}
where $({\bm U_A}',Z')$ is an independent copy of $(\bm U_A,Z)$. This expression makes estimating the HSIC very easy, since it can be estimated with biased or unbiased estimators based on U- and V-statistics as introduced in \textcite{Gretton_hsic1} and \textcite{song_supervised_2007}. If the sample size is large enough, asymptotic estimation of the p-value associated to the independence test can be used, as suggested in \textcite{gretton_kernel_2007}. For smaller samples, a permutation-based technique can be used (see \textcite{delozzo2014new}). 

However, HSICs are not sufficient for ranking. The values $\operatorname{HSIC}(U_i,Z)$ must first be normalized in order to compare them and rank the inputs. In the context of sensitivity analysis, several normalizations have been studied, such as normalization by $\sum_{i=1}^p \operatorname{HSIC}(U_i,Z)$ (see \textcite{spagnolarticle}) or by $\sqrt{\operatorname{HSIC}(U_i,U_i)\operatorname{HSIC}(Z,Z)}$ (\textcite{veiga2015}). These indices have been used for ranking purposes, but they are not as satisfactory as the ANOVA decomposition of Sobol' indices. However, a recent study (\textcite{daveiga:hal-03108628}) provides an ANOVA-like decomposition that justifies the use of HSIC for ranking purposes. 

\subsection{HSIC-ANOVA indices}
\label{subsec:ranking}
In \textcite{daveiga:hal-03108628} an ANOVA-like decomposition of the HSIC is proposed, which makes it usable to rank the inputs by influence. This decomposition requires strong assumptions, particularly the independence between the inputs and the ANOVA property of the input kernel.
\begin{definition}[Orthogonal and ANOVA kernel]
Let $\Udom$ be a measurable space. A kernel $k:\Udom \times \Udom \rightarrow \Real$ is said to be orthogonal with respect to a probability measure $\nu \in \mathcal M_1^+(\Udom)$ if
\begin{equation*}
\forall u \in \Udom, \int_{\Udom} k(u,z)d\nu(z)=0.
\end{equation*}

A kernel $K:\Udom \times \Udom \rightarrow \Real$ is said to be ANOVA w.r.t. $\nu$ if it can be decomposed as $K=1+k$ where $k$ is orthogonal w.r.t. $\nu$.
\end{definition}

\begin{theorem}[ANOVA decomposition of $\operatorname{HSIC}$]
\label{anovahsic}
Assuming that:
\begin{itemize}
\item[i.]The inputs $U_1,...,U_p$ are mutually independent;
\item[ii.]Each input has an ANOVA kernel $K_i$ w.r.t. the input distribution $\Proba_{U_i}$. For any group of inputs $\bm U_A$ with $A \subset \{1,...,p\}$, the associated kernel is defined by
$$K_A = \bigotimes_{i\in A} K_i;$$
\item[iii.] For any $A \subset \{1,...,p\}$, $\Esp [K_A(\bm U_A, \bm U_A)] < + \infty$ and  $\Esp[k_{\Zdom}(Z,Z)] < + \infty$.
\end{itemize}
Then the ANOVA decomposition of the $\operatorname{HSIC}$ is given by

$$\operatorname{HSIC}(\boldsymbol{U}, Z)=\sum_{A \subseteq\{1, \ldots, p\}} \sum_{B \subseteq A}(-1)^{|A|-|B|} \operatorname{HSIC}\left(\boldsymbol{U}_B,Z\right).$$
\end{theorem}
In addition to providing a decomposition of the HSIC, an ANOVA input kernel also simplifies the HSIC expression of the Equation \ref{hsic_sum_expect1}, which becomes
\begin{equation*}
    \operatorname{HSIC}(\bm U_A,Z)= \Esp \left[ (K_A(\bm U_A,{\bm U_A}')-1)k_{\Zdom}(Z,Z') \right].
\end{equation*}
Given an independent and identically distributed (iid) sample $(\bm U_A^{(1)},Z^{(1)}),...,(\bm U_A^{(n)},Z^{(n)})$ of $(\bm U_A, Z)$, the previous expression is easily estimated by a U-statistic

\begin{equation}
\label{unbiased estimator}
\widehat{\operatorname{HSIC}}\left(\bm U_A, Z\right)=\frac{2}{n(n-1)} \sum_{i <j}^{n}\left(K_{A}\left(U_A^{(i)}, U_A^{(j)}\right)-1\right) k_{\Zdom}\left(Z^{(i)}, Z^{(j)}\right).
\end{equation}

The decomposition given in Theorem \ref{anovahsic} allows SA indices to be defined in a similar way to Sobol' indices.
\begin{definition}[HSIC-ANOVA indices]
The first-order and total-order HSIC-ANOVA indices can be defined by
$$\forall i \in \{1,...,p\}~ S^{\operatorname{HSIC}}_i :=\frac{\operatorname{HSIC}(U_i,Z)}{\operatorname{HSIC}(\bm U,Z)} \text{  and  } S^{\operatorname{HSIC}}_{T_i} :=1-\frac{\operatorname{HSIC}(\bm U_{-i},Z)}{\operatorname{HSIC}(\bm U,Z)} ,
$$
with $\bm U_{-i}=(U_1,...,U_{i-1},U_{i+1},...,U_d)$. These indices can be generalized to groups of inputs.
\end{definition}
These indices can be used to rank the inputs by influence by ranking their indices. It is also possible to use either  first-order or total-order indices for screening, as we still have 
$$S^{\operatorname{HSIC}}_{i}=0 \Longleftrightarrow S^{\operatorname{HSIC}}_{T_i}=0 \Longleftrightarrow U_i \perp \Gamma,$$
as shown in \textcite{hsic_anova_test}. The authors also propose associated independence tests that have better statistical power than the usual tests performed with HSIC indices. 

It is important to note that the three assumptions of Theorem \ref{anovahsic}, and especially ii., restrict the choice of the input kernel to be ANOVA. Few kernels are known to be ANOVA. The best known are the Sobolev kernels (see Theorem 3.3 of \textcite{sarazin:cea-04320711}), which correspond to the reproducing kernel of some Sobolev spaces. 
These kernels are ANOVA with respect to the uniform law on $[0,1]$ and are also characteristic (Remark 3.4 and Proposition 3.6 of \textcite{sarazin:cea-04320711}). It is also possible to derive ANOVA kernels from classic kernels. One of the possible transformations is suggested in \textcite{anova_kernel}: for any given kernel $k$, its ANOVA counterpart $K$ w.r.t. $\nu$ is defined by
\begin{align}
\label{kernelanovatransfo}
    K(x,y)=1+k(x,y)&-\int k(x,z) d\nu(z) - \int k(z,y) d\nu(z) \\
    &+ \int \int  k(z,z') d\nu(z)d\nu(z'). \nonumber
\end{align}
 In the case of input variables uniformly distributed on $[0,1]$, the previous transformation is known analytically for some classic kernels given in the appendix of \textcite{anova_kernel}. 

 Given ANOVA and characteristic input kernels $K_i$ and a characteristic output kernel $k_{\Zdom}$, first-order and total-order HSIC-ANOVA indices can be computed and used for both screening and ranking. This approach has the strength of being very permissive about the nature of the output. Indeed, once one has a measurable and bounded characteristic kernel on an output space on which a random element can be defined, one can derive HSIC-ANOVA sensitivity indices to perform a sensitivity analysis of the complex model. We propose to do so in the next section dealing with an output space of sets.


\section{Kernel-based Sensitivity Analysis for sets}
\label{sec:kset}
In this section, we apply the methodology of the previous section to the case of set-valued outputs, i.e., where $\Zdom$ is a space of sets. We introduce a probabilistic framework for random sets on which we propose a customized kernel and study its property so that it can be used to define HSIC-ANOVA indices for sets. We also introduce and study the estimation of these new indices. Basic concepts of functional analysis and topology useful in this section can be found, for example, in \textcite{book:functionalanalysis}.

\subsection{Probabilistic framework of random sets}
\label{subsec:random_set}
A general theory of random sets can be found in \textcite{molchanov_random_2017}. In our case, however, we propose to define a random set as $\Gamma=\eta(\bm U)$, where $\eta$ is a measurable function from $\Udom$ to a space of sets to be defined. First, we define the space of sets we are working in, and a corresponding $\sigma$-algebra.

Let $\mathscr L(\Xdom)$ be the space of all Lebesgue-measurable subsets of a compact space $\Xdom \subset \Real ^d$.  Let $\delta : \mathscr L(\Xdom) \times \mathscr L(\Xdom) \rightarrow \Real$ be the Lebesgue measure of the symmetric difference defined by $\delta(\gamma_1, \gamma_2) = \lambda(\gamma_1\Delta \gamma_2)$ where $\Delta$ is the symmetric difference defined by $\gamma_1\Delta \gamma_2= (\gamma_1\cup \gamma_2) \backslash (\gamma_1\cap \gamma_2)$ and $\lambda$ is the Lebesgue measure on $\Xdom$. $L^2(\Xdom)$ denotes the space of squared Lebesgue integrable functions on $\Xdom$, and the $L^2$ norm is denoted by $||\cdot||_2$. $\delta$ can also be seen in terms of the $L^2$ norm of the indicator functions of $\gamma_1$ and $\gamma_2$, as given in the following lemma. 
\begin{restatable}{lemme}{lemma}
     Let $\gamma_1,\gamma_2 \in \mathscr L (\Xdom)$, we have
 $$
 \delta(\gamma_1,\gamma_2) = \left|\left|\mathbbold 1_{\gamma_1} - \mathbbold 1_{\gamma_2}\right| \right|^2_{2},
 $$
 and 
 $$
 \delta(\gamma_1,\gamma_2)=0  \Leftrightarrow \gamma_1 = \gamma_2 ~\lambda \text{-almost everywhere.}
 $$
 \label{lemme}
\end{restatable}
Thus $\delta$ is a pseudo-metric in $\mathscr L(X)$, since $\delta(\gamma_1,\gamma_2)=0$ only implies that $\gamma_1=\gamma_2$ $\lambda$-almost everywhere. To work in a metric space, we quotient $\mathscr L (\Xdom)$ by the equivalence relation $\delta(\gamma_1,\gamma_2)=0$. The resulting quotient space is called $\mathscr L^* (\Xdom)$. $\delta$ remains well defined on $\mathscr L_2^\star (\Xdom)=\mathscr L^\star (\Xdom)\times \mathscr L^\star (\Xdom)$ and becomes a distance sometimes called the Fréchet-Nikodym-Aronszajn metric (see \textcite{Marczewski1958OnAC}). With this distance, $\mathscr L^*(\Xdom)$ can be provided with the Borel $\sigma$-algebra $\mathcal B_{\Gamma}:= \Bdom ( \mathscr L^*(\Xdom),\delta)$ to make it a measurable space. A random set is then simply defined as $\Gamma = \eta (\bm U)$ where $\eta : \Udom \rightarrow \mathscr L^*(\Xdom)$ is measurable with respect to $\Bdom_{\bm U}$ and $\Bdom_{\Gamma}$. The probability distribution of a random set $\Gamma$ is then defined as the push-forward probability measure of $\Proba_{\Udom}$ through $\eta$.

Examples of random set are numerous in industrial applications, especially in the context of optimization or inversion, where determining the feasible set of solutions is crucial. In the presence of uncertainties, this feasible set is random and can be called an excursion set as given in the following definition.
\begin{restatable}{definition}{excursionrandom}
\label{excursion_random}
Let $g: \Xdom \times \Udom \rightarrow \Real$ be a measurable function. The excursion set $\Gamma_g$ associated with the constraint $g \leq 0$ is defined by
\begin{equation*}
\Gamma_g= \{x \in X, ~~g(x,\bm U) \leq 0 \}.
\end{equation*} 
\end{restatable}

\subsection{A kernel between sets}
\label{subsec:k_set}
Now that we have a framework of random sets, we want to follow the methodology of Section \ref{sec:generic} but with $\mathcal Z = \mathscr L^*(\Xdom)$ to define HSIC indices for sets. To do this, we first need a kernel on sets, i.e. a symmetric and positive definite function $k : \mathscr L^*(\Xdom)\times \mathscr L^*(\Xdom)\rightarrow \Real$. We propose the function $k_{set}$ defined by
\begin{equation*}
    \forall \gamma_1,\gamma_2 \in \mathscr L^* (\Xdom), ~~ k_{set}(\gamma_1,\gamma_2):= e^{- \frac{\lambda(\gamma_1\Delta \gamma_2)}{2\sigma^2}},
\end{equation*}
 where $\sigma$ is a positive scalar.  $k_{set}$ is inspired by the classic Gaussian kernel, since the Lebesgue measure of the symmetric difference is equal to the $L^2$ norm of the difference between two indicator functions, as shown in Lemma \ref{lemme}. This function was shown to be positive definite in \textcite{set-index-process-herbin} when the space of sets is an indexing collection, but the proof is similar in our case. Moreover, $k_{set}$ is bounded and measurable.
 
\begin{restatable}{prop}{ksetkernel}
\label{prop_k_def_pos}
The function $k_{set}: \mathscr L^* (\Xdom) \times \mathscr L^* (\Xdom) \rightarrow \Real$ defined by
$$\forall \gamma_1,\gamma_2 \in \mathscr L^* (\Xdom), ~~ k_{set}(\gamma_1,\gamma_2)=e^{- \frac{\lambda(\gamma_1\Delta \gamma_2)}{2\sigma^2}},$$ is symmetric and positive definite for any positive scalar $\sigma$ which means that $k_{set}$ is a kernel.
\end{restatable}

\begin{restatable}{prop}{kboundedmes}
\label{prop_k_bounded_mes}
The kernel $k_{set}$ is bounded and measurable from $\mathscr L_2^\star (\Xdom)=\mathscr L^\star (\Xdom) \times \mathscr L^\star (\Xdom)$ to $\Real$.
\end{restatable}
As in Definition \ref{mean_embedding}, this allows to define the mean embedding of random set distributions $\mu_{k_{set}}$. We then show that it is injective, i.e. $k_{set}$ is characteristic.

\begin{restatable}{prop}{charackernel}
The kernel $k_{set}$ is characteristic. 
\label{charac}
\end{restatable}
We now have all the ingredients to use HSIC-ANOVA indices for set-valued outputs. We can now screen the inputs using independence tests and rank them by their influence on the random set $\Gamma$ using either first-order or total-order indices, that we denote $S^{\operatorname{H}_{set}}_{i}$ and $S^{\operatorname{H}_{set}}_{T_i}$. HSIC-ANOVA indices applied to sets provide a solution for the screening and ranking of the inputs of a set-valued model. However, to achieve these purposes, they must first be estimated, but the presence of sets raises some difficulties.

\subsection{Estimation of HSIC-ANOVA indices for sets}
\label{subsec:estim}
As given in Equation \ref{unbiased estimator},  $\operatorname{HSIC}\left(\bm U_A, \Gamma\right)$ can be estimated by

\begin{equation*}
\widehat{\operatorname{HSIC}}\left(\bm U_A, \Gamma \right)=\frac{2}{n(n-1)} \sum_{i <j}^{n}\left(K_{A}\left(\bm U_A^{(i)}, \bm U_A^{(j)}\right)-1\right) k_{set}\left(\Gamma^{(i)}, \Gamma^{(j)}\right),
\end{equation*}
where $(\bm U_A^{(i)},\Gamma^{(i)}), i=1,...,n$ is an iid sample of $(\bm U_A,\Gamma)$.

For set-valued models, it is common to have access only to the knowledge of whether a given point $x \in \Xdom$ is in a set output $\Gamma$ or not. In terms of estimation, this means that the Lebesgue measure of the symmetric difference and thus $k_{set}$ must be estimated. 

For two random sets $(\Gamma^{(i)},\Gamma^{(j)})$, $k_{set}(\Gamma^{(i)},\Gamma^{(j)})=e^{- \frac{\lambda(\Gamma^{(i)}\Delta \Gamma^{(j)})}{2\sigma^2}}$ can be written as
\begin{equation*}
k_{set}(\Gamma^{(i)},\Gamma^{(j)})=\exp\left(-\frac{\lambda(\Xdom)}{2\sigma^2}\Esp[ \mathbbold 1 _ {\Gamma^{(i)} \Delta\Gamma^{(j)} }(\bm X)| (\Gamma^{(i)},\Gamma^{(j)}) ]\right),
\end{equation*} 
where $\bm X\sim \Udom(\Xdom)$ using that $\lambda(\Gamma)=\lambda(\Xdom) \Esp_{\bm X\sim \Udom(\Xdom)}[\mathbbold 1_\Gamma (\bm X) | \Gamma ]$. Then, given an iid sample $(\bm X^{(1)},...,\bm X^{(m)})$ of $\bm X$, we can estimate $k_{set}(\Gamma^{(i)},\Gamma^{(j)}) $ by
\begin{equation*}
 \widehat{k_{set}}(\Gamma^{(i)},\Gamma^{(j)}) =  \exp\left(-\frac{\lambda(\Xdom)}{2\sigma^2}\frac{1}{m}\sum_{k=1}^{m} \mathbbold 1_{\Gamma^{(i)} \Delta\Gamma^{(j)}} (\bm X^{(k)})\right).
\end{equation*}
We now need to plug this estimator into our previous one, resulting in a Nested Monte Carlo (NMC) estimator. To be compatible with the framework of NMC estimator, it is necessary for samples of $\bm X$ to be drawn independently for every pair of indices $(i,j)$ within the outer loop. Let  $(\bm X_{i,j}^{(k)})$ be an iid sample of $\bm X \sim \Udom (\Xdom)$, with $k \in \{1,...,m\}$ and $(i, j) \in \{1,...,n\}^2$ s.t. $i<j$. 
Including the previous estimation of the kernel $k_{set}$, the NMC estimator of $\operatorname{HSIC}(\bm U_A,\Gamma)$ is given by
\begin{equation}
\widehat{\operatorname{HSIC}}^{nest}\left(\bm U_A, \Gamma\right)=\frac{2}{n(n-1)} \sum_{i < j}^{n}\left(K_{A}\left(\bm U_A^{(i)}, \bm U_A^{(j)}\right)-1\right) e^{-\frac{\lambda(\Xdom)}{2\sigma^2}\frac{1}{m}\sum_{k=1}^{m} \mathbbold 1_{\Gamma^{(i)} \Delta \Gamma^{(j)}} ( \bm X_{i,j}^{(k)})}.
\label{real_nested_estimator}
\end{equation}
Note that this estimator is biased like any NMC estimator (see \textcite{rainforth2016pitfalls}).
For each pair $(\Gamma^{(i)},\Gamma^{(j)})$, the previous estimator requires checking whether $\bm X_{i,j}^{(k)} \in \Gamma^{(i)} \Delta\Gamma^{(j)}$ for each $k \in \{1,...,m\}$. Each of these requires checking whether $\bm X_{i,j}^{(k)} \in \Gamma^{(i)}$ and $\bm X_{i,j}^{(k)} \in \Gamma^{(j)}$, corresponding in the example of excursion sets, to compute $g(\bm X_{i,j}^{(k)}, \bm U_A^{(i)})$. This means $n(n-1)m$ evaluations to estimate the index, which is not affordable. To solve this problem, we propose to reuse the same $\bm X^{(k)}$ for each $(i,j)$. By doing so, we only need to test whether $\bm X^{(k)} \in \Gamma^{(i)}$ for each $k$ and $i$, reducing the number of evaluations to $nm$. The estimator is then given by
\begin{equation*}
\widehat{\widehat{\operatorname{HSIC}}}\left(\bm U_A, \Gamma\right)=\frac{2}{n(n-1)} \sum_{i < j}^{n}\left(K_{A}\left(\bm U_A^{(i)}, \bm U_A^{(j)}\right)-1\right) e^{-\frac{\lambda(\Xdom)}{2\sigma^2}\frac{1}{m}\sum_{k=1}^{m} \mathbbold 1_{\Gamma^{(i)} \Delta \Gamma^{(j)}} (\bm X^{(k)})}.
\end{equation*}
By simulating a single $m$-sample of $\bm X$, $\widehat{\widehat{\operatorname{HSIC}}}\left(\bm U_A, \Gamma\right)$ is no longer a classic NMC estimator. Nevertheless, we show that its quadratic risk goes to $0$ and we give an upper bound.
\begin{restatable}{prop}{nestedestimator}
\label{prop nested}
With the previous notations, we have
\begin{align*}
\Esp \left(\widehat{\widehat{\operatorname{HSIC}}}\left(\bm U_A, \Gamma\right) -  \operatorname{HSIC}(\bm U_A, \Gamma) \right)^2 
\leq 2 \left( \frac{2\sigma_1^2}{n(n-1)}+ \frac{4(n-2)\sigma_2^2}{n(n-1)}+  \frac{L^2 \sigma_3^2}{m} \right),
\end{align*}
where
\begin{itemize}
\item $\sigma_1^2=Var \left(\left(K_{A}\left(\bm U_A, {\bm U_A}'\right)-1\right) k_{set}(\Gamma,\Gamma')\right)$,
\item $\begin{aligned}[t] \! \sigma_2^2&
&= Var\left( \Esp \left[ \left(K_{A}\left(\bm U_A, {\bm U_A}'\right)-1\right) k_{set}(\Gamma,\Gamma') | (\Gamma,\bm U_A) \right] \right),\end{aligned}$
\item $\sigma_3^2=\Esp \left[ \left(K_{A}\left(\bm U_A, {\bm U_A}'\right)-1\right)^2\operatorname{Var}\left(  \mathbbold 1 _{\Gamma \Delta \Gamma'}(\bm X) | (\bm U_A,{\bm U_A}',\Gamma,\Gamma')\right)\right],$
\item $L=\frac{\lambda(\Xdom)}{2\sigma^2}$,
\end{itemize}
and where $({\bm U_A}', \Gamma')$ is an independent copy of $({\bm U_A},\Gamma)$.
\end{restatable}
The quadratic risk has a rate of $\mathcal O (\frac{1}{n}+\frac{1}{m})$, which tends to say that we should use $n=m$. In the case of the classic NMC estimator (see \textcite{rainforth_nesting_2018}), here without reusing the same samples of $\bm X$, the rate is $\mathcal O (\frac{1}{n}+\frac{1}{m^2})$ as shown in the Appendix \ref{annex:indep}. However, even if the convergence rate is better, the number of evaluations required is $(n-1)$ times larger, which makes our choice to use the same sample of $\bm X$ more efficient. 
Since the previous result is only an upper bound, we can also hope to get closer to a $\mathcal O (\frac{1}{n}+\frac{1}{m^2})$ convergence rate in the application tests. That is, we can expect to have achieved convergence of our estimator without having to take high $n$ and $m$. This will be highlighted in the numerical results presented in the next section. By plugging the estimator $\widehat{\widehat{\operatorname{HSIC}}}$, we denote $\hat{\hat{S}}^{\operatorname{H}_{set}}_{i}$ and $\hat{\hat{S}}^{\operatorname{H}_{set}}_{T_i}$ the estimators of the first-order and total-order HSIC ANOVA indices on sets.


\section{Numerical Experiments}
\label{numerical}

In this part, our goal is to quantify the influence of the inputs $\bm U$ on different excursion sets $\Gamma_g$ (defined in \autoref{excursion_random}). To do this, we will consider three test cases. We first consider an analytically known function $g$ defined on $\Xdom \subset \Real^2$ and $\Udom \subset \Real^2$, borrowed from \textcite{Reda_optim}. In this first toy case, we will also study numerically the quadratic risk bound given in Proposition \ref{prop nested}. The second test case is related to an optimization problem with a stationary harmonic oscillator, from \textcite{cousin_two-step_2022}, on which we want to quantify the impact of some uncertain inputs on the feasible sets. In the last case, a sensitivity analysis is performed in the context of a bi-objective robust optimization of a permanent magnet-assisted synchronous reluctance machine for electrical vehicle application.  For each example, our sensitivity analysis is performed in two steps:
\begin{itemize}
    \item Screening: We compute the p-values associated with the test $\operatorname{HSIC}(U_i,\Gamma_g) =0 $ versus $\operatorname{HSIC}(U_i,\Gamma_g) >0 $. We use permutation based estimation. If the p-value is greater than a risk $\alpha$, the input is negligible, and influential otherwise. We use $\alpha = 0.05$ but this value can be changed depending on the application.
    \item Ranking: We compute the first-order indices $\hat{\hat{S}}^{\operatorname{H}_{set}}_{i}$ and the total-order indices $\hat{\hat{S}}^{\operatorname{H}_{set}}_{T_i}$ of all inputs. The first-order indices $\hat{\hat{S}}^{\operatorname{H}_{set}}_{i}$ are used to rank the inputs. With the total-order indices $\hat{\hat{S}}^{\operatorname{H}_{set}}_{T_i}$ we can quantify the HSIC interaction effects. 
\end{itemize}
The numerical implementation of the indices is done in the R language, using the \texttt{sensitivity} package and in particular the \texttt{sensiHSIC} and \texttt{testHSIC} functions, which allows to compute HSIC-ANOVA indices and p-values. 
\subsection{Excursion sets of a toy function}
\label{toy_fct}
In this part, we will estimate the previous HSIC-ANOVA indices on the excursion set $\Gamma_g$ defined by the following function $g$ from \textcite{Reda_optim},
\begin{equation*}
\forall \bm x,\bm u \in [-5,5]^2 \times [-5,5]^2 ~~ g(\bm x,\bm u)=-x_1^2 + 5x_2 - u_1 +u_2^2-1.
\end{equation*}
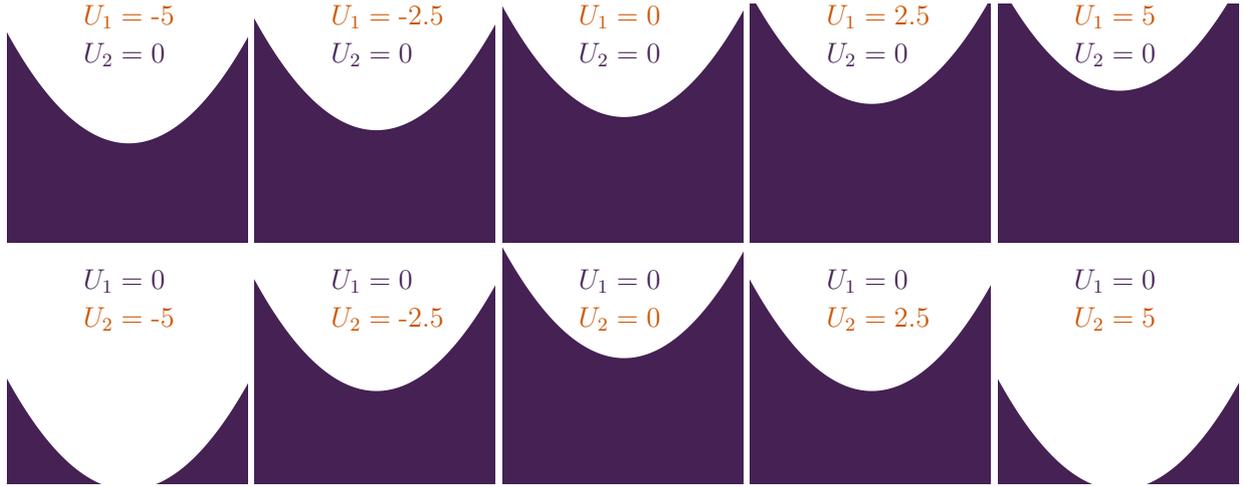
\begin{figure}
\centering
\resizebox{\textwidth}{!}{
\begin{tikzpicture}[x=1pt,y=1pt]
\definecolor{fillColor}{RGB}{255,255,255}
\definecolor{MyOrange}{rgb}{0.8, 0.33, 0}
\path[use as bounding box,fill=fillColor,fill opacity=0.00] (0,0) rectangle (252.94,252.94);
\begin{scope}
\path[clip] (  0.00,  0.00) rectangle (252.94,252.94);
\definecolor{fillColor}{RGB}{69,33,84}

\path[fill=fillColor] (-11.74,245.10) --
	( -8.92,239.52) --
	( -6.10,234.05) --
	( -3.28,228.70) --
	( -0.46,223.46) --
	(  2.36,218.33) --
	(  5.18,213.31) --
	(  8.00,208.42) --
	( 10.82,203.63) --
	( 13.64,198.96) --
	( 16.46,194.40) --
	( 19.28,189.96) --
	( 22.10,185.63) --
	( 24.92,181.41) --
	( 27.74,177.31) --
	( 30.56,173.32) --
	( 33.38,169.45) --
	( 36.20,165.69) --
	( 39.02,162.04) --
	( 41.84,158.51) --
	( 44.66,155.09) --
	( 47.48,151.79) --
	( 50.30,148.60) --
	( 53.12,145.52) --
	( 55.94,142.56) --
	( 58.76,139.71) --
	( 61.58,136.97) --
	( 64.40,134.35) --
	( 67.22,131.85) --
	( 70.04,129.45) --
	( 72.86,127.18) --
	( 75.68,125.01) --
	( 78.50,122.96) --
	( 81.32,121.02) --
	( 84.14,119.20) --
	( 86.96,117.49) --
	( 89.78,115.90) --
	( 92.60,114.41) --
	( 95.42,113.05) --
	( 98.24,111.79) --
	(101.06,110.65) --
	(103.88,109.63) --
	(106.70,108.72) --
	(109.52,107.92) --
	(112.34,107.24) --
	(115.16,106.67) --
	(117.98,106.21) --
	(120.80,105.87) --
	(123.62,105.64) --
	(126.44,105.53) --
	(129.26,105.53) --
	(132.08,105.64) --
	(134.90,105.87) --
	(137.72,106.21) --
	(140.54,106.67) --
	(143.36,107.24) --
	(146.18,107.92) --
	(149.00,108.72) --
	(151.82,109.63) --
	(154.64,110.65) --
	(157.46,111.79) --
	(160.28,113.05) --
	(163.10,114.41) --
	(165.92,115.90) --
	(168.74,117.49) --
	(171.56,119.20) --
	(174.38,121.02) --
	(177.20,122.96) --
	(180.02,125.01) --
	(182.84,127.18) --
	(185.66,129.45) --
	(188.48,131.85) --
	(191.30,134.35) --
	(194.12,136.97) --
	(196.94,139.71) --
	(199.76,142.56) --
	(202.58,145.52) --
	(205.40,148.60) --
	(208.22,151.79) --
	(211.04,155.09) --
	(213.86,158.51) --
	(216.68,162.04) --
	(219.50,165.69) --
	(222.32,169.45) --
	(225.14,173.32) --
	(227.96,177.31) --
	(230.78,181.41) --
	(233.60,185.63) --
	(236.42,189.96) --
	(239.24,194.40) --
	(242.06,198.96) --
	(244.88,203.63) --
	(247.70,208.42) --
	(250.52,213.31) --
	(253.34,218.33) --
	(256.16,223.46) --
	(258.98,228.70) --
	(261.80,234.05) --
	(264.62,239.52) --
	(267.44,245.10) --
	(267.44,-11.74) --
	(264.62,-11.74) --
	(261.80,-11.74) --
	(258.98,-11.74) --
	(256.16,-11.74) --
	(253.34,-11.74) --
	(250.52,-11.74) --
	(247.70,-11.74) --
	(244.88,-11.74) --
	(242.06,-11.74) --
	(239.24,-11.74) --
	(236.42,-11.74) --
	(233.60,-11.74) --
	(230.78,-11.74) --
	(227.96,-11.74) --
	(225.14,-11.74) --
	(222.32,-11.74) --
	(219.50,-11.74) --
	(216.68,-11.74) --
	(213.86,-11.74) --
	(211.04,-11.74) --
	(208.22,-11.74) --
	(205.40,-11.74) --
	(202.58,-11.74) --
	(199.76,-11.74) --
	(196.94,-11.74) --
	(194.12,-11.74) --
	(191.30,-11.74) --
	(188.48,-11.74) --
	(185.66,-11.74) --
	(182.84,-11.74) --
	(180.02,-11.74) --
	(177.20,-11.74) --
	(174.38,-11.74) --
	(171.56,-11.74) --
	(168.74,-11.74) --
	(165.92,-11.74) --
	(163.10,-11.74) --
	(160.28,-11.74) --
	(157.46,-11.74) --
	(154.64,-11.74) --
	(151.82,-11.74) --
	(149.00,-11.74) --
	(146.18,-11.74) --
	(143.36,-11.74) --
	(140.54,-11.74) --
	(137.72,-11.74) --
	(134.90,-11.74) --
	(132.08,-11.74) --
	(129.26,-11.74) --
	(126.44,-11.74) --
	(123.62,-11.74) --
	(120.80,-11.74) --
	(117.98,-11.74) --
	(115.16,-11.74) --
	(112.34,-11.74) --
	(109.52,-11.74) --
	(106.70,-11.74) --
	(103.88,-11.74) --
	(101.06,-11.74) --
	( 98.24,-11.74) --
	( 95.42,-11.74) --
	( 92.60,-11.74) --
	( 89.78,-11.74) --
	( 86.96,-11.74) --
	( 84.14,-11.74) --
	( 81.32,-11.74) --
	( 78.50,-11.74) --
	( 75.68,-11.74) --
	( 72.86,-11.74) --
	( 70.04,-11.74) --
	( 67.22,-11.74) --
	( 64.40,-11.74) --
	( 61.58,-11.74) --
	( 58.76,-11.74) --
	( 55.94,-11.74) --
	( 53.12,-11.74) --
	( 50.30,-11.74) --
	( 47.48,-11.74) --
	( 44.66,-11.74) --
	( 41.84,-11.74) --
	( 39.02,-11.74) --
	( 36.20,-11.74) --
	( 33.38,-11.74) --
	( 30.56,-11.74) --
	( 27.74,-11.74) --
	( 24.92,-11.74) --
	( 22.10,-11.74) --
	( 19.28,-11.74) --
	( 16.46,-11.74) --
	( 13.64,-11.74) --
	( 10.82,-11.74) --
	(  8.00,-11.74) --
	(  5.18,-11.74) --
	(  2.36,-11.74) --
	( -0.46,-11.74) --
	( -3.28,-11.74) --
	( -6.10,-11.74) --
	( -8.92,-11.74) --
	(-11.74,-11.74) --
	cycle;

\path[] (-11.74,245.10) --
	( -8.92,239.52) --
	( -6.10,234.05) --
	( -3.28,228.70) --
	( -0.46,223.46) --
	(  2.36,218.33) --
	(  5.18,213.31) --
	(  8.00,208.42) --
	( 10.82,203.63) --
	( 13.64,198.96) --
	( 16.46,194.40) --
	( 19.28,189.96) --
	( 22.10,185.63) --
	( 24.92,181.41) --
	( 27.74,177.31) --
	( 30.56,173.32) --
	( 33.38,169.45) --
	( 36.20,165.69) --
	( 39.02,162.04) --
	( 41.84,158.51) --
	( 44.66,155.09) --
	( 47.48,151.79) --
	( 50.30,148.60) --
	( 53.12,145.52) --
	( 55.94,142.56) --
	( 58.76,139.71) --
	( 61.58,136.97) --
	( 64.40,134.35) --
	( 67.22,131.85) --
	( 70.04,129.45) --
	( 72.86,127.18) --
	( 75.68,125.01) --
	( 78.50,122.96) --
	( 81.32,121.02) --
	( 84.14,119.20) --
	( 86.96,117.49) --
	( 89.78,115.90) --
	( 92.60,114.41) --
	( 95.42,113.05) --
	( 98.24,111.79) --
	(101.06,110.65) --
	(103.88,109.63) --
	(106.70,108.72) --
	(109.52,107.92) --
	(112.34,107.24) --
	(115.16,106.67) --
	(117.98,106.21) --
	(120.80,105.87) --
	(123.62,105.64) --
	(126.44,105.53) --
	(129.26,105.53) --
	(132.08,105.64) --
	(134.90,105.87) --
	(137.72,106.21) --
	(140.54,106.67) --
	(143.36,107.24) --
	(146.18,107.92) --
	(149.00,108.72) --
	(151.82,109.63) --
	(154.64,110.65) --
	(157.46,111.79) --
	(160.28,113.05) --
	(163.10,114.41) --
	(165.92,115.90) --
	(168.74,117.49) --
	(171.56,119.20) --
	(174.38,121.02) --
	(177.20,122.96) --
	(180.02,125.01) --
	(182.84,127.18) --
	(185.66,129.45) --
	(188.48,131.85) --
	(191.30,134.35) --
	(194.12,136.97) --
	(196.94,139.71) --
	(199.76,142.56) --
	(202.58,145.52) --
	(205.40,148.60) --
	(208.22,151.79) --
	(211.04,155.09) --
	(213.86,158.51) --
	(216.68,162.04) --
	(219.50,165.69) --
	(222.32,169.45) --
	(225.14,173.32) --
	(227.96,177.31) --
	(230.78,181.41) --
	(233.60,185.63) --
	(236.42,189.96) --
	(239.24,194.40) --
	(242.06,198.96) --
	(244.88,203.63) --
	(247.70,208.42) --
	(250.52,213.31) --
	(253.34,218.33) --
	(256.16,223.46) --
	(258.98,228.70) --
	(261.80,234.05) --
	(264.62,239.52) --
	(267.44,245.10);

\path[] (267.44,-11.74) --
	(264.62,-11.74) --
	(261.80,-11.74) --
	(258.98,-11.74) --
	(256.16,-11.74) --
	(253.34,-11.74) --
	(250.52,-11.74) --
	(247.70,-11.74) --
	(244.88,-11.74) --
	(242.06,-11.74) --
	(239.24,-11.74) --
	(236.42,-11.74) --
	(233.60,-11.74) --
	(230.78,-11.74) --
	(227.96,-11.74) --
	(225.14,-11.74) --
	(222.32,-11.74) --
	(219.50,-11.74) --
	(216.68,-11.74) --
	(213.86,-11.74) --
	(211.04,-11.74) --
	(208.22,-11.74) --
	(205.40,-11.74) --
	(202.58,-11.74) --
	(199.76,-11.74) --
	(196.94,-11.74) --
	(194.12,-11.74) --
	(191.30,-11.74) --
	(188.48,-11.74) --
	(185.66,-11.74) --
	(182.84,-11.74) --
	(180.02,-11.74) --
	(177.20,-11.74) --
	(174.38,-11.74) --
	(171.56,-11.74) --
	(168.74,-11.74) --
	(165.92,-11.74) --
	(163.10,-11.74) --
	(160.28,-11.74) --
	(157.46,-11.74) --
	(154.64,-11.74) --
	(151.82,-11.74) --
	(149.00,-11.74) --
	(146.18,-11.74) --
	(143.36,-11.74) --
	(140.54,-11.74) --
	(137.72,-11.74) --
	(134.90,-11.74) --
	(132.08,-11.74) --
	(129.26,-11.74) --
	(126.44,-11.74) --
	(123.62,-11.74) --
	(120.80,-11.74) --
	(117.98,-11.74) --
	(115.16,-11.74) --
	(112.34,-11.74) --
	(109.52,-11.74) --
	(106.70,-11.74) --
	(103.88,-11.74) --
	(101.06,-11.74) --
	( 98.24,-11.74) --
	( 95.42,-11.74) --
	( 92.60,-11.74) --
	( 89.78,-11.74) --
	( 86.96,-11.74) --
	( 84.14,-11.74) --
	( 81.32,-11.74) --
	( 78.50,-11.74) --
	( 75.68,-11.74) --
	( 72.86,-11.74) --
	( 70.04,-11.74) --
	( 67.22,-11.74) --
	( 64.40,-11.74) --
	( 61.58,-11.74) --
	( 58.76,-11.74) --
	( 55.94,-11.74) --
	( 53.12,-11.74) --
	( 50.30,-11.74) --
	( 47.48,-11.74) --
	( 44.66,-11.74) --
	( 41.84,-11.74) --
	( 39.02,-11.74) --
	( 36.20,-11.74) --
	( 33.38,-11.74) --
	( 30.56,-11.74) --
	( 27.74,-11.74) --
	( 24.92,-11.74) --
	( 22.10,-11.74) --
	( 19.28,-11.74) --
	( 16.46,-11.74) --
	( 13.64,-11.74) --
	( 10.82,-11.74) --
	(  8.00,-11.74) --
	(  5.18,-11.74) --
	(  2.36,-11.74) --
	( -0.46,-11.74) --
	( -3.28,-11.74) --
	( -6.10,-11.74) --
	( -8.92,-11.74) --
	(-11.74,-11.74);
\node[text=MyOrange,inner sep=0pt,anchor=west, outer sep=0pt, scale=  2.5] at ( 80.00,240.00) {$U_1=\text{-}5$};
 \node[text=fillColor,inner sep=0pt,anchor=west, outer sep=0pt, scale=  2.5] at ( 80.00,200.00) {$U_2=0$};
\end{scope}
\end{tikzpicture}
\begin{tikzpicture}[x=1pt,y=1pt]
\definecolor{fillColor}{RGB}{255,255,255}
\definecolor{MyOrange}{rgb}{0.8, 0.33, 0}
\path[use as bounding box,fill=fillColor,fill opacity=0.00] (0,0) rectangle (252.94,252.94);
\begin{scope}
\path[clip] (  0.00,  0.00) rectangle (252.94,252.94);
\definecolor{fillColor}{RGB}{69,33,84}

\path[fill=fillColor] (-11.74,259.06) --
	( -8.92,253.48) --
	( -6.10,248.01) --
	( -3.28,242.66) --
	( -0.46,237.41) --
	(  2.36,232.29) --
	(  5.18,227.27) --
	(  8.00,222.37) --
	( 10.82,217.59) --
	( 13.64,212.92) --
	( 16.46,208.36) --
	( 19.28,203.92) --
	( 22.10,199.59) --
	( 24.92,195.37) --
	( 27.74,191.27) --
	( 30.56,187.28) --
	( 33.38,183.41) --
	( 36.20,179.65) --
	( 39.02,176.00) --
	( 41.84,172.47) --
	( 44.66,169.05) --
	( 47.48,165.75) --
	( 50.30,162.56) --
	( 53.12,159.48) --
	( 55.94,156.52) --
	( 58.76,153.67) --
	( 61.58,150.93) --
	( 64.40,148.31) --
	( 67.22,145.81) --
	( 70.04,143.41) --
	( 72.86,141.13) --
	( 75.68,138.97) --
	( 78.50,136.92) --
	( 81.32,134.98) --
	( 84.14,133.16) --
	( 86.96,131.45) --
	( 89.78,129.85) --
	( 92.60,128.37) --
	( 95.42,127.01) --
	( 98.24,125.75) --
	(101.06,124.61) --
	(103.88,123.59) --
	(106.70,122.68) --
	(109.52,121.88) --
	(112.34,121.20) --
	(115.16,120.63) --
	(117.98,120.17) --
	(120.80,119.83) --
	(123.62,119.60) --
	(126.44,119.49) --
	(129.26,119.49) --
	(132.08,119.60) --
	(134.90,119.83) --
	(137.72,120.17) --
	(140.54,120.63) --
	(143.36,121.20) --
	(146.18,121.88) --
	(149.00,122.68) --
	(151.82,123.59) --
	(154.64,124.61) --
	(157.46,125.75) --
	(160.28,127.01) --
	(163.10,128.37) --
	(165.92,129.85) --
	(168.74,131.45) --
	(171.56,133.16) --
	(174.38,134.98) --
	(177.20,136.92) --
	(180.02,138.97) --
	(182.84,141.13) --
	(185.66,143.41) --
	(188.48,145.81) --
	(191.30,148.31) --
	(194.12,150.93) --
	(196.94,153.67) --
	(199.76,156.52) --
	(202.58,159.48) --
	(205.40,162.56) --
	(208.22,165.75) --
	(211.04,169.05) --
	(213.86,172.47) --
	(216.68,176.00) --
	(219.50,179.65) --
	(222.32,183.41) --
	(225.14,187.28) --
	(227.96,191.27) --
	(230.78,195.37) --
	(233.60,199.59) --
	(236.42,203.92) --
	(239.24,208.36) --
	(242.06,212.92) --
	(244.88,217.59) --
	(247.70,222.37) --
	(250.52,227.27) --
	(253.34,232.29) --
	(256.16,237.41) --
	(258.98,242.66) --
	(261.80,248.01) --
	(264.62,253.48) --
	(267.44,259.06) --
	(267.44,-11.74) --
	(264.62,-11.74) --
	(261.80,-11.74) --
	(258.98,-11.74) --
	(256.16,-11.74) --
	(253.34,-11.74) --
	(250.52,-11.74) --
	(247.70,-11.74) --
	(244.88,-11.74) --
	(242.06,-11.74) --
	(239.24,-11.74) --
	(236.42,-11.74) --
	(233.60,-11.74) --
	(230.78,-11.74) --
	(227.96,-11.74) --
	(225.14,-11.74) --
	(222.32,-11.74) --
	(219.50,-11.74) --
	(216.68,-11.74) --
	(213.86,-11.74) --
	(211.04,-11.74) --
	(208.22,-11.74) --
	(205.40,-11.74) --
	(202.58,-11.74) --
	(199.76,-11.74) --
	(196.94,-11.74) --
	(194.12,-11.74) --
	(191.30,-11.74) --
	(188.48,-11.74) --
	(185.66,-11.74) --
	(182.84,-11.74) --
	(180.02,-11.74) --
	(177.20,-11.74) --
	(174.38,-11.74) --
	(171.56,-11.74) --
	(168.74,-11.74) --
	(165.92,-11.74) --
	(163.10,-11.74) --
	(160.28,-11.74) --
	(157.46,-11.74) --
	(154.64,-11.74) --
	(151.82,-11.74) --
	(149.00,-11.74) --
	(146.18,-11.74) --
	(143.36,-11.74) --
	(140.54,-11.74) --
	(137.72,-11.74) --
	(134.90,-11.74) --
	(132.08,-11.74) --
	(129.26,-11.74) --
	(126.44,-11.74) --
	(123.62,-11.74) --
	(120.80,-11.74) --
	(117.98,-11.74) --
	(115.16,-11.74) --
	(112.34,-11.74) --
	(109.52,-11.74) --
	(106.70,-11.74) --
	(103.88,-11.74) --
	(101.06,-11.74) --
	( 98.24,-11.74) --
	( 95.42,-11.74) --
	( 92.60,-11.74) --
	( 89.78,-11.74) --
	( 86.96,-11.74) --
	( 84.14,-11.74) --
	( 81.32,-11.74) --
	( 78.50,-11.74) --
	( 75.68,-11.74) --
	( 72.86,-11.74) --
	( 70.04,-11.74) --
	( 67.22,-11.74) --
	( 64.40,-11.74) --
	( 61.58,-11.74) --
	( 58.76,-11.74) --
	( 55.94,-11.74) --
	( 53.12,-11.74) --
	( 50.30,-11.74) --
	( 47.48,-11.74) --
	( 44.66,-11.74) --
	( 41.84,-11.74) --
	( 39.02,-11.74) --
	( 36.20,-11.74) --
	( 33.38,-11.74) --
	( 30.56,-11.74) --
	( 27.74,-11.74) --
	( 24.92,-11.74) --
	( 22.10,-11.74) --
	( 19.28,-11.74) --
	( 16.46,-11.74) --
	( 13.64,-11.74) --
	( 10.82,-11.74) --
	(  8.00,-11.74) --
	(  5.18,-11.74) --
	(  2.36,-11.74) --
	( -0.46,-11.74) --
	( -3.28,-11.74) --
	( -6.10,-11.74) --
	( -8.92,-11.74) --
	(-11.74,-11.74) --
	cycle;

\path[] (-11.74,259.06) --
	( -8.92,253.48) --
	( -6.10,248.01) --
	( -3.28,242.66) --
	( -0.46,237.41) --
	(  2.36,232.29) --
	(  5.18,227.27) --
	(  8.00,222.37) --
	( 10.82,217.59) --
	( 13.64,212.92) --
	( 16.46,208.36) --
	( 19.28,203.92) --
	( 22.10,199.59) --
	( 24.92,195.37) --
	( 27.74,191.27) --
	( 30.56,187.28) --
	( 33.38,183.41) --
	( 36.20,179.65) --
	( 39.02,176.00) --
	( 41.84,172.47) --
	( 44.66,169.05) --
	( 47.48,165.75) --
	( 50.30,162.56) --
	( 53.12,159.48) --
	( 55.94,156.52) --
	( 58.76,153.67) --
	( 61.58,150.93) --
	( 64.40,148.31) --
	( 67.22,145.81) --
	( 70.04,143.41) --
	( 72.86,141.13) --
	( 75.68,138.97) --
	( 78.50,136.92) --
	( 81.32,134.98) --
	( 84.14,133.16) --
	( 86.96,131.45) --
	( 89.78,129.85) --
	( 92.60,128.37) --
	( 95.42,127.01) --
	( 98.24,125.75) --
	(101.06,124.61) --
	(103.88,123.59) --
	(106.70,122.68) --
	(109.52,121.88) --
	(112.34,121.20) --
	(115.16,120.63) --
	(117.98,120.17) --
	(120.80,119.83) --
	(123.62,119.60) --
	(126.44,119.49) --
	(129.26,119.49) --
	(132.08,119.60) --
	(134.90,119.83) --
	(137.72,120.17) --
	(140.54,120.63) --
	(143.36,121.20) --
	(146.18,121.88) --
	(149.00,122.68) --
	(151.82,123.59) --
	(154.64,124.61) --
	(157.46,125.75) --
	(160.28,127.01) --
	(163.10,128.37) --
	(165.92,129.85) --
	(168.74,131.45) --
	(171.56,133.16) --
	(174.38,134.98) --
	(177.20,136.92) --
	(180.02,138.97) --
	(182.84,141.13) --
	(185.66,143.41) --
	(188.48,145.81) --
	(191.30,148.31) --
	(194.12,150.93) --
	(196.94,153.67) --
	(199.76,156.52) --
	(202.58,159.48) --
	(205.40,162.56) --
	(208.22,165.75) --
	(211.04,169.05) --
	(213.86,172.47) --
	(216.68,176.00) --
	(219.50,179.65) --
	(222.32,183.41) --
	(225.14,187.28) --
	(227.96,191.27) --
	(230.78,195.37) --
	(233.60,199.59) --
	(236.42,203.92) --
	(239.24,208.36) --
	(242.06,212.92) --
	(244.88,217.59) --
	(247.70,222.37) --
	(250.52,227.27) --
	(253.34,232.29) --
	(256.16,237.41) --
	(258.98,242.66) --
	(261.80,248.01) --
	(264.62,253.48) --
	(267.44,259.06);

\path[] (267.44,-11.74) --
	(264.62,-11.74) --
	(261.80,-11.74) --
	(258.98,-11.74) --
	(256.16,-11.74) --
	(253.34,-11.74) --
	(250.52,-11.74) --
	(247.70,-11.74) --
	(244.88,-11.74) --
	(242.06,-11.74) --
	(239.24,-11.74) --
	(236.42,-11.74) --
	(233.60,-11.74) --
	(230.78,-11.74) --
	(227.96,-11.74) --
	(225.14,-11.74) --
	(222.32,-11.74) --
	(219.50,-11.74) --
	(216.68,-11.74) --
	(213.86,-11.74) --
	(211.04,-11.74) --
	(208.22,-11.74) --
	(205.40,-11.74) --
	(202.58,-11.74) --
	(199.76,-11.74) --
	(196.94,-11.74) --
	(194.12,-11.74) --
	(191.30,-11.74) --
	(188.48,-11.74) --
	(185.66,-11.74) --
	(182.84,-11.74) --
	(180.02,-11.74) --
	(177.20,-11.74) --
	(174.38,-11.74) --
	(171.56,-11.74) --
	(168.74,-11.74) --
	(165.92,-11.74) --
	(163.10,-11.74) --
	(160.28,-11.74) --
	(157.46,-11.74) --
	(154.64,-11.74) --
	(151.82,-11.74) --
	(149.00,-11.74) --
	(146.18,-11.74) --
	(143.36,-11.74) --
	(140.54,-11.74) --
	(137.72,-11.74) --
	(134.90,-11.74) --
	(132.08,-11.74) --
	(129.26,-11.74) --
	(126.44,-11.74) --
	(123.62,-11.74) --
	(120.80,-11.74) --
	(117.98,-11.74) --
	(115.16,-11.74) --
	(112.34,-11.74) --
	(109.52,-11.74) --
	(106.70,-11.74) --
	(103.88,-11.74) --
	(101.06,-11.74) --
	( 98.24,-11.74) --
	( 95.42,-11.74) --
	( 92.60,-11.74) --
	( 89.78,-11.74) --
	( 86.96,-11.74) --
	( 84.14,-11.74) --
	( 81.32,-11.74) --
	( 78.50,-11.74) --
	( 75.68,-11.74) --
	( 72.86,-11.74) --
	( 70.04,-11.74) --
	( 67.22,-11.74) --
	( 64.40,-11.74) --
	( 61.58,-11.74) --
	( 58.76,-11.74) --
	( 55.94,-11.74) --
	( 53.12,-11.74) --
	( 50.30,-11.74) --
	( 47.48,-11.74) --
	( 44.66,-11.74) --
	( 41.84,-11.74) --
	( 39.02,-11.74) --
	( 36.20,-11.74) --
	( 33.38,-11.74) --
	( 30.56,-11.74) --
	( 27.74,-11.74) --
	( 24.92,-11.74) --
	( 22.10,-11.74) --
	( 19.28,-11.74) --
	( 16.46,-11.74) --
	( 13.64,-11.74) --
	( 10.82,-11.74) --
	(  8.00,-11.74) --
	(  5.18,-11.74) --
	(  2.36,-11.74) --
	( -0.46,-11.74) --
	( -3.28,-11.74) --
	( -6.10,-11.74) --
	( -8.92,-11.74) --
	(-11.74,-11.74);
  \node[text=MyOrange,inner sep=0pt,anchor=west, outer sep=0pt, scale=  2.5] at ( 80.00,240.00) {$U_1=\text{-}2.5$};
 \node[text=fillColor,inner sep=0pt,anchor=west, outer sep=0pt, scale=  2.5] at ( 80.00,200.00) {$U_2=0$};
\end{scope}
\end{tikzpicture}
\begin{tikzpicture}[x=1pt,y=1pt]
\definecolor{fillColor}{RGB}{255,255,255}
\definecolor{MyOrange}{rgb}{0.8, 0.33, 0}
\path[use as bounding box,fill=fillColor,fill opacity=0.00] (0,0) rectangle (252.94,252.94);
\begin{scope}
\path[clip] (  0.00,  0.00) rectangle (252.94,252.94);
\definecolor{fillColor}{RGB}{69,33,84}

\path[fill=fillColor] (-11.74,267.44) --
	( -8.92,267.44) --
	( -6.10,261.97) --
	( -3.28,256.61) --
	( -0.46,251.37) --
	(  2.36,246.25) --
	(  5.18,241.23) --
	(  8.00,236.33) --
	( 10.82,231.55) --
	( 13.64,226.88) --
	( 16.46,222.32) --
	( 19.28,217.88) --
	( 22.10,213.55) --
	( 24.92,209.33) --
	( 27.74,205.23) --
	( 30.56,201.24) --
	( 33.38,197.37) --
	( 36.20,193.61) --
	( 39.02,189.96) --
	( 41.84,186.43) --
	( 44.66,183.01) --
	( 47.48,179.71) --
	( 50.30,176.51) --
	( 53.12,173.44) --
	( 55.94,170.48) --
	( 58.76,167.63) --
	( 61.58,164.89) --
	( 64.40,162.27) --
	( 67.22,159.77) --
	( 70.04,157.37) --
	( 72.86,155.09) --
	( 75.68,152.93) --
	( 78.50,150.88) --
	( 81.32,148.94) --
	( 84.14,147.12) --
	( 86.96,145.41) --
	( 89.78,143.81) --
	( 92.60,142.33) --
	( 95.42,140.97) --
	( 98.24,139.71) --
	(101.06,138.57) --
	(103.88,137.55) --
	(106.70,136.64) --
	(109.52,135.84) --
	(112.34,135.15) --
	(115.16,134.58) --
	(117.98,134.13) --
	(120.80,133.79) --
	(123.62,133.56) --
	(126.44,133.45) --
	(129.26,133.45) --
	(132.08,133.56) --
	(134.90,133.79) --
	(137.72,134.13) --
	(140.54,134.58) --
	(143.36,135.15) --
	(146.18,135.84) --
	(149.00,136.64) --
	(151.82,137.55) --
	(154.64,138.57) --
	(157.46,139.71) --
	(160.28,140.97) --
	(163.10,142.33) --
	(165.92,143.81) --
	(168.74,145.41) --
	(171.56,147.12) --
	(174.38,148.94) --
	(177.20,150.88) --
	(180.02,152.93) --
	(182.84,155.09) --
	(185.66,157.37) --
	(188.48,159.77) --
	(191.30,162.27) --
	(194.12,164.89) --
	(196.94,167.63) --
	(199.76,170.48) --
	(202.58,173.44) --
	(205.40,176.51) --
	(208.22,179.71) --
	(211.04,183.01) --
	(213.86,186.43) --
	(216.68,189.96) --
	(219.50,193.61) --
	(222.32,197.37) --
	(225.14,201.24) --
	(227.96,205.23) --
	(230.78,209.33) --
	(233.60,213.55) --
	(236.42,217.88) --
	(239.24,222.32) --
	(242.06,226.88) --
	(244.88,231.55) --
	(247.70,236.33) --
	(250.52,241.23) --
	(253.34,246.25) --
	(256.16,251.37) --
	(258.98,256.61) --
	(261.80,261.97) --
	(264.62,267.44) --
	(267.44,267.44) --
	(267.44,-11.74) --
	(264.62,-11.74) --
	(261.80,-11.74) --
	(258.98,-11.74) --
	(256.16,-11.74) --
	(253.34,-11.74) --
	(250.52,-11.74) --
	(247.70,-11.74) --
	(244.88,-11.74) --
	(242.06,-11.74) --
	(239.24,-11.74) --
	(236.42,-11.74) --
	(233.60,-11.74) --
	(230.78,-11.74) --
	(227.96,-11.74) --
	(225.14,-11.74) --
	(222.32,-11.74) --
	(219.50,-11.74) --
	(216.68,-11.74) --
	(213.86,-11.74) --
	(211.04,-11.74) --
	(208.22,-11.74) --
	(205.40,-11.74) --
	(202.58,-11.74) --
	(199.76,-11.74) --
	(196.94,-11.74) --
	(194.12,-11.74) --
	(191.30,-11.74) --
	(188.48,-11.74) --
	(185.66,-11.74) --
	(182.84,-11.74) --
	(180.02,-11.74) --
	(177.20,-11.74) --
	(174.38,-11.74) --
	(171.56,-11.74) --
	(168.74,-11.74) --
	(165.92,-11.74) --
	(163.10,-11.74) --
	(160.28,-11.74) --
	(157.46,-11.74) --
	(154.64,-11.74) --
	(151.82,-11.74) --
	(149.00,-11.74) --
	(146.18,-11.74) --
	(143.36,-11.74) --
	(140.54,-11.74) --
	(137.72,-11.74) --
	(134.90,-11.74) --
	(132.08,-11.74) --
	(129.26,-11.74) --
	(126.44,-11.74) --
	(123.62,-11.74) --
	(120.80,-11.74) --
	(117.98,-11.74) --
	(115.16,-11.74) --
	(112.34,-11.74) --
	(109.52,-11.74) --
	(106.70,-11.74) --
	(103.88,-11.74) --
	(101.06,-11.74) --
	( 98.24,-11.74) --
	( 95.42,-11.74) --
	( 92.60,-11.74) --
	( 89.78,-11.74) --
	( 86.96,-11.74) --
	( 84.14,-11.74) --
	( 81.32,-11.74) --
	( 78.50,-11.74) --
	( 75.68,-11.74) --
	( 72.86,-11.74) --
	( 70.04,-11.74) --
	( 67.22,-11.74) --
	( 64.40,-11.74) --
	( 61.58,-11.74) --
	( 58.76,-11.74) --
	( 55.94,-11.74) --
	( 53.12,-11.74) --
	( 50.30,-11.74) --
	( 47.48,-11.74) --
	( 44.66,-11.74) --
	( 41.84,-11.74) --
	( 39.02,-11.74) --
	( 36.20,-11.74) --
	( 33.38,-11.74) --
	( 30.56,-11.74) --
	( 27.74,-11.74) --
	( 24.92,-11.74) --
	( 22.10,-11.74) --
	( 19.28,-11.74) --
	( 16.46,-11.74) --
	( 13.64,-11.74) --
	( 10.82,-11.74) --
	(  8.00,-11.74) --
	(  5.18,-11.74) --
	(  2.36,-11.74) --
	( -0.46,-11.74) --
	( -3.28,-11.74) --
	( -6.10,-11.74) --
	( -8.92,-11.74) --
	(-11.74,-11.74) --
	cycle;

\path[] (-11.74,267.44) --
	( -8.92,267.44) --
	( -6.10,261.97) --
	( -3.28,256.61) --
	( -0.46,251.37) --
	(  2.36,246.25) --
	(  5.18,241.23) --
	(  8.00,236.33) --
	( 10.82,231.55) --
	( 13.64,226.88) --
	( 16.46,222.32) --
	( 19.28,217.88) --
	( 22.10,213.55) --
	( 24.92,209.33) --
	( 27.74,205.23) --
	( 30.56,201.24) --
	( 33.38,197.37) --
	( 36.20,193.61) --
	( 39.02,189.96) --
	( 41.84,186.43) --
	( 44.66,183.01) --
	( 47.48,179.71) --
	( 50.30,176.51) --
	( 53.12,173.44) --
	( 55.94,170.48) --
	( 58.76,167.63) --
	( 61.58,164.89) --
	( 64.40,162.27) --
	( 67.22,159.77) --
	( 70.04,157.37) --
	( 72.86,155.09) --
	( 75.68,152.93) --
	( 78.50,150.88) --
	( 81.32,148.94) --
	( 84.14,147.12) --
	( 86.96,145.41) --
	( 89.78,143.81) --
	( 92.60,142.33) --
	( 95.42,140.97) --
	( 98.24,139.71) --
	(101.06,138.57) --
	(103.88,137.55) --
	(106.70,136.64) --
	(109.52,135.84) --
	(112.34,135.15) --
	(115.16,134.58) --
	(117.98,134.13) --
	(120.80,133.79) --
	(123.62,133.56) --
	(126.44,133.45) --
	(129.26,133.45) --
	(132.08,133.56) --
	(134.90,133.79) --
	(137.72,134.13) --
	(140.54,134.58) --
	(143.36,135.15) --
	(146.18,135.84) --
	(149.00,136.64) --
	(151.82,137.55) --
	(154.64,138.57) --
	(157.46,139.71) --
	(160.28,140.97) --
	(163.10,142.33) --
	(165.92,143.81) --
	(168.74,145.41) --
	(171.56,147.12) --
	(174.38,148.94) --
	(177.20,150.88) --
	(180.02,152.93) --
	(182.84,155.09) --
	(185.66,157.37) --
	(188.48,159.77) --
	(191.30,162.27) --
	(194.12,164.89) --
	(196.94,167.63) --
	(199.76,170.48) --
	(202.58,173.44) --
	(205.40,176.51) --
	(208.22,179.71) --
	(211.04,183.01) --
	(213.86,186.43) --
	(216.68,189.96) --
	(219.50,193.61) --
	(222.32,197.37) --
	(225.14,201.24) --
	(227.96,205.23) --
	(230.78,209.33) --
	(233.60,213.55) --
	(236.42,217.88) --
	(239.24,222.32) --
	(242.06,226.88) --
	(244.88,231.55) --
	(247.70,236.33) --
	(250.52,241.23) --
	(253.34,246.25) --
	(256.16,251.37) --
	(258.98,256.61) --
	(261.80,261.97) --
	(264.62,267.44) --
	(267.44,267.44);

\path[] (267.44,-11.74) --
	(264.62,-11.74) --
	(261.80,-11.74) --
	(258.98,-11.74) --
	(256.16,-11.74) --
	(253.34,-11.74) --
	(250.52,-11.74) --
	(247.70,-11.74) --
	(244.88,-11.74) --
	(242.06,-11.74) --
	(239.24,-11.74) --
	(236.42,-11.74) --
	(233.60,-11.74) --
	(230.78,-11.74) --
	(227.96,-11.74) --
	(225.14,-11.74) --
	(222.32,-11.74) --
	(219.50,-11.74) --
	(216.68,-11.74) --
	(213.86,-11.74) --
	(211.04,-11.74) --
	(208.22,-11.74) --
	(205.40,-11.74) --
	(202.58,-11.74) --
	(199.76,-11.74) --
	(196.94,-11.74) --
	(194.12,-11.74) --
	(191.30,-11.74) --
	(188.48,-11.74) --
	(185.66,-11.74) --
	(182.84,-11.74) --
	(180.02,-11.74) --
	(177.20,-11.74) --
	(174.38,-11.74) --
	(171.56,-11.74) --
	(168.74,-11.74) --
	(165.92,-11.74) --
	(163.10,-11.74) --
	(160.28,-11.74) --
	(157.46,-11.74) --
	(154.64,-11.74) --
	(151.82,-11.74) --
	(149.00,-11.74) --
	(146.18,-11.74) --
	(143.36,-11.74) --
	(140.54,-11.74) --
	(137.72,-11.74) --
	(134.90,-11.74) --
	(132.08,-11.74) --
	(129.26,-11.74) --
	(126.44,-11.74) --
	(123.62,-11.74) --
	(120.80,-11.74) --
	(117.98,-11.74) --
	(115.16,-11.74) --
	(112.34,-11.74) --
	(109.52,-11.74) --
	(106.70,-11.74) --
	(103.88,-11.74) --
	(101.06,-11.74) --
	( 98.24,-11.74) --
	( 95.42,-11.74) --
	( 92.60,-11.74) --
	( 89.78,-11.74) --
	( 86.96,-11.74) --
	( 84.14,-11.74) --
	( 81.32,-11.74) --
	( 78.50,-11.74) --
	( 75.68,-11.74) --
	( 72.86,-11.74) --
	( 70.04,-11.74) --
	( 67.22,-11.74) --
	( 64.40,-11.74) --
	( 61.58,-11.74) --
	( 58.76,-11.74) --
	( 55.94,-11.74) --
	( 53.12,-11.74) --
	( 50.30,-11.74) --
	( 47.48,-11.74) --
	( 44.66,-11.74) --
	( 41.84,-11.74) --
	( 39.02,-11.74) --
	( 36.20,-11.74) --
	( 33.38,-11.74) --
	( 30.56,-11.74) --
	( 27.74,-11.74) --
	( 24.92,-11.74) --
	( 22.10,-11.74) --
	( 19.28,-11.74) --
	( 16.46,-11.74) --
	( 13.64,-11.74) --
	( 10.82,-11.74) --
	(  8.00,-11.74) --
	(  5.18,-11.74) --
	(  2.36,-11.74) --
	( -0.46,-11.74) --
	( -3.28,-11.74) --
	( -6.10,-11.74) --
	( -8.92,-11.74) --
	(-11.74,-11.74);
  \node[text=MyOrange,inner sep=0pt,anchor=west, outer sep=0pt, scale=  2.5] at ( 80.00,240.00) {$U_1=0$};
 \node[text=fillColor,inner sep=0pt,anchor=west, outer sep=0pt, scale=  2.5] at ( 80.00,200.00) {$U_2=0$};
\end{scope}
\end{tikzpicture}
\begin{tikzpicture}[x=1pt,y=1pt]
\definecolor{fillColor}{RGB}{255,255,255}
\definecolor{MyOrange}{rgb}{0.8, 0.33, 0}
\path[use as bounding box,fill=fillColor,fill opacity=0.00] (0,0) rectangle (252.94,252.94);
\begin{scope}
\path[clip] (  0.00,  0.00) rectangle (252.94,252.94);
\definecolor{fillColor}{RGB}{69,33,84}

\path[fill=fillColor] (-11.74,267.44) --
	( -8.92,267.44) --
	( -6.10,267.44) --
	( -3.28,267.44) --
	( -0.46,265.33) --
	(  2.36,260.21) --
	(  5.18,255.19) --
	(  8.00,250.29) --
	( 10.82,245.51) --
	( 13.64,240.84) --
	( 16.46,236.28) --
	( 19.28,231.83) --
	( 22.10,227.50) --
	( 24.92,223.29) --
	( 27.74,219.19) --
	( 30.56,215.20) --
	( 33.38,211.33) --
	( 36.20,207.56) --
	( 39.02,203.92) --
	( 41.84,200.39) --
	( 44.66,196.97) --
	( 47.48,193.66) --
	( 50.30,190.47) --
	( 53.12,187.40) --
	( 55.94,184.44) --
	( 58.76,181.59) --
	( 61.58,178.85) --
	( 64.40,176.23) --
	( 67.22,173.72) --
	( 70.04,171.33) --
	( 72.86,169.05) --
	( 75.68,166.89) --
	( 78.50,164.84) --
	( 81.32,162.90) --
	( 84.14,161.08) --
	( 86.96,159.37) --
	( 89.78,157.77) --
	( 92.60,156.29) --
	( 95.42,154.92) --
	( 98.24,153.67) --
	(101.06,152.53) --
	(103.88,151.51) --
	(106.70,150.59) --
	(109.52,149.80) --
	(112.34,149.11) --
	(115.16,148.54) --
	(117.98,148.09) --
	(120.80,147.75) --
	(123.62,147.52) --
	(126.44,147.40) --
	(129.26,147.40) --
	(132.08,147.52) --
	(134.90,147.75) --
	(137.72,148.09) --
	(140.54,148.54) --
	(143.36,149.11) --
	(146.18,149.80) --
	(149.00,150.59) --
	(151.82,151.51) --
	(154.64,152.53) --
	(157.46,153.67) --
	(160.28,154.92) --
	(163.10,156.29) --
	(165.92,157.77) --
	(168.74,159.37) --
	(171.56,161.08) --
	(174.38,162.90) --
	(177.20,164.84) --
	(180.02,166.89) --
	(182.84,169.05) --
	(185.66,171.33) --
	(188.48,173.72) --
	(191.30,176.23) --
	(194.12,178.85) --
	(196.94,181.59) --
	(199.76,184.44) --
	(202.58,187.40) --
	(205.40,190.47) --
	(208.22,193.66) --
	(211.04,196.97) --
	(213.86,200.39) --
	(216.68,203.92) --
	(219.50,207.56) --
	(222.32,211.33) --
	(225.14,215.20) --
	(227.96,219.19) --
	(230.78,223.29) --
	(233.60,227.50) --
	(236.42,231.83) --
	(239.24,236.28) --
	(242.06,240.84) --
	(244.88,245.51) --
	(247.70,250.29) --
	(250.52,255.19) --
	(253.34,260.21) --
	(256.16,265.33) --
	(258.98,267.44) --
	(261.80,267.44) --
	(264.62,267.44) --
	(267.44,267.44) --
	(267.44,-11.74) --
	(264.62,-11.74) --
	(261.80,-11.74) --
	(258.98,-11.74) --
	(256.16,-11.74) --
	(253.34,-11.74) --
	(250.52,-11.74) --
	(247.70,-11.74) --
	(244.88,-11.74) --
	(242.06,-11.74) --
	(239.24,-11.74) --
	(236.42,-11.74) --
	(233.60,-11.74) --
	(230.78,-11.74) --
	(227.96,-11.74) --
	(225.14,-11.74) --
	(222.32,-11.74) --
	(219.50,-11.74) --
	(216.68,-11.74) --
	(213.86,-11.74) --
	(211.04,-11.74) --
	(208.22,-11.74) --
	(205.40,-11.74) --
	(202.58,-11.74) --
	(199.76,-11.74) --
	(196.94,-11.74) --
	(194.12,-11.74) --
	(191.30,-11.74) --
	(188.48,-11.74) --
	(185.66,-11.74) --
	(182.84,-11.74) --
	(180.02,-11.74) --
	(177.20,-11.74) --
	(174.38,-11.74) --
	(171.56,-11.74) --
	(168.74,-11.74) --
	(165.92,-11.74) --
	(163.10,-11.74) --
	(160.28,-11.74) --
	(157.46,-11.74) --
	(154.64,-11.74) --
	(151.82,-11.74) --
	(149.00,-11.74) --
	(146.18,-11.74) --
	(143.36,-11.74) --
	(140.54,-11.74) --
	(137.72,-11.74) --
	(134.90,-11.74) --
	(132.08,-11.74) --
	(129.26,-11.74) --
	(126.44,-11.74) --
	(123.62,-11.74) --
	(120.80,-11.74) --
	(117.98,-11.74) --
	(115.16,-11.74) --
	(112.34,-11.74) --
	(109.52,-11.74) --
	(106.70,-11.74) --
	(103.88,-11.74) --
	(101.06,-11.74) --
	( 98.24,-11.74) --
	( 95.42,-11.74) --
	( 92.60,-11.74) --
	( 89.78,-11.74) --
	( 86.96,-11.74) --
	( 84.14,-11.74) --
	( 81.32,-11.74) --
	( 78.50,-11.74) --
	( 75.68,-11.74) --
	( 72.86,-11.74) --
	( 70.04,-11.74) --
	( 67.22,-11.74) --
	( 64.40,-11.74) --
	( 61.58,-11.74) --
	( 58.76,-11.74) --
	( 55.94,-11.74) --
	( 53.12,-11.74) --
	( 50.30,-11.74) --
	( 47.48,-11.74) --
	( 44.66,-11.74) --
	( 41.84,-11.74) --
	( 39.02,-11.74) --
	( 36.20,-11.74) --
	( 33.38,-11.74) --
	( 30.56,-11.74) --
	( 27.74,-11.74) --
	( 24.92,-11.74) --
	( 22.10,-11.74) --
	( 19.28,-11.74) --
	( 16.46,-11.74) --
	( 13.64,-11.74) --
	( 10.82,-11.74) --
	(  8.00,-11.74) --
	(  5.18,-11.74) --
	(  2.36,-11.74) --
	( -0.46,-11.74) --
	( -3.28,-11.74) --
	( -6.10,-11.74) --
	( -8.92,-11.74) --
	(-11.74,-11.74) --
	cycle;

\path[] (-11.74,267.44) --
	( -8.92,267.44) --
	( -6.10,267.44) --
	( -3.28,267.44) --
	( -0.46,265.33) --
	(  2.36,260.21) --
	(  5.18,255.19) --
	(  8.00,250.29) --
	( 10.82,245.51) --
	( 13.64,240.84) --
	( 16.46,236.28) --
	( 19.28,231.83) --
	( 22.10,227.50) --
	( 24.92,223.29) --
	( 27.74,219.19) --
	( 30.56,215.20) --
	( 33.38,211.33) --
	( 36.20,207.56) --
	( 39.02,203.92) --
	( 41.84,200.39) --
	( 44.66,196.97) --
	( 47.48,193.66) --
	( 50.30,190.47) --
	( 53.12,187.40) --
	( 55.94,184.44) --
	( 58.76,181.59) --
	( 61.58,178.85) --
	( 64.40,176.23) --
	( 67.22,173.72) --
	( 70.04,171.33) --
	( 72.86,169.05) --
	( 75.68,166.89) --
	( 78.50,164.84) --
	( 81.32,162.90) --
	( 84.14,161.08) --
	( 86.96,159.37) --
	( 89.78,157.77) --
	( 92.60,156.29) --
	( 95.42,154.92) --
	( 98.24,153.67) --
	(101.06,152.53) --
	(103.88,151.51) --
	(106.70,150.59) --
	(109.52,149.80) --
	(112.34,149.11) --
	(115.16,148.54) --
	(117.98,148.09) --
	(120.80,147.75) --
	(123.62,147.52) --
	(126.44,147.40) --
	(129.26,147.40) --
	(132.08,147.52) --
	(134.90,147.75) --
	(137.72,148.09) --
	(140.54,148.54) --
	(143.36,149.11) --
	(146.18,149.80) --
	(149.00,150.59) --
	(151.82,151.51) --
	(154.64,152.53) --
	(157.46,153.67) --
	(160.28,154.92) --
	(163.10,156.29) --
	(165.92,157.77) --
	(168.74,159.37) --
	(171.56,161.08) --
	(174.38,162.90) --
	(177.20,164.84) --
	(180.02,166.89) --
	(182.84,169.05) --
	(185.66,171.33) --
	(188.48,173.72) --
	(191.30,176.23) --
	(194.12,178.85) --
	(196.94,181.59) --
	(199.76,184.44) --
	(202.58,187.40) --
	(205.40,190.47) --
	(208.22,193.66) --
	(211.04,196.97) --
	(213.86,200.39) --
	(216.68,203.92) --
	(219.50,207.56) --
	(222.32,211.33) --
	(225.14,215.20) --
	(227.96,219.19) --
	(230.78,223.29) --
	(233.60,227.50) --
	(236.42,231.83) --
	(239.24,236.28) --
	(242.06,240.84) --
	(244.88,245.51) --
	(247.70,250.29) --
	(250.52,255.19) --
	(253.34,260.21) --
	(256.16,265.33) --
	(258.98,267.44) --
	(261.80,267.44) --
	(264.62,267.44) --
	(267.44,267.44);

\path[] (267.44,-11.74) --
	(264.62,-11.74) --
	(261.80,-11.74) --
	(258.98,-11.74) --
	(256.16,-11.74) --
	(253.34,-11.74) --
	(250.52,-11.74) --
	(247.70,-11.74) --
	(244.88,-11.74) --
	(242.06,-11.74) --
	(239.24,-11.74) --
	(236.42,-11.74) --
	(233.60,-11.74) --
	(230.78,-11.74) --
	(227.96,-11.74) --
	(225.14,-11.74) --
	(222.32,-11.74) --
	(219.50,-11.74) --
	(216.68,-11.74) --
	(213.86,-11.74) --
	(211.04,-11.74) --
	(208.22,-11.74) --
	(205.40,-11.74) --
	(202.58,-11.74) --
	(199.76,-11.74) --
	(196.94,-11.74) --
	(194.12,-11.74) --
	(191.30,-11.74) --
	(188.48,-11.74) --
	(185.66,-11.74) --
	(182.84,-11.74) --
	(180.02,-11.74) --
	(177.20,-11.74) --
	(174.38,-11.74) --
	(171.56,-11.74) --
	(168.74,-11.74) --
	(165.92,-11.74) --
	(163.10,-11.74) --
	(160.28,-11.74) --
	(157.46,-11.74) --
	(154.64,-11.74) --
	(151.82,-11.74) --
	(149.00,-11.74) --
	(146.18,-11.74) --
	(143.36,-11.74) --
	(140.54,-11.74) --
	(137.72,-11.74) --
	(134.90,-11.74) --
	(132.08,-11.74) --
	(129.26,-11.74) --
	(126.44,-11.74) --
	(123.62,-11.74) --
	(120.80,-11.74) --
	(117.98,-11.74) --
	(115.16,-11.74) --
	(112.34,-11.74) --
	(109.52,-11.74) --
	(106.70,-11.74) --
	(103.88,-11.74) --
	(101.06,-11.74) --
	( 98.24,-11.74) --
	( 95.42,-11.74) --
	( 92.60,-11.74) --
	( 89.78,-11.74) --
	( 86.96,-11.74) --
	( 84.14,-11.74) --
	( 81.32,-11.74) --
	( 78.50,-11.74) --
	( 75.68,-11.74) --
	( 72.86,-11.74) --
	( 70.04,-11.74) --
	( 67.22,-11.74) --
	( 64.40,-11.74) --
	( 61.58,-11.74) --
	( 58.76,-11.74) --
	( 55.94,-11.74) --
	( 53.12,-11.74) --
	( 50.30,-11.74) --
	( 47.48,-11.74) --
	( 44.66,-11.74) --
	( 41.84,-11.74) --
	( 39.02,-11.74) --
	( 36.20,-11.74) --
	( 33.38,-11.74) --
	( 30.56,-11.74) --
	( 27.74,-11.74) --
	( 24.92,-11.74) --
	( 22.10,-11.74) --
	( 19.28,-11.74) --
	( 16.46,-11.74) --
	( 13.64,-11.74) --
	( 10.82,-11.74) --
	(  8.00,-11.74) --
	(  5.18,-11.74) --
	(  2.36,-11.74) --
	( -0.46,-11.74) --
	( -3.28,-11.74) --
	( -6.10,-11.74) --
	( -8.92,-11.74) --
	(-11.74,-11.74);
  \node[text=MyOrange,inner sep=0pt,anchor=west, outer sep=0pt, scale=  2.5] at ( 80.00,240.00) {$U_1=2.5$};
 \node[text=fillColor,inner sep=0pt,anchor=west, outer sep=0pt, scale=  2.5] at ( 80.00,200.00) {$U_2=0$};
\end{scope}
\end{tikzpicture}
\begin{tikzpicture}[x=1pt,y=1pt]
\definecolor{fillColor}{RGB}{255,255,255}
\definecolor{MyOrange}{rgb}{0.8, 0.33, 0}
\path[use as bounding box,fill=fillColor,fill opacity=0.00] (0,0) rectangle (252.94,252.94);
\begin{scope}
\path[clip] (  0.00,  0.00) rectangle (252.94,252.94);
\definecolor{fillColor}{RGB}{69,33,84}

\path[fill=fillColor] (-11.74,267.44) --
	( -8.92,267.44) --
	( -6.10,267.44) --
	( -3.28,267.44) --
	( -0.46,267.44) --
	(  2.36,267.44) --
	(  5.18,267.44) --
	(  8.00,264.25) --
	( 10.82,259.47) --
	( 13.64,254.79) --
	( 16.46,250.24) --
	( 19.28,245.79) --
	( 22.10,241.46) --
	( 24.92,237.25) --
	( 27.74,233.15) --
	( 30.56,229.16) --
	( 33.38,225.28) --
	( 36.20,221.52) --
	( 39.02,217.88) --
	( 41.84,214.35) --
	( 44.66,210.93) --
	( 47.48,207.62) --
	( 50.30,204.43) --
	( 53.12,201.36) --
	( 55.94,198.39) --
	( 58.76,195.55) --
	( 61.58,192.81) --
	( 64.40,190.19) --
	( 67.22,187.68) --
	( 70.04,185.29) --
	( 72.86,183.01) --
	( 75.68,180.85) --
	( 78.50,178.80) --
	( 81.32,176.86) --
	( 84.14,175.04) --
	( 86.96,173.33) --
	( 89.78,171.73) --
	( 92.60,170.25) --
	( 95.42,168.88) --
	( 98.24,167.63) --
	(101.06,166.49) --
	(103.88,165.47) --
	(106.70,164.55) --
	(109.52,163.76) --
	(112.34,163.07) --
	(115.16,162.50) --
	(117.98,162.05) --
	(120.80,161.71) --
	(123.62,161.48) --
	(126.44,161.36) --
	(129.26,161.36) --
	(132.08,161.48) --
	(134.90,161.71) --
	(137.72,162.05) --
	(140.54,162.50) --
	(143.36,163.07) --
	(146.18,163.76) --
	(149.00,164.55) --
	(151.82,165.47) --
	(154.64,166.49) --
	(157.46,167.63) --
	(160.28,168.88) --
	(163.10,170.25) --
	(165.92,171.73) --
	(168.74,173.33) --
	(171.56,175.04) --
	(174.38,176.86) --
	(177.20,178.80) --
	(180.02,180.85) --
	(182.84,183.01) --
	(185.66,185.29) --
	(188.48,187.68) --
	(191.30,190.19) --
	(194.12,192.81) --
	(196.94,195.55) --
	(199.76,198.39) --
	(202.58,201.36) --
	(205.40,204.43) --
	(208.22,207.62) --
	(211.04,210.93) --
	(213.86,214.35) --
	(216.68,217.88) --
	(219.50,221.52) --
	(222.32,225.28) --
	(225.14,229.16) --
	(227.96,233.15) --
	(230.78,237.25) --
	(233.60,241.46) --
	(236.42,245.79) --
	(239.24,250.24) --
	(242.06,254.79) --
	(244.88,259.47) --
	(247.70,264.25) --
	(250.52,267.44) --
	(253.34,267.44) --
	(256.16,267.44) --
	(258.98,267.44) --
	(261.80,267.44) --
	(264.62,267.44) --
	(267.44,267.44) --
	(267.44,-11.74) --
	(264.62,-11.74) --
	(261.80,-11.74) --
	(258.98,-11.74) --
	(256.16,-11.74) --
	(253.34,-11.74) --
	(250.52,-11.74) --
	(247.70,-11.74) --
	(244.88,-11.74) --
	(242.06,-11.74) --
	(239.24,-11.74) --
	(236.42,-11.74) --
	(233.60,-11.74) --
	(230.78,-11.74) --
	(227.96,-11.74) --
	(225.14,-11.74) --
	(222.32,-11.74) --
	(219.50,-11.74) --
	(216.68,-11.74) --
	(213.86,-11.74) --
	(211.04,-11.74) --
	(208.22,-11.74) --
	(205.40,-11.74) --
	(202.58,-11.74) --
	(199.76,-11.74) --
	(196.94,-11.74) --
	(194.12,-11.74) --
	(191.30,-11.74) --
	(188.48,-11.74) --
	(185.66,-11.74) --
	(182.84,-11.74) --
	(180.02,-11.74) --
	(177.20,-11.74) --
	(174.38,-11.74) --
	(171.56,-11.74) --
	(168.74,-11.74) --
	(165.92,-11.74) --
	(163.10,-11.74) --
	(160.28,-11.74) --
	(157.46,-11.74) --
	(154.64,-11.74) --
	(151.82,-11.74) --
	(149.00,-11.74) --
	(146.18,-11.74) --
	(143.36,-11.74) --
	(140.54,-11.74) --
	(137.72,-11.74) --
	(134.90,-11.74) --
	(132.08,-11.74) --
	(129.26,-11.74) --
	(126.44,-11.74) --
	(123.62,-11.74) --
	(120.80,-11.74) --
	(117.98,-11.74) --
	(115.16,-11.74) --
	(112.34,-11.74) --
	(109.52,-11.74) --
	(106.70,-11.74) --
	(103.88,-11.74) --
	(101.06,-11.74) --
	( 98.24,-11.74) --
	( 95.42,-11.74) --
	( 92.60,-11.74) --
	( 89.78,-11.74) --
	( 86.96,-11.74) --
	( 84.14,-11.74) --
	( 81.32,-11.74) --
	( 78.50,-11.74) --
	( 75.68,-11.74) --
	( 72.86,-11.74) --
	( 70.04,-11.74) --
	( 67.22,-11.74) --
	( 64.40,-11.74) --
	( 61.58,-11.74) --
	( 58.76,-11.74) --
	( 55.94,-11.74) --
	( 53.12,-11.74) --
	( 50.30,-11.74) --
	( 47.48,-11.74) --
	( 44.66,-11.74) --
	( 41.84,-11.74) --
	( 39.02,-11.74) --
	( 36.20,-11.74) --
	( 33.38,-11.74) --
	( 30.56,-11.74) --
	( 27.74,-11.74) --
	( 24.92,-11.74) --
	( 22.10,-11.74) --
	( 19.28,-11.74) --
	( 16.46,-11.74) --
	( 13.64,-11.74) --
	( 10.82,-11.74) --
	(  8.00,-11.74) --
	(  5.18,-11.74) --
	(  2.36,-11.74) --
	( -0.46,-11.74) --
	( -3.28,-11.74) --
	( -6.10,-11.74) --
	( -8.92,-11.74) --
	(-11.74,-11.74) --
	cycle;

\path[] (-11.74,267.44) --
	( -8.92,267.44) --
	( -6.10,267.44) --
	( -3.28,267.44) --
	( -0.46,267.44) --
	(  2.36,267.44) --
	(  5.18,267.44) --
	(  8.00,264.25) --
	( 10.82,259.47) --
	( 13.64,254.79) --
	( 16.46,250.24) --
	( 19.28,245.79) --
	( 22.10,241.46) --
	( 24.92,237.25) --
	( 27.74,233.15) --
	( 30.56,229.16) --
	( 33.38,225.28) --
	( 36.20,221.52) --
	( 39.02,217.88) --
	( 41.84,214.35) --
	( 44.66,210.93) --
	( 47.48,207.62) --
	( 50.30,204.43) --
	( 53.12,201.36) --
	( 55.94,198.39) --
	( 58.76,195.55) --
	( 61.58,192.81) --
	( 64.40,190.19) --
	( 67.22,187.68) --
	( 70.04,185.29) --
	( 72.86,183.01) --
	( 75.68,180.85) --
	( 78.50,178.80) --
	( 81.32,176.86) --
	( 84.14,175.04) --
	( 86.96,173.33) --
	( 89.78,171.73) --
	( 92.60,170.25) --
	( 95.42,168.88) --
	( 98.24,167.63) --
	(101.06,166.49) --
	(103.88,165.47) --
	(106.70,164.55) --
	(109.52,163.76) --
	(112.34,163.07) --
	(115.16,162.50) --
	(117.98,162.05) --
	(120.80,161.71) --
	(123.62,161.48) --
	(126.44,161.36) --
	(129.26,161.36) --
	(132.08,161.48) --
	(134.90,161.71) --
	(137.72,162.05) --
	(140.54,162.50) --
	(143.36,163.07) --
	(146.18,163.76) --
	(149.00,164.55) --
	(151.82,165.47) --
	(154.64,166.49) --
	(157.46,167.63) --
	(160.28,168.88) --
	(163.10,170.25) --
	(165.92,171.73) --
	(168.74,173.33) --
	(171.56,175.04) --
	(174.38,176.86) --
	(177.20,178.80) --
	(180.02,180.85) --
	(182.84,183.01) --
	(185.66,185.29) --
	(188.48,187.68) --
	(191.30,190.19) --
	(194.12,192.81) --
	(196.94,195.55) --
	(199.76,198.39) --
	(202.58,201.36) --
	(205.40,204.43) --
	(208.22,207.62) --
	(211.04,210.93) --
	(213.86,214.35) --
	(216.68,217.88) --
	(219.50,221.52) --
	(222.32,225.28) --
	(225.14,229.16) --
	(227.96,233.15) --
	(230.78,237.25) --
	(233.60,241.46) --
	(236.42,245.79) --
	(239.24,250.24) --
	(242.06,254.79) --
	(244.88,259.47) --
	(247.70,264.25) --
	(250.52,267.44) --
	(253.34,267.44) --
	(256.16,267.44) --
	(258.98,267.44) --
	(261.80,267.44) --
	(264.62,267.44) --
	(267.44,267.44);

\path[] (267.44,-11.74) --
	(264.62,-11.74) --
	(261.80,-11.74) --
	(258.98,-11.74) --
	(256.16,-11.74) --
	(253.34,-11.74) --
	(250.52,-11.74) --
	(247.70,-11.74) --
	(244.88,-11.74) --
	(242.06,-11.74) --
	(239.24,-11.74) --
	(236.42,-11.74) --
	(233.60,-11.74) --
	(230.78,-11.74) --
	(227.96,-11.74) --
	(225.14,-11.74) --
	(222.32,-11.74) --
	(219.50,-11.74) --
	(216.68,-11.74) --
	(213.86,-11.74) --
	(211.04,-11.74) --
	(208.22,-11.74) --
	(205.40,-11.74) --
	(202.58,-11.74) --
	(199.76,-11.74) --
	(196.94,-11.74) --
	(194.12,-11.74) --
	(191.30,-11.74) --
	(188.48,-11.74) --
	(185.66,-11.74) --
	(182.84,-11.74) --
	(180.02,-11.74) --
	(177.20,-11.74) --
	(174.38,-11.74) --
	(171.56,-11.74) --
	(168.74,-11.74) --
	(165.92,-11.74) --
	(163.10,-11.74) --
	(160.28,-11.74) --
	(157.46,-11.74) --
	(154.64,-11.74) --
	(151.82,-11.74) --
	(149.00,-11.74) --
	(146.18,-11.74) --
	(143.36,-11.74) --
	(140.54,-11.74) --
	(137.72,-11.74) --
	(134.90,-11.74) --
	(132.08,-11.74) --
	(129.26,-11.74) --
	(126.44,-11.74) --
	(123.62,-11.74) --
	(120.80,-11.74) --
	(117.98,-11.74) --
	(115.16,-11.74) --
	(112.34,-11.74) --
	(109.52,-11.74) --
	(106.70,-11.74) --
	(103.88,-11.74) --
	(101.06,-11.74) --
	( 98.24,-11.74) --
	( 95.42,-11.74) --
	( 92.60,-11.74) --
	( 89.78,-11.74) --
	( 86.96,-11.74) --
	( 84.14,-11.74) --
	( 81.32,-11.74) --
	( 78.50,-11.74) --
	( 75.68,-11.74) --
	( 72.86,-11.74) --
	( 70.04,-11.74) --
	( 67.22,-11.74) --
	( 64.40,-11.74) --
	( 61.58,-11.74) --
	( 58.76,-11.74) --
	( 55.94,-11.74) --
	( 53.12,-11.74) --
	( 50.30,-11.74) --
	( 47.48,-11.74) --
	( 44.66,-11.74) --
	( 41.84,-11.74) --
	( 39.02,-11.74) --
	( 36.20,-11.74) --
	( 33.38,-11.74) --
	( 30.56,-11.74) --
	( 27.74,-11.74) --
	( 24.92,-11.74) --
	( 22.10,-11.74) --
	( 19.28,-11.74) --
	( 16.46,-11.74) --
	( 13.64,-11.74) --
	( 10.82,-11.74) --
	(  8.00,-11.74) --
	(  5.18,-11.74) --
	(  2.36,-11.74) --
	( -0.46,-11.74) --
	( -3.28,-11.74) --
	( -6.10,-11.74) --
	( -8.92,-11.74) --
	(-11.74,-11.74);
  \node[text=MyOrange,inner sep=0pt,anchor=west, outer sep=0pt, scale=  2.5] at ( 80.00,240.00) {$U_1=5$};
 \node[text=fillColor,inner sep=0pt,anchor=west, outer sep=0pt, scale=  2.5] at ( 80.00,200.00) {$U_2=0$};
\end{scope}
\end{tikzpicture}}\\
\resizebox{\textwidth}{!}{
\begin{tikzpicture}[x=1pt,y=1pt]
\definecolor{fillColor}{RGB}{255,255,255}
\definecolor{MyOrange}{rgb}{0.8, 0.33, 0}
\path[use as bounding box,fill=fillColor,fill opacity=0.00] (0,0) rectangle (252.94,252.94);
\begin{scope}
\path[clip] (  0.00,  0.00) rectangle (252.94,252.94);
\definecolor{fillColor}{RGB}{69,33,84}

\path[fill=fillColor] (-11.74,133.43) --
	( -8.92,127.85) --
	( -6.10,122.38) --
	( -3.28,117.02) --
	( -0.46,111.78) --
	(  2.36,106.66) --
	(  5.18,101.64) --
	(  8.00, 96.74) --
	( 10.82, 91.96) --
	( 13.64, 87.29) --
	( 16.46, 82.73) --
	( 19.28, 78.28) --
	( 22.10, 73.95) --
	( 24.92, 69.74) --
	( 27.74, 65.64) --
	( 30.56, 61.65) --
	( 33.38, 57.77) --
	( 36.20, 54.01) --
	( 39.02, 50.37) --
	( 41.84, 46.84) --
	( 44.66, 43.42) --
	( 47.48, 40.11) --
	( 50.30, 36.92) --
	( 53.12, 33.85) --
	( 55.94, 30.88) --
	( 58.76, 28.04) --
	( 61.58, 25.30) --
	( 64.40, 22.68) --
	( 67.22, 20.17) --
	( 70.04, 17.78) --
	( 72.86, 15.50) --
	( 75.68, 13.34) --
	( 78.50, 11.29) --
	( 81.32,  9.35) --
	( 84.14,  7.53) --
	( 86.96,  5.82) --
	( 89.78,  4.22) --
	( 92.60,  2.74) --
	( 95.42,  1.37) --
	( 98.24,  0.12) --
	(101.06, -1.02) --
	(103.88, -2.04) --
	(106.70, -2.96) --
	(109.52, -3.75) --
	(112.34, -4.44) --
	(115.16, -5.01) --
	(117.98, -5.46) --
	(120.80, -5.80) --
	(123.62, -6.03) --
	(126.44, -6.15) --
	(129.26, -6.15) --
	(132.08, -6.03) --
	(134.90, -5.80) --
	(137.72, -5.46) --
	(140.54, -5.01) --
	(143.36, -4.44) --
	(146.18, -3.75) --
	(149.00, -2.96) --
	(151.82, -2.04) --
	(154.64, -1.02) --
	(157.46,  0.12) --
	(160.28,  1.37) --
	(163.10,  2.74) --
	(165.92,  4.22) --
	(168.74,  5.82) --
	(171.56,  7.53) --
	(174.38,  9.35) --
	(177.20, 11.29) --
	(180.02, 13.34) --
	(182.84, 15.50) --
	(185.66, 17.78) --
	(188.48, 20.17) --
	(191.30, 22.68) --
	(194.12, 25.30) --
	(196.94, 28.04) --
	(199.76, 30.88) --
	(202.58, 33.85) --
	(205.40, 36.92) --
	(208.22, 40.11) --
	(211.04, 43.42) --
	(213.86, 46.84) --
	(216.68, 50.37) --
	(219.50, 54.01) --
	(222.32, 57.77) --
	(225.14, 61.65) --
	(227.96, 65.64) --
	(230.78, 69.74) --
	(233.60, 73.95) --
	(236.42, 78.28) --
	(239.24, 82.73) --
	(242.06, 87.29) --
	(244.88, 91.96) --
	(247.70, 96.74) --
	(250.52,101.64) --
	(253.34,106.66) --
	(256.16,111.78) --
	(258.98,117.02) --
	(261.80,122.38) --
	(264.62,127.85) --
	(267.44,133.43) --
	(267.44,-11.74) --
	(264.62,-11.74) --
	(261.80,-11.74) --
	(258.98,-11.74) --
	(256.16,-11.74) --
	(253.34,-11.74) --
	(250.52,-11.74) --
	(247.70,-11.74) --
	(244.88,-11.74) --
	(242.06,-11.74) --
	(239.24,-11.74) --
	(236.42,-11.74) --
	(233.60,-11.74) --
	(230.78,-11.74) --
	(227.96,-11.74) --
	(225.14,-11.74) --
	(222.32,-11.74) --
	(219.50,-11.74) --
	(216.68,-11.74) --
	(213.86,-11.74) --
	(211.04,-11.74) --
	(208.22,-11.74) --
	(205.40,-11.74) --
	(202.58,-11.74) --
	(199.76,-11.74) --
	(196.94,-11.74) --
	(194.12,-11.74) --
	(191.30,-11.74) --
	(188.48,-11.74) --
	(185.66,-11.74) --
	(182.84,-11.74) --
	(180.02,-11.74) --
	(177.20,-11.74) --
	(174.38,-11.74) --
	(171.56,-11.74) --
	(168.74,-11.74) --
	(165.92,-11.74) --
	(163.10,-11.74) --
	(160.28,-11.74) --
	(157.46,-11.74) --
	(154.64,-11.74) --
	(151.82,-11.74) --
	(149.00,-11.74) --
	(146.18,-11.74) --
	(143.36,-11.74) --
	(140.54,-11.74) --
	(137.72,-11.74) --
	(134.90,-11.74) --
	(132.08,-11.74) --
	(129.26,-11.74) --
	(126.44,-11.74) --
	(123.62,-11.74) --
	(120.80,-11.74) --
	(117.98,-11.74) --
	(115.16,-11.74) --
	(112.34,-11.74) --
	(109.52,-11.74) --
	(106.70,-11.74) --
	(103.88,-11.74) --
	(101.06,-11.74) --
	( 98.24,-11.74) --
	( 95.42,-11.74) --
	( 92.60,-11.74) --
	( 89.78,-11.74) --
	( 86.96,-11.74) --
	( 84.14,-11.74) --
	( 81.32,-11.74) --
	( 78.50,-11.74) --
	( 75.68,-11.74) --
	( 72.86,-11.74) --
	( 70.04,-11.74) --
	( 67.22,-11.74) --
	( 64.40,-11.74) --
	( 61.58,-11.74) --
	( 58.76,-11.74) --
	( 55.94,-11.74) --
	( 53.12,-11.74) --
	( 50.30,-11.74) --
	( 47.48,-11.74) --
	( 44.66,-11.74) --
	( 41.84,-11.74) --
	( 39.02,-11.74) --
	( 36.20,-11.74) --
	( 33.38,-11.74) --
	( 30.56,-11.74) --
	( 27.74,-11.74) --
	( 24.92,-11.74) --
	( 22.10,-11.74) --
	( 19.28,-11.74) --
	( 16.46,-11.74) --
	( 13.64,-11.74) --
	( 10.82,-11.74) --
	(  8.00,-11.74) --
	(  5.18,-11.74) --
	(  2.36,-11.74) --
	( -0.46,-11.74) --
	( -3.28,-11.74) --
	( -6.10,-11.74) --
	( -8.92,-11.74) --
	(-11.74,-11.74) --
	cycle;

\path[] (-11.74,133.43) --
	( -8.92,127.85) --
	( -6.10,122.38) --
	( -3.28,117.02) --
	( -0.46,111.78) --
	(  2.36,106.66) --
	(  5.18,101.64) --
	(  8.00, 96.74) --
	( 10.82, 91.96) --
	( 13.64, 87.29) --
	( 16.46, 82.73) --
	( 19.28, 78.28) --
	( 22.10, 73.95) --
	( 24.92, 69.74) --
	( 27.74, 65.64) --
	( 30.56, 61.65) --
	( 33.38, 57.77) --
	( 36.20, 54.01) --
	( 39.02, 50.37) --
	( 41.84, 46.84) --
	( 44.66, 43.42) --
	( 47.48, 40.11) --
	( 50.30, 36.92) --
	( 53.12, 33.85) --
	( 55.94, 30.88) --
	( 58.76, 28.04) --
	( 61.58, 25.30) --
	( 64.40, 22.68) --
	( 67.22, 20.17) --
	( 70.04, 17.78) --
	( 72.86, 15.50) --
	( 75.68, 13.34) --
	( 78.50, 11.29) --
	( 81.32,  9.35) --
	( 84.14,  7.53) --
	( 86.96,  5.82) --
	( 89.78,  4.22) --
	( 92.60,  2.74) --
	( 95.42,  1.37) --
	( 98.24,  0.12) --
	(101.06, -1.02) --
	(103.88, -2.04) --
	(106.70, -2.96) --
	(109.52, -3.75) --
	(112.34, -4.44) --
	(115.16, -5.01) --
	(117.98, -5.46) --
	(120.80, -5.80) --
	(123.62, -6.03) --
	(126.44, -6.15) --
	(129.26, -6.15) --
	(132.08, -6.03) --
	(134.90, -5.80) --
	(137.72, -5.46) --
	(140.54, -5.01) --
	(143.36, -4.44) --
	(146.18, -3.75) --
	(149.00, -2.96) --
	(151.82, -2.04) --
	(154.64, -1.02) --
	(157.46,  0.12) --
	(160.28,  1.37) --
	(163.10,  2.74) --
	(165.92,  4.22) --
	(168.74,  5.82) --
	(171.56,  7.53) --
	(174.38,  9.35) --
	(177.20, 11.29) --
	(180.02, 13.34) --
	(182.84, 15.50) --
	(185.66, 17.78) --
	(188.48, 20.17) --
	(191.30, 22.68) --
	(194.12, 25.30) --
	(196.94, 28.04) --
	(199.76, 30.88) --
	(202.58, 33.85) --
	(205.40, 36.92) --
	(208.22, 40.11) --
	(211.04, 43.42) --
	(213.86, 46.84) --
	(216.68, 50.37) --
	(219.50, 54.01) --
	(222.32, 57.77) --
	(225.14, 61.65) --
	(227.96, 65.64) --
	(230.78, 69.74) --
	(233.60, 73.95) --
	(236.42, 78.28) --
	(239.24, 82.73) --
	(242.06, 87.29) --
	(244.88, 91.96) --
	(247.70, 96.74) --
	(250.52,101.64) --
	(253.34,106.66) --
	(256.16,111.78) --
	(258.98,117.02) --
	(261.80,122.38) --
	(264.62,127.85) --
	(267.44,133.43);

\path[] (267.44,-11.74) --
	(264.62,-11.74) --
	(261.80,-11.74) --
	(258.98,-11.74) --
	(256.16,-11.74) --
	(253.34,-11.74) --
	(250.52,-11.74) --
	(247.70,-11.74) --
	(244.88,-11.74) --
	(242.06,-11.74) --
	(239.24,-11.74) --
	(236.42,-11.74) --
	(233.60,-11.74) --
	(230.78,-11.74) --
	(227.96,-11.74) --
	(225.14,-11.74) --
	(222.32,-11.74) --
	(219.50,-11.74) --
	(216.68,-11.74) --
	(213.86,-11.74) --
	(211.04,-11.74) --
	(208.22,-11.74) --
	(205.40,-11.74) --
	(202.58,-11.74) --
	(199.76,-11.74) --
	(196.94,-11.74) --
	(194.12,-11.74) --
	(191.30,-11.74) --
	(188.48,-11.74) --
	(185.66,-11.74) --
	(182.84,-11.74) --
	(180.02,-11.74) --
	(177.20,-11.74) --
	(174.38,-11.74) --
	(171.56,-11.74) --
	(168.74,-11.74) --
	(165.92,-11.74) --
	(163.10,-11.74) --
	(160.28,-11.74) --
	(157.46,-11.74) --
	(154.64,-11.74) --
	(151.82,-11.74) --
	(149.00,-11.74) --
	(146.18,-11.74) --
	(143.36,-11.74) --
	(140.54,-11.74) --
	(137.72,-11.74) --
	(134.90,-11.74) --
	(132.08,-11.74) --
	(129.26,-11.74) --
	(126.44,-11.74) --
	(123.62,-11.74) --
	(120.80,-11.74) --
	(117.98,-11.74) --
	(115.16,-11.74) --
	(112.34,-11.74) --
	(109.52,-11.74) --
	(106.70,-11.74) --
	(103.88,-11.74) --
	(101.06,-11.74) --
	( 98.24,-11.74) --
	( 95.42,-11.74) --
	( 92.60,-11.74) --
	( 89.78,-11.74) --
	( 86.96,-11.74) --
	( 84.14,-11.74) --
	( 81.32,-11.74) --
	( 78.50,-11.74) --
	( 75.68,-11.74) --
	( 72.86,-11.74) --
	( 70.04,-11.74) --
	( 67.22,-11.74) --
	( 64.40,-11.74) --
	( 61.58,-11.74) --
	( 58.76,-11.74) --
	( 55.94,-11.74) --
	( 53.12,-11.74) --
	( 50.30,-11.74) --
	( 47.48,-11.74) --
	( 44.66,-11.74) --
	( 41.84,-11.74) --
	( 39.02,-11.74) --
	( 36.20,-11.74) --
	( 33.38,-11.74) --
	( 30.56,-11.74) --
	( 27.74,-11.74) --
	( 24.92,-11.74) --
	( 22.10,-11.74) --
	( 19.28,-11.74) --
	( 16.46,-11.74) --
	( 13.64,-11.74) --
	( 10.82,-11.74) --
	(  8.00,-11.74) --
	(  5.18,-11.74) --
	(  2.36,-11.74) --
	( -0.46,-11.74) --
	( -3.28,-11.74) --
	( -6.10,-11.74) --
	( -8.92,-11.74) --
	(-11.74,-11.74);

\node[text=fillColor,inner sep=0pt,anchor=west, outer sep=0pt, scale=  2.5] at ( 80.00,215.00) {$U_1=0$};
 \node[text=MyOrange,inner sep=0pt,anchor=west, outer sep=0pt, scale=  2.5] at ( 80.00,175.00) {$U_2=\text{-}5$};
\end{scope}
\end{tikzpicture}
\begin{tikzpicture}[x=1pt,y=1pt]
\definecolor{fillColor}{RGB}{255,255,255}
\definecolor{MyOrange}{rgb}{0.8, 0.33, 0}
\path[use as bounding box,fill=fillColor,fill opacity=0.00] (0,0) rectangle (252.94,252.94);
\begin{scope}
\path[clip] (  0.00,  0.00) rectangle (252.94,252.94);
\definecolor{fillColor}{RGB}{69,33,84}

\path[fill=fillColor] (-11.74,238.12) --
	( -8.92,232.54) --
	( -6.10,227.07) --
	( -3.28,221.72) --
	( -0.46,216.48) --
	(  2.36,211.35) --
	(  5.18,206.34) --
	(  8.00,201.44) --
	( 10.82,196.65) --
	( 13.64,191.98) --
	( 16.46,187.42) --
	( 19.28,182.98) --
	( 22.10,178.65) --
	( 24.92,174.43) --
	( 27.74,170.33) --
	( 30.56,166.34) --
	( 33.38,162.47) --
	( 36.20,158.71) --
	( 39.02,155.06) --
	( 41.84,151.53) --
	( 44.66,148.11) --
	( 47.48,144.81) --
	( 50.30,141.62) --
	( 53.12,138.54) --
	( 55.94,135.58) --
	( 58.76,132.73) --
	( 61.58,130.00) --
	( 64.40,127.37) --
	( 67.22,124.87) --
	( 70.04,122.48) --
	( 72.86,120.20) --
	( 75.68,118.03) --
	( 78.50,115.98) --
	( 81.32,114.04) --
	( 84.14,112.22) --
	( 86.96,110.51) --
	( 89.78,108.92) --
	( 92.60,107.43) --
	( 95.42,106.07) --
	( 98.24,104.81) --
	(101.06,103.67) --
	(103.88,102.65) --
	(106.70,101.74) --
	(109.52,100.94) --
	(112.34,100.26) --
	(115.16, 99.69) --
	(117.98, 99.23) --
	(120.80, 98.89) --
	(123.62, 98.66) --
	(126.44, 98.55) --
	(129.26, 98.55) --
	(132.08, 98.66) --
	(134.90, 98.89) --
	(137.72, 99.23) --
	(140.54, 99.69) --
	(143.36,100.26) --
	(146.18,100.94) --
	(149.00,101.74) --
	(151.82,102.65) --
	(154.64,103.67) --
	(157.46,104.81) --
	(160.28,106.07) --
	(163.10,107.43) --
	(165.92,108.92) --
	(168.74,110.51) --
	(171.56,112.22) --
	(174.38,114.04) --
	(177.20,115.98) --
	(180.02,118.03) --
	(182.84,120.20) --
	(185.66,122.48) --
	(188.48,124.87) --
	(191.30,127.37) --
	(194.12,130.00) --
	(196.94,132.73) --
	(199.76,135.58) --
	(202.58,138.54) --
	(205.40,141.62) --
	(208.22,144.81) --
	(211.04,148.11) --
	(213.86,151.53) --
	(216.68,155.06) --
	(219.50,158.71) --
	(222.32,162.47) --
	(225.14,166.34) --
	(227.96,170.33) --
	(230.78,174.43) --
	(233.60,178.65) --
	(236.42,182.98) --
	(239.24,187.42) --
	(242.06,191.98) --
	(244.88,196.65) --
	(247.70,201.44) --
	(250.52,206.34) --
	(253.34,211.35) --
	(256.16,216.48) --
	(258.98,221.72) --
	(261.80,227.07) --
	(264.62,232.54) --
	(267.44,238.12) --
	(267.44,-11.74) --
	(264.62,-11.74) --
	(261.80,-11.74) --
	(258.98,-11.74) --
	(256.16,-11.74) --
	(253.34,-11.74) --
	(250.52,-11.74) --
	(247.70,-11.74) --
	(244.88,-11.74) --
	(242.06,-11.74) --
	(239.24,-11.74) --
	(236.42,-11.74) --
	(233.60,-11.74) --
	(230.78,-11.74) --
	(227.96,-11.74) --
	(225.14,-11.74) --
	(222.32,-11.74) --
	(219.50,-11.74) --
	(216.68,-11.74) --
	(213.86,-11.74) --
	(211.04,-11.74) --
	(208.22,-11.74) --
	(205.40,-11.74) --
	(202.58,-11.74) --
	(199.76,-11.74) --
	(196.94,-11.74) --
	(194.12,-11.74) --
	(191.30,-11.74) --
	(188.48,-11.74) --
	(185.66,-11.74) --
	(182.84,-11.74) --
	(180.02,-11.74) --
	(177.20,-11.74) --
	(174.38,-11.74) --
	(171.56,-11.74) --
	(168.74,-11.74) --
	(165.92,-11.74) --
	(163.10,-11.74) --
	(160.28,-11.74) --
	(157.46,-11.74) --
	(154.64,-11.74) --
	(151.82,-11.74) --
	(149.00,-11.74) --
	(146.18,-11.74) --
	(143.36,-11.74) --
	(140.54,-11.74) --
	(137.72,-11.74) --
	(134.90,-11.74) --
	(132.08,-11.74) --
	(129.26,-11.74) --
	(126.44,-11.74) --
	(123.62,-11.74) --
	(120.80,-11.74) --
	(117.98,-11.74) --
	(115.16,-11.74) --
	(112.34,-11.74) --
	(109.52,-11.74) --
	(106.70,-11.74) --
	(103.88,-11.74) --
	(101.06,-11.74) --
	( 98.24,-11.74) --
	( 95.42,-11.74) --
	( 92.60,-11.74) --
	( 89.78,-11.74) --
	( 86.96,-11.74) --
	( 84.14,-11.74) --
	( 81.32,-11.74) --
	( 78.50,-11.74) --
	( 75.68,-11.74) --
	( 72.86,-11.74) --
	( 70.04,-11.74) --
	( 67.22,-11.74) --
	( 64.40,-11.74) --
	( 61.58,-11.74) --
	( 58.76,-11.74) --
	( 55.94,-11.74) --
	( 53.12,-11.74) --
	( 50.30,-11.74) --
	( 47.48,-11.74) --
	( 44.66,-11.74) --
	( 41.84,-11.74) --
	( 39.02,-11.74) --
	( 36.20,-11.74) --
	( 33.38,-11.74) --
	( 30.56,-11.74) --
	( 27.74,-11.74) --
	( 24.92,-11.74) --
	( 22.10,-11.74) --
	( 19.28,-11.74) --
	( 16.46,-11.74) --
	( 13.64,-11.74) --
	( 10.82,-11.74) --
	(  8.00,-11.74) --
	(  5.18,-11.74) --
	(  2.36,-11.74) --
	( -0.46,-11.74) --
	( -3.28,-11.74) --
	( -6.10,-11.74) --
	( -8.92,-11.74) --
	(-11.74,-11.74) --
	cycle;

\path[] (-11.74,238.12) --
	( -8.92,232.54) --
	( -6.10,227.07) --
	( -3.28,221.72) --
	( -0.46,216.48) --
	(  2.36,211.35) --
	(  5.18,206.34) --
	(  8.00,201.44) --
	( 10.82,196.65) --
	( 13.64,191.98) --
	( 16.46,187.42) --
	( 19.28,182.98) --
	( 22.10,178.65) --
	( 24.92,174.43) --
	( 27.74,170.33) --
	( 30.56,166.34) --
	( 33.38,162.47) --
	( 36.20,158.71) --
	( 39.02,155.06) --
	( 41.84,151.53) --
	( 44.66,148.11) --
	( 47.48,144.81) --
	( 50.30,141.62) --
	( 53.12,138.54) --
	( 55.94,135.58) --
	( 58.76,132.73) --
	( 61.58,130.00) --
	( 64.40,127.37) --
	( 67.22,124.87) --
	( 70.04,122.48) --
	( 72.86,120.20) --
	( 75.68,118.03) --
	( 78.50,115.98) --
	( 81.32,114.04) --
	( 84.14,112.22) --
	( 86.96,110.51) --
	( 89.78,108.92) --
	( 92.60,107.43) --
	( 95.42,106.07) --
	( 98.24,104.81) --
	(101.06,103.67) --
	(103.88,102.65) --
	(106.70,101.74) --
	(109.52,100.94) --
	(112.34,100.26) --
	(115.16, 99.69) --
	(117.98, 99.23) --
	(120.80, 98.89) --
	(123.62, 98.66) --
	(126.44, 98.55) --
	(129.26, 98.55) --
	(132.08, 98.66) --
	(134.90, 98.89) --
	(137.72, 99.23) --
	(140.54, 99.69) --
	(143.36,100.26) --
	(146.18,100.94) --
	(149.00,101.74) --
	(151.82,102.65) --
	(154.64,103.67) --
	(157.46,104.81) --
	(160.28,106.07) --
	(163.10,107.43) --
	(165.92,108.92) --
	(168.74,110.51) --
	(171.56,112.22) --
	(174.38,114.04) --
	(177.20,115.98) --
	(180.02,118.03) --
	(182.84,120.20) --
	(185.66,122.48) --
	(188.48,124.87) --
	(191.30,127.37) --
	(194.12,130.00) --
	(196.94,132.73) --
	(199.76,135.58) --
	(202.58,138.54) --
	(205.40,141.62) --
	(208.22,144.81) --
	(211.04,148.11) --
	(213.86,151.53) --
	(216.68,155.06) --
	(219.50,158.71) --
	(222.32,162.47) --
	(225.14,166.34) --
	(227.96,170.33) --
	(230.78,174.43) --
	(233.60,178.65) --
	(236.42,182.98) --
	(239.24,187.42) --
	(242.06,191.98) --
	(244.88,196.65) --
	(247.70,201.44) --
	(250.52,206.34) --
	(253.34,211.35) --
	(256.16,216.48) --
	(258.98,221.72) --
	(261.80,227.07) --
	(264.62,232.54) --
	(267.44,238.12);

\path[] (267.44,-11.74) --
	(264.62,-11.74) --
	(261.80,-11.74) --
	(258.98,-11.74) --
	(256.16,-11.74) --
	(253.34,-11.74) --
	(250.52,-11.74) --
	(247.70,-11.74) --
	(244.88,-11.74) --
	(242.06,-11.74) --
	(239.24,-11.74) --
	(236.42,-11.74) --
	(233.60,-11.74) --
	(230.78,-11.74) --
	(227.96,-11.74) --
	(225.14,-11.74) --
	(222.32,-11.74) --
	(219.50,-11.74) --
	(216.68,-11.74) --
	(213.86,-11.74) --
	(211.04,-11.74) --
	(208.22,-11.74) --
	(205.40,-11.74) --
	(202.58,-11.74) --
	(199.76,-11.74) --
	(196.94,-11.74) --
	(194.12,-11.74) --
	(191.30,-11.74) --
	(188.48,-11.74) --
	(185.66,-11.74) --
	(182.84,-11.74) --
	(180.02,-11.74) --
	(177.20,-11.74) --
	(174.38,-11.74) --
	(171.56,-11.74) --
	(168.74,-11.74) --
	(165.92,-11.74) --
	(163.10,-11.74) --
	(160.28,-11.74) --
	(157.46,-11.74) --
	(154.64,-11.74) --
	(151.82,-11.74) --
	(149.00,-11.74) --
	(146.18,-11.74) --
	(143.36,-11.74) --
	(140.54,-11.74) --
	(137.72,-11.74) --
	(134.90,-11.74) --
	(132.08,-11.74) --
	(129.26,-11.74) --
	(126.44,-11.74) --
	(123.62,-11.74) --
	(120.80,-11.74) --
	(117.98,-11.74) --
	(115.16,-11.74) --
	(112.34,-11.74) --
	(109.52,-11.74) --
	(106.70,-11.74) --
	(103.88,-11.74) --
	(101.06,-11.74) --
	( 98.24,-11.74) --
	( 95.42,-11.74) --
	( 92.60,-11.74) --
	( 89.78,-11.74) --
	( 86.96,-11.74) --
	( 84.14,-11.74) --
	( 81.32,-11.74) --
	( 78.50,-11.74) --
	( 75.68,-11.74) --
	( 72.86,-11.74) --
	( 70.04,-11.74) --
	( 67.22,-11.74) --
	( 64.40,-11.74) --
	( 61.58,-11.74) --
	( 58.76,-11.74) --
	( 55.94,-11.74) --
	( 53.12,-11.74) --
	( 50.30,-11.74) --
	( 47.48,-11.74) --
	( 44.66,-11.74) --
	( 41.84,-11.74) --
	( 39.02,-11.74) --
	( 36.20,-11.74) --
	( 33.38,-11.74) --
	( 30.56,-11.74) --
	( 27.74,-11.74) --
	( 24.92,-11.74) --
	( 22.10,-11.74) --
	( 19.28,-11.74) --
	( 16.46,-11.74) --
	( 13.64,-11.74) --
	( 10.82,-11.74) --
	(  8.00,-11.74) --
	(  5.18,-11.74) --
	(  2.36,-11.74) --
	( -0.46,-11.74) --
	( -3.28,-11.74) --
	( -6.10,-11.74) --
	( -8.92,-11.74) --
	(-11.74,-11.74);
\node[text=fillColor,inner sep=0pt,anchor=west, outer sep=0pt, scale=  2.5] at ( 80.00,215.00) {$U_1=0$};
 \node[text=MyOrange,inner sep=0pt,anchor=west, outer sep=0pt, scale=  2.5] at ( 80.00,175.00) {$U_2=\text{-}2.5$};
\end{scope}
\end{tikzpicture}
\begin{tikzpicture}[x=1pt,y=1pt]
\definecolor{fillColor}{RGB}{255,255,255}
\definecolor{MyOrange}{rgb}{0.8, 0.33, 0}
\path[use as bounding box,fill=fillColor,fill opacity=0.00] (0,0) rectangle (252.94,252.94);
\begin{scope}
\path[clip] (  0.00,  0.00) rectangle (252.94,252.94);
\definecolor{fillColor}{RGB}{69,33,84}

\path[fill=fillColor] (-11.74,267.44) --
	( -8.92,267.44) --
	( -6.10,261.97) --
	( -3.28,256.61) --
	( -0.46,251.37) --
	(  2.36,246.25) --
	(  5.18,241.23) --
	(  8.00,236.33) --
	( 10.82,231.55) --
	( 13.64,226.88) --
	( 16.46,222.32) --
	( 19.28,217.88) --
	( 22.10,213.55) --
	( 24.92,209.33) --
	( 27.74,205.23) --
	( 30.56,201.24) --
	( 33.38,197.37) --
	( 36.20,193.61) --
	( 39.02,189.96) --
	( 41.84,186.43) --
	( 44.66,183.01) --
	( 47.48,179.71) --
	( 50.30,176.51) --
	( 53.12,173.44) --
	( 55.94,170.48) --
	( 58.76,167.63) --
	( 61.58,164.89) --
	( 64.40,162.27) --
	( 67.22,159.77) --
	( 70.04,157.37) --
	( 72.86,155.09) --
	( 75.68,152.93) --
	( 78.50,150.88) --
	( 81.32,148.94) --
	( 84.14,147.12) --
	( 86.96,145.41) --
	( 89.78,143.81) --
	( 92.60,142.33) --
	( 95.42,140.97) --
	( 98.24,139.71) --
	(101.06,138.57) --
	(103.88,137.55) --
	(106.70,136.64) --
	(109.52,135.84) --
	(112.34,135.15) --
	(115.16,134.58) --
	(117.98,134.13) --
	(120.80,133.79) --
	(123.62,133.56) --
	(126.44,133.45) --
	(129.26,133.45) --
	(132.08,133.56) --
	(134.90,133.79) --
	(137.72,134.13) --
	(140.54,134.58) --
	(143.36,135.15) --
	(146.18,135.84) --
	(149.00,136.64) --
	(151.82,137.55) --
	(154.64,138.57) --
	(157.46,139.71) --
	(160.28,140.97) --
	(163.10,142.33) --
	(165.92,143.81) --
	(168.74,145.41) --
	(171.56,147.12) --
	(174.38,148.94) --
	(177.20,150.88) --
	(180.02,152.93) --
	(182.84,155.09) --
	(185.66,157.37) --
	(188.48,159.77) --
	(191.30,162.27) --
	(194.12,164.89) --
	(196.94,167.63) --
	(199.76,170.48) --
	(202.58,173.44) --
	(205.40,176.51) --
	(208.22,179.71) --
	(211.04,183.01) --
	(213.86,186.43) --
	(216.68,189.96) --
	(219.50,193.61) --
	(222.32,197.37) --
	(225.14,201.24) --
	(227.96,205.23) --
	(230.78,209.33) --
	(233.60,213.55) --
	(236.42,217.88) --
	(239.24,222.32) --
	(242.06,226.88) --
	(244.88,231.55) --
	(247.70,236.33) --
	(250.52,241.23) --
	(253.34,246.25) --
	(256.16,251.37) --
	(258.98,256.61) --
	(261.80,261.97) --
	(264.62,267.44) --
	(267.44,267.44) --
	(267.44,-11.74) --
	(264.62,-11.74) --
	(261.80,-11.74) --
	(258.98,-11.74) --
	(256.16,-11.74) --
	(253.34,-11.74) --
	(250.52,-11.74) --
	(247.70,-11.74) --
	(244.88,-11.74) --
	(242.06,-11.74) --
	(239.24,-11.74) --
	(236.42,-11.74) --
	(233.60,-11.74) --
	(230.78,-11.74) --
	(227.96,-11.74) --
	(225.14,-11.74) --
	(222.32,-11.74) --
	(219.50,-11.74) --
	(216.68,-11.74) --
	(213.86,-11.74) --
	(211.04,-11.74) --
	(208.22,-11.74) --
	(205.40,-11.74) --
	(202.58,-11.74) --
	(199.76,-11.74) --
	(196.94,-11.74) --
	(194.12,-11.74) --
	(191.30,-11.74) --
	(188.48,-11.74) --
	(185.66,-11.74) --
	(182.84,-11.74) --
	(180.02,-11.74) --
	(177.20,-11.74) --
	(174.38,-11.74) --
	(171.56,-11.74) --
	(168.74,-11.74) --
	(165.92,-11.74) --
	(163.10,-11.74) --
	(160.28,-11.74) --
	(157.46,-11.74) --
	(154.64,-11.74) --
	(151.82,-11.74) --
	(149.00,-11.74) --
	(146.18,-11.74) --
	(143.36,-11.74) --
	(140.54,-11.74) --
	(137.72,-11.74) --
	(134.90,-11.74) --
	(132.08,-11.74) --
	(129.26,-11.74) --
	(126.44,-11.74) --
	(123.62,-11.74) --
	(120.80,-11.74) --
	(117.98,-11.74) --
	(115.16,-11.74) --
	(112.34,-11.74) --
	(109.52,-11.74) --
	(106.70,-11.74) --
	(103.88,-11.74) --
	(101.06,-11.74) --
	( 98.24,-11.74) --
	( 95.42,-11.74) --
	( 92.60,-11.74) --
	( 89.78,-11.74) --
	( 86.96,-11.74) --
	( 84.14,-11.74) --
	( 81.32,-11.74) --
	( 78.50,-11.74) --
	( 75.68,-11.74) --
	( 72.86,-11.74) --
	( 70.04,-11.74) --
	( 67.22,-11.74) --
	( 64.40,-11.74) --
	( 61.58,-11.74) --
	( 58.76,-11.74) --
	( 55.94,-11.74) --
	( 53.12,-11.74) --
	( 50.30,-11.74) --
	( 47.48,-11.74) --
	( 44.66,-11.74) --
	( 41.84,-11.74) --
	( 39.02,-11.74) --
	( 36.20,-11.74) --
	( 33.38,-11.74) --
	( 30.56,-11.74) --
	( 27.74,-11.74) --
	( 24.92,-11.74) --
	( 22.10,-11.74) --
	( 19.28,-11.74) --
	( 16.46,-11.74) --
	( 13.64,-11.74) --
	( 10.82,-11.74) --
	(  8.00,-11.74) --
	(  5.18,-11.74) --
	(  2.36,-11.74) --
	( -0.46,-11.74) --
	( -3.28,-11.74) --
	( -6.10,-11.74) --
	( -8.92,-11.74) --
	(-11.74,-11.74) --
	cycle;

\path[] (-11.74,267.44) --
	( -8.92,267.44) --
	( -6.10,261.97) --
	( -3.28,256.61) --
	( -0.46,251.37) --
	(  2.36,246.25) --
	(  5.18,241.23) --
	(  8.00,236.33) --
	( 10.82,231.55) --
	( 13.64,226.88) --
	( 16.46,222.32) --
	( 19.28,217.88) --
	( 22.10,213.55) --
	( 24.92,209.33) --
	( 27.74,205.23) --
	( 30.56,201.24) --
	( 33.38,197.37) --
	( 36.20,193.61) --
	( 39.02,189.96) --
	( 41.84,186.43) --
	( 44.66,183.01) --
	( 47.48,179.71) --
	( 50.30,176.51) --
	( 53.12,173.44) --
	( 55.94,170.48) --
	( 58.76,167.63) --
	( 61.58,164.89) --
	( 64.40,162.27) --
	( 67.22,159.77) --
	( 70.04,157.37) --
	( 72.86,155.09) --
	( 75.68,152.93) --
	( 78.50,150.88) --
	( 81.32,148.94) --
	( 84.14,147.12) --
	( 86.96,145.41) --
	( 89.78,143.81) --
	( 92.60,142.33) --
	( 95.42,140.97) --
	( 98.24,139.71) --
	(101.06,138.57) --
	(103.88,137.55) --
	(106.70,136.64) --
	(109.52,135.84) --
	(112.34,135.15) --
	(115.16,134.58) --
	(117.98,134.13) --
	(120.80,133.79) --
	(123.62,133.56) --
	(126.44,133.45) --
	(129.26,133.45) --
	(132.08,133.56) --
	(134.90,133.79) --
	(137.72,134.13) --
	(140.54,134.58) --
	(143.36,135.15) --
	(146.18,135.84) --
	(149.00,136.64) --
	(151.82,137.55) --
	(154.64,138.57) --
	(157.46,139.71) --
	(160.28,140.97) --
	(163.10,142.33) --
	(165.92,143.81) --
	(168.74,145.41) --
	(171.56,147.12) --
	(174.38,148.94) --
	(177.20,150.88) --
	(180.02,152.93) --
	(182.84,155.09) --
	(185.66,157.37) --
	(188.48,159.77) --
	(191.30,162.27) --
	(194.12,164.89) --
	(196.94,167.63) --
	(199.76,170.48) --
	(202.58,173.44) --
	(205.40,176.51) --
	(208.22,179.71) --
	(211.04,183.01) --
	(213.86,186.43) --
	(216.68,189.96) --
	(219.50,193.61) --
	(222.32,197.37) --
	(225.14,201.24) --
	(227.96,205.23) --
	(230.78,209.33) --
	(233.60,213.55) --
	(236.42,217.88) --
	(239.24,222.32) --
	(242.06,226.88) --
	(244.88,231.55) --
	(247.70,236.33) --
	(250.52,241.23) --
	(253.34,246.25) --
	(256.16,251.37) --
	(258.98,256.61) --
	(261.80,261.97) --
	(264.62,267.44) --
	(267.44,267.44);

\path[] (267.44,-11.74) --
	(264.62,-11.74) --
	(261.80,-11.74) --
	(258.98,-11.74) --
	(256.16,-11.74) --
	(253.34,-11.74) --
	(250.52,-11.74) --
	(247.70,-11.74) --
	(244.88,-11.74) --
	(242.06,-11.74) --
	(239.24,-11.74) --
	(236.42,-11.74) --
	(233.60,-11.74) --
	(230.78,-11.74) --
	(227.96,-11.74) --
	(225.14,-11.74) --
	(222.32,-11.74) --
	(219.50,-11.74) --
	(216.68,-11.74) --
	(213.86,-11.74) --
	(211.04,-11.74) --
	(208.22,-11.74) --
	(205.40,-11.74) --
	(202.58,-11.74) --
	(199.76,-11.74) --
	(196.94,-11.74) --
	(194.12,-11.74) --
	(191.30,-11.74) --
	(188.48,-11.74) --
	(185.66,-11.74) --
	(182.84,-11.74) --
	(180.02,-11.74) --
	(177.20,-11.74) --
	(174.38,-11.74) --
	(171.56,-11.74) --
	(168.74,-11.74) --
	(165.92,-11.74) --
	(163.10,-11.74) --
	(160.28,-11.74) --
	(157.46,-11.74) --
	(154.64,-11.74) --
	(151.82,-11.74) --
	(149.00,-11.74) --
	(146.18,-11.74) --
	(143.36,-11.74) --
	(140.54,-11.74) --
	(137.72,-11.74) --
	(134.90,-11.74) --
	(132.08,-11.74) --
	(129.26,-11.74) --
	(126.44,-11.74) --
	(123.62,-11.74) --
	(120.80,-11.74) --
	(117.98,-11.74) --
	(115.16,-11.74) --
	(112.34,-11.74) --
	(109.52,-11.74) --
	(106.70,-11.74) --
	(103.88,-11.74) --
	(101.06,-11.74) --
	( 98.24,-11.74) --
	( 95.42,-11.74) --
	( 92.60,-11.74) --
	( 89.78,-11.74) --
	( 86.96,-11.74) --
	( 84.14,-11.74) --
	( 81.32,-11.74) --
	( 78.50,-11.74) --
	( 75.68,-11.74) --
	( 72.86,-11.74) --
	( 70.04,-11.74) --
	( 67.22,-11.74) --
	( 64.40,-11.74) --
	( 61.58,-11.74) --
	( 58.76,-11.74) --
	( 55.94,-11.74) --
	( 53.12,-11.74) --
	( 50.30,-11.74) --
	( 47.48,-11.74) --
	( 44.66,-11.74) --
	( 41.84,-11.74) --
	( 39.02,-11.74) --
	( 36.20,-11.74) --
	( 33.38,-11.74) --
	( 30.56,-11.74) --
	( 27.74,-11.74) --
	( 24.92,-11.74) --
	( 22.10,-11.74) --
	( 19.28,-11.74) --
	( 16.46,-11.74) --
	( 13.64,-11.74) --
	( 10.82,-11.74) --
	(  8.00,-11.74) --
	(  5.18,-11.74) --
	(  2.36,-11.74) --
	( -0.46,-11.74) --
	( -3.28,-11.74) --
	( -6.10,-11.74) --
	( -8.92,-11.74) --
	(-11.74,-11.74);
\node[text=fillColor,inner sep=0pt,anchor=west, outer sep=0pt, scale=  2.5] at ( 80.00,215.00) {$U_1=0$};
 \node[text=MyOrange,inner sep=0pt,anchor=west, outer sep=0pt, scale=  2.5] at ( 80.00,175.00) {$U_2=0$};
\end{scope}
\end{tikzpicture}
\begin{tikzpicture}[x=1pt,y=1pt]
\definecolor{fillColor}{RGB}{255,255,255}
\definecolor{MyOrange}{rgb}{0.8, 0.33, 0}
\path[use as bounding box,fill=fillColor,fill opacity=0.00] (0,0) rectangle (252.94,252.94);
\begin{scope}
\path[clip] (  0.00,  0.00) rectangle (252.94,252.94);
\definecolor{fillColor}{RGB}{69,33,84}

\path[fill=fillColor] (-11.74,238.12) --
	( -8.92,232.54) --
	( -6.10,227.07) --
	( -3.28,221.72) --
	( -0.46,216.48) --
	(  2.36,211.35) --
	(  5.18,206.34) --
	(  8.00,201.44) --
	( 10.82,196.65) --
	( 13.64,191.98) --
	( 16.46,187.42) --
	( 19.28,182.98) --
	( 22.10,178.65) --
	( 24.92,174.43) --
	( 27.74,170.33) --
	( 30.56,166.34) --
	( 33.38,162.47) --
	( 36.20,158.71) --
	( 39.02,155.06) --
	( 41.84,151.53) --
	( 44.66,148.11) --
	( 47.48,144.81) --
	( 50.30,141.62) --
	( 53.12,138.54) --
	( 55.94,135.58) --
	( 58.76,132.73) --
	( 61.58,130.00) --
	( 64.40,127.37) --
	( 67.22,124.87) --
	( 70.04,122.48) --
	( 72.86,120.20) --
	( 75.68,118.03) --
	( 78.50,115.98) --
	( 81.32,114.04) --
	( 84.14,112.22) --
	( 86.96,110.51) --
	( 89.78,108.92) --
	( 92.60,107.43) --
	( 95.42,106.07) --
	( 98.24,104.81) --
	(101.06,103.67) --
	(103.88,102.65) --
	(106.70,101.74) --
	(109.52,100.94) --
	(112.34,100.26) --
	(115.16, 99.69) --
	(117.98, 99.23) --
	(120.80, 98.89) --
	(123.62, 98.66) --
	(126.44, 98.55) --
	(129.26, 98.55) --
	(132.08, 98.66) --
	(134.90, 98.89) --
	(137.72, 99.23) --
	(140.54, 99.69) --
	(143.36,100.26) --
	(146.18,100.94) --
	(149.00,101.74) --
	(151.82,102.65) --
	(154.64,103.67) --
	(157.46,104.81) --
	(160.28,106.07) --
	(163.10,107.43) --
	(165.92,108.92) --
	(168.74,110.51) --
	(171.56,112.22) --
	(174.38,114.04) --
	(177.20,115.98) --
	(180.02,118.03) --
	(182.84,120.20) --
	(185.66,122.48) --
	(188.48,124.87) --
	(191.30,127.37) --
	(194.12,130.00) --
	(196.94,132.73) --
	(199.76,135.58) --
	(202.58,138.54) --
	(205.40,141.62) --
	(208.22,144.81) --
	(211.04,148.11) --
	(213.86,151.53) --
	(216.68,155.06) --
	(219.50,158.71) --
	(222.32,162.47) --
	(225.14,166.34) --
	(227.96,170.33) --
	(230.78,174.43) --
	(233.60,178.65) --
	(236.42,182.98) --
	(239.24,187.42) --
	(242.06,191.98) --
	(244.88,196.65) --
	(247.70,201.44) --
	(250.52,206.34) --
	(253.34,211.35) --
	(256.16,216.48) --
	(258.98,221.72) --
	(261.80,227.07) --
	(264.62,232.54) --
	(267.44,238.12) --
	(267.44,-11.74) --
	(264.62,-11.74) --
	(261.80,-11.74) --
	(258.98,-11.74) --
	(256.16,-11.74) --
	(253.34,-11.74) --
	(250.52,-11.74) --
	(247.70,-11.74) --
	(244.88,-11.74) --
	(242.06,-11.74) --
	(239.24,-11.74) --
	(236.42,-11.74) --
	(233.60,-11.74) --
	(230.78,-11.74) --
	(227.96,-11.74) --
	(225.14,-11.74) --
	(222.32,-11.74) --
	(219.50,-11.74) --
	(216.68,-11.74) --
	(213.86,-11.74) --
	(211.04,-11.74) --
	(208.22,-11.74) --
	(205.40,-11.74) --
	(202.58,-11.74) --
	(199.76,-11.74) --
	(196.94,-11.74) --
	(194.12,-11.74) --
	(191.30,-11.74) --
	(188.48,-11.74) --
	(185.66,-11.74) --
	(182.84,-11.74) --
	(180.02,-11.74) --
	(177.20,-11.74) --
	(174.38,-11.74) --
	(171.56,-11.74) --
	(168.74,-11.74) --
	(165.92,-11.74) --
	(163.10,-11.74) --
	(160.28,-11.74) --
	(157.46,-11.74) --
	(154.64,-11.74) --
	(151.82,-11.74) --
	(149.00,-11.74) --
	(146.18,-11.74) --
	(143.36,-11.74) --
	(140.54,-11.74) --
	(137.72,-11.74) --
	(134.90,-11.74) --
	(132.08,-11.74) --
	(129.26,-11.74) --
	(126.44,-11.74) --
	(123.62,-11.74) --
	(120.80,-11.74) --
	(117.98,-11.74) --
	(115.16,-11.74) --
	(112.34,-11.74) --
	(109.52,-11.74) --
	(106.70,-11.74) --
	(103.88,-11.74) --
	(101.06,-11.74) --
	( 98.24,-11.74) --
	( 95.42,-11.74) --
	( 92.60,-11.74) --
	( 89.78,-11.74) --
	( 86.96,-11.74) --
	( 84.14,-11.74) --
	( 81.32,-11.74) --
	( 78.50,-11.74) --
	( 75.68,-11.74) --
	( 72.86,-11.74) --
	( 70.04,-11.74) --
	( 67.22,-11.74) --
	( 64.40,-11.74) --
	( 61.58,-11.74) --
	( 58.76,-11.74) --
	( 55.94,-11.74) --
	( 53.12,-11.74) --
	( 50.30,-11.74) --
	( 47.48,-11.74) --
	( 44.66,-11.74) --
	( 41.84,-11.74) --
	( 39.02,-11.74) --
	( 36.20,-11.74) --
	( 33.38,-11.74) --
	( 30.56,-11.74) --
	( 27.74,-11.74) --
	( 24.92,-11.74) --
	( 22.10,-11.74) --
	( 19.28,-11.74) --
	( 16.46,-11.74) --
	( 13.64,-11.74) --
	( 10.82,-11.74) --
	(  8.00,-11.74) --
	(  5.18,-11.74) --
	(  2.36,-11.74) --
	( -0.46,-11.74) --
	( -3.28,-11.74) --
	( -6.10,-11.74) --
	( -8.92,-11.74) --
	(-11.74,-11.74) --
	cycle;

\path[] (-11.74,238.12) --
	( -8.92,232.54) --
	( -6.10,227.07) --
	( -3.28,221.72) --
	( -0.46,216.48) --
	(  2.36,211.35) --
	(  5.18,206.34) --
	(  8.00,201.44) --
	( 10.82,196.65) --
	( 13.64,191.98) --
	( 16.46,187.42) --
	( 19.28,182.98) --
	( 22.10,178.65) --
	( 24.92,174.43) --
	( 27.74,170.33) --
	( 30.56,166.34) --
	( 33.38,162.47) --
	( 36.20,158.71) --
	( 39.02,155.06) --
	( 41.84,151.53) --
	( 44.66,148.11) --
	( 47.48,144.81) --
	( 50.30,141.62) --
	( 53.12,138.54) --
	( 55.94,135.58) --
	( 58.76,132.73) --
	( 61.58,130.00) --
	( 64.40,127.37) --
	( 67.22,124.87) --
	( 70.04,122.48) --
	( 72.86,120.20) --
	( 75.68,118.03) --
	( 78.50,115.98) --
	( 81.32,114.04) --
	( 84.14,112.22) --
	( 86.96,110.51) --
	( 89.78,108.92) --
	( 92.60,107.43) --
	( 95.42,106.07) --
	( 98.24,104.81) --
	(101.06,103.67) --
	(103.88,102.65) --
	(106.70,101.74) --
	(109.52,100.94) --
	(112.34,100.26) --
	(115.16, 99.69) --
	(117.98, 99.23) --
	(120.80, 98.89) --
	(123.62, 98.66) --
	(126.44, 98.55) --
	(129.26, 98.55) --
	(132.08, 98.66) --
	(134.90, 98.89) --
	(137.72, 99.23) --
	(140.54, 99.69) --
	(143.36,100.26) --
	(146.18,100.94) --
	(149.00,101.74) --
	(151.82,102.65) --
	(154.64,103.67) --
	(157.46,104.81) --
	(160.28,106.07) --
	(163.10,107.43) --
	(165.92,108.92) --
	(168.74,110.51) --
	(171.56,112.22) --
	(174.38,114.04) --
	(177.20,115.98) --
	(180.02,118.03) --
	(182.84,120.20) --
	(185.66,122.48) --
	(188.48,124.87) --
	(191.30,127.37) --
	(194.12,130.00) --
	(196.94,132.73) --
	(199.76,135.58) --
	(202.58,138.54) --
	(205.40,141.62) --
	(208.22,144.81) --
	(211.04,148.11) --
	(213.86,151.53) --
	(216.68,155.06) --
	(219.50,158.71) --
	(222.32,162.47) --
	(225.14,166.34) --
	(227.96,170.33) --
	(230.78,174.43) --
	(233.60,178.65) --
	(236.42,182.98) --
	(239.24,187.42) --
	(242.06,191.98) --
	(244.88,196.65) --
	(247.70,201.44) --
	(250.52,206.34) --
	(253.34,211.35) --
	(256.16,216.48) --
	(258.98,221.72) --
	(261.80,227.07) --
	(264.62,232.54) --
	(267.44,238.12);

\path[] (267.44,-11.74) --
	(264.62,-11.74) --
	(261.80,-11.74) --
	(258.98,-11.74) --
	(256.16,-11.74) --
	(253.34,-11.74) --
	(250.52,-11.74) --
	(247.70,-11.74) --
	(244.88,-11.74) --
	(242.06,-11.74) --
	(239.24,-11.74) --
	(236.42,-11.74) --
	(233.60,-11.74) --
	(230.78,-11.74) --
	(227.96,-11.74) --
	(225.14,-11.74) --
	(222.32,-11.74) --
	(219.50,-11.74) --
	(216.68,-11.74) --
	(213.86,-11.74) --
	(211.04,-11.74) --
	(208.22,-11.74) --
	(205.40,-11.74) --
	(202.58,-11.74) --
	(199.76,-11.74) --
	(196.94,-11.74) --
	(194.12,-11.74) --
	(191.30,-11.74) --
	(188.48,-11.74) --
	(185.66,-11.74) --
	(182.84,-11.74) --
	(180.02,-11.74) --
	(177.20,-11.74) --
	(174.38,-11.74) --
	(171.56,-11.74) --
	(168.74,-11.74) --
	(165.92,-11.74) --
	(163.10,-11.74) --
	(160.28,-11.74) --
	(157.46,-11.74) --
	(154.64,-11.74) --
	(151.82,-11.74) --
	(149.00,-11.74) --
	(146.18,-11.74) --
	(143.36,-11.74) --
	(140.54,-11.74) --
	(137.72,-11.74) --
	(134.90,-11.74) --
	(132.08,-11.74) --
	(129.26,-11.74) --
	(126.44,-11.74) --
	(123.62,-11.74) --
	(120.80,-11.74) --
	(117.98,-11.74) --
	(115.16,-11.74) --
	(112.34,-11.74) --
	(109.52,-11.74) --
	(106.70,-11.74) --
	(103.88,-11.74) --
	(101.06,-11.74) --
	( 98.24,-11.74) --
	( 95.42,-11.74) --
	( 92.60,-11.74) --
	( 89.78,-11.74) --
	( 86.96,-11.74) --
	( 84.14,-11.74) --
	( 81.32,-11.74) --
	( 78.50,-11.74) --
	( 75.68,-11.74) --
	( 72.86,-11.74) --
	( 70.04,-11.74) --
	( 67.22,-11.74) --
	( 64.40,-11.74) --
	( 61.58,-11.74) --
	( 58.76,-11.74) --
	( 55.94,-11.74) --
	( 53.12,-11.74) --
	( 50.30,-11.74) --
	( 47.48,-11.74) --
	( 44.66,-11.74) --
	( 41.84,-11.74) --
	( 39.02,-11.74) --
	( 36.20,-11.74) --
	( 33.38,-11.74) --
	( 30.56,-11.74) --
	( 27.74,-11.74) --
	( 24.92,-11.74) --
	( 22.10,-11.74) --
	( 19.28,-11.74) --
	( 16.46,-11.74) --
	( 13.64,-11.74) --
	( 10.82,-11.74) --
	(  8.00,-11.74) --
	(  5.18,-11.74) --
	(  2.36,-11.74) --
	( -0.46,-11.74) --
	( -3.28,-11.74) --
	( -6.10,-11.74) --
	( -8.92,-11.74) --
	(-11.74,-11.74);
 \node[text=fillColor,inner sep=0pt,anchor=west, outer sep=0pt, scale=  2.5] at ( 80.00,215.00) {$U_1=0$};
 \node[text=MyOrange,inner sep=0pt,anchor=west, outer sep=0pt, scale=  2.5] at ( 80.00,175.00) {$U_2=2.5$};
\end{scope}
\end{tikzpicture}
\begin{tikzpicture}[x=1pt,y=1pt]
\definecolor{fillColor}{RGB}{255,255,255}
\definecolor{MyOrange}{rgb}{0.8, 0.33, 0}
\path[use as bounding box,fill=fillColor,fill opacity=0.00] (0,0) rectangle (252.94,252.94);
\begin{scope}
\path[clip] (  0.00,  0.00) rectangle (252.94,252.94);
\definecolor{fillColor}{RGB}{69,33,84}

\path[fill=fillColor] (-11.74,133.43) --
	( -8.92,127.85) --
	( -6.10,122.38) --
	( -3.28,117.02) --
	( -0.46,111.78) --
	(  2.36,106.66) --
	(  5.18,101.64) --
	(  8.00, 96.74) --
	( 10.82, 91.96) --
	( 13.64, 87.29) --
	( 16.46, 82.73) --
	( 19.28, 78.28) --
	( 22.10, 73.95) --
	( 24.92, 69.74) --
	( 27.74, 65.64) --
	( 30.56, 61.65) --
	( 33.38, 57.77) --
	( 36.20, 54.01) --
	( 39.02, 50.37) --
	( 41.84, 46.84) --
	( 44.66, 43.42) --
	( 47.48, 40.11) --
	( 50.30, 36.92) --
	( 53.12, 33.85) --
	( 55.94, 30.88) --
	( 58.76, 28.04) --
	( 61.58, 25.30) --
	( 64.40, 22.68) --
	( 67.22, 20.17) --
	( 70.04, 17.78) --
	( 72.86, 15.50) --
	( 75.68, 13.34) --
	( 78.50, 11.29) --
	( 81.32,  9.35) --
	( 84.14,  7.53) --
	( 86.96,  5.82) --
	( 89.78,  4.22) --
	( 92.60,  2.74) --
	( 95.42,  1.37) --
	( 98.24,  0.12) --
	(101.06, -1.02) --
	(103.88, -2.04) --
	(106.70, -2.96) --
	(109.52, -3.75) --
	(112.34, -4.44) --
	(115.16, -5.01) --
	(117.98, -5.46) --
	(120.80, -5.80) --
	(123.62, -6.03) --
	(126.44, -6.15) --
	(129.26, -6.15) --
	(132.08, -6.03) --
	(134.90, -5.80) --
	(137.72, -5.46) --
	(140.54, -5.01) --
	(143.36, -4.44) --
	(146.18, -3.75) --
	(149.00, -2.96) --
	(151.82, -2.04) --
	(154.64, -1.02) --
	(157.46,  0.12) --
	(160.28,  1.37) --
	(163.10,  2.74) --
	(165.92,  4.22) --
	(168.74,  5.82) --
	(171.56,  7.53) --
	(174.38,  9.35) --
	(177.20, 11.29) --
	(180.02, 13.34) --
	(182.84, 15.50) --
	(185.66, 17.78) --
	(188.48, 20.17) --
	(191.30, 22.68) --
	(194.12, 25.30) --
	(196.94, 28.04) --
	(199.76, 30.88) --
	(202.58, 33.85) --
	(205.40, 36.92) --
	(208.22, 40.11) --
	(211.04, 43.42) --
	(213.86, 46.84) --
	(216.68, 50.37) --
	(219.50, 54.01) --
	(222.32, 57.77) --
	(225.14, 61.65) --
	(227.96, 65.64) --
	(230.78, 69.74) --
	(233.60, 73.95) --
	(236.42, 78.28) --
	(239.24, 82.73) --
	(242.06, 87.29) --
	(244.88, 91.96) --
	(247.70, 96.74) --
	(250.52,101.64) --
	(253.34,106.66) --
	(256.16,111.78) --
	(258.98,117.02) --
	(261.80,122.38) --
	(264.62,127.85) --
	(267.44,133.43) --
	(267.44,-11.74) --
	(264.62,-11.74) --
	(261.80,-11.74) --
	(258.98,-11.74) --
	(256.16,-11.74) --
	(253.34,-11.74) --
	(250.52,-11.74) --
	(247.70,-11.74) --
	(244.88,-11.74) --
	(242.06,-11.74) --
	(239.24,-11.74) --
	(236.42,-11.74) --
	(233.60,-11.74) --
	(230.78,-11.74) --
	(227.96,-11.74) --
	(225.14,-11.74) --
	(222.32,-11.74) --
	(219.50,-11.74) --
	(216.68,-11.74) --
	(213.86,-11.74) --
	(211.04,-11.74) --
	(208.22,-11.74) --
	(205.40,-11.74) --
	(202.58,-11.74) --
	(199.76,-11.74) --
	(196.94,-11.74) --
	(194.12,-11.74) --
	(191.30,-11.74) --
	(188.48,-11.74) --
	(185.66,-11.74) --
	(182.84,-11.74) --
	(180.02,-11.74) --
	(177.20,-11.74) --
	(174.38,-11.74) --
	(171.56,-11.74) --
	(168.74,-11.74) --
	(165.92,-11.74) --
	(163.10,-11.74) --
	(160.28,-11.74) --
	(157.46,-11.74) --
	(154.64,-11.74) --
	(151.82,-11.74) --
	(149.00,-11.74) --
	(146.18,-11.74) --
	(143.36,-11.74) --
	(140.54,-11.74) --
	(137.72,-11.74) --
	(134.90,-11.74) --
	(132.08,-11.74) --
	(129.26,-11.74) --
	(126.44,-11.74) --
	(123.62,-11.74) --
	(120.80,-11.74) --
	(117.98,-11.74) --
	(115.16,-11.74) --
	(112.34,-11.74) --
	(109.52,-11.74) --
	(106.70,-11.74) --
	(103.88,-11.74) --
	(101.06,-11.74) --
	( 98.24,-11.74) --
	( 95.42,-11.74) --
	( 92.60,-11.74) --
	( 89.78,-11.74) --
	( 86.96,-11.74) --
	( 84.14,-11.74) --
	( 81.32,-11.74) --
	( 78.50,-11.74) --
	( 75.68,-11.74) --
	( 72.86,-11.74) --
	( 70.04,-11.74) --
	( 67.22,-11.74) --
	( 64.40,-11.74) --
	( 61.58,-11.74) --
	( 58.76,-11.74) --
	( 55.94,-11.74) --
	( 53.12,-11.74) --
	( 50.30,-11.74) --
	( 47.48,-11.74) --
	( 44.66,-11.74) --
	( 41.84,-11.74) --
	( 39.02,-11.74) --
	( 36.20,-11.74) --
	( 33.38,-11.74) --
	( 30.56,-11.74) --
	( 27.74,-11.74) --
	( 24.92,-11.74) --
	( 22.10,-11.74) --
	( 19.28,-11.74) --
	( 16.46,-11.74) --
	( 13.64,-11.74) --
	( 10.82,-11.74) --
	(  8.00,-11.74) --
	(  5.18,-11.74) --
	(  2.36,-11.74) --
	( -0.46,-11.74) --
	( -3.28,-11.74) --
	( -6.10,-11.74) --
	( -8.92,-11.74) --
	(-11.74,-11.74) --
	cycle;

\path[] (-11.74,133.43) --
	( -8.92,127.85) --
	( -6.10,122.38) --
	( -3.28,117.02) --
	( -0.46,111.78) --
	(  2.36,106.66) --
	(  5.18,101.64) --
	(  8.00, 96.74) --
	( 10.82, 91.96) --
	( 13.64, 87.29) --
	( 16.46, 82.73) --
	( 19.28, 78.28) --
	( 22.10, 73.95) --
	( 24.92, 69.74) --
	( 27.74, 65.64) --
	( 30.56, 61.65) --
	( 33.38, 57.77) --
	( 36.20, 54.01) --
	( 39.02, 50.37) --
	( 41.84, 46.84) --
	( 44.66, 43.42) --
	( 47.48, 40.11) --
	( 50.30, 36.92) --
	( 53.12, 33.85) --
	( 55.94, 30.88) --
	( 58.76, 28.04) --
	( 61.58, 25.30) --
	( 64.40, 22.68) --
	( 67.22, 20.17) --
	( 70.04, 17.78) --
	( 72.86, 15.50) --
	( 75.68, 13.34) --
	( 78.50, 11.29) --
	( 81.32,  9.35) --
	( 84.14,  7.53) --
	( 86.96,  5.82) --
	( 89.78,  4.22) --
	( 92.60,  2.74) --
	( 95.42,  1.37) --
	( 98.24,  0.12) --
	(101.06, -1.02) --
	(103.88, -2.04) --
	(106.70, -2.96) --
	(109.52, -3.75) --
	(112.34, -4.44) --
	(115.16, -5.01) --
	(117.98, -5.46) --
	(120.80, -5.80) --
	(123.62, -6.03) --
	(126.44, -6.15) --
	(129.26, -6.15) --
	(132.08, -6.03) --
	(134.90, -5.80) --
	(137.72, -5.46) --
	(140.54, -5.01) --
	(143.36, -4.44) --
	(146.18, -3.75) --
	(149.00, -2.96) --
	(151.82, -2.04) --
	(154.64, -1.02) --
	(157.46,  0.12) --
	(160.28,  1.37) --
	(163.10,  2.74) --
	(165.92,  4.22) --
	(168.74,  5.82) --
	(171.56,  7.53) --
	(174.38,  9.35) --
	(177.20, 11.29) --
	(180.02, 13.34) --
	(182.84, 15.50) --
	(185.66, 17.78) --
	(188.48, 20.17) --
	(191.30, 22.68) --
	(194.12, 25.30) --
	(196.94, 28.04) --
	(199.76, 30.88) --
	(202.58, 33.85) --
	(205.40, 36.92) --
	(208.22, 40.11) --
	(211.04, 43.42) --
	(213.86, 46.84) --
	(216.68, 50.37) --
	(219.50, 54.01) --
	(222.32, 57.77) --
	(225.14, 61.65) --
	(227.96, 65.64) --
	(230.78, 69.74) --
	(233.60, 73.95) --
	(236.42, 78.28) --
	(239.24, 82.73) --
	(242.06, 87.29) --
	(244.88, 91.96) --
	(247.70, 96.74) --
	(250.52,101.64) --
	(253.34,106.66) --
	(256.16,111.78) --
	(258.98,117.02) --
	(261.80,122.38) --
	(264.62,127.85) --
	(267.44,133.43);

\path[] (267.44,-11.74) --
	(264.62,-11.74) --
	(261.80,-11.74) --
	(258.98,-11.74) --
	(256.16,-11.74) --
	(253.34,-11.74) --
	(250.52,-11.74) --
	(247.70,-11.74) --
	(244.88,-11.74) --
	(242.06,-11.74) --
	(239.24,-11.74) --
	(236.42,-11.74) --
	(233.60,-11.74) --
	(230.78,-11.74) --
	(227.96,-11.74) --
	(225.14,-11.74) --
	(222.32,-11.74) --
	(219.50,-11.74) --
	(216.68,-11.74) --
	(213.86,-11.74) --
	(211.04,-11.74) --
	(208.22,-11.74) --
	(205.40,-11.74) --
	(202.58,-11.74) --
	(199.76,-11.74) --
	(196.94,-11.74) --
	(194.12,-11.74) --
	(191.30,-11.74) --
	(188.48,-11.74) --
	(185.66,-11.74) --
	(182.84,-11.74) --
	(180.02,-11.74) --
	(177.20,-11.74) --
	(174.38,-11.74) --
	(171.56,-11.74) --
	(168.74,-11.74) --
	(165.92,-11.74) --
	(163.10,-11.74) --
	(160.28,-11.74) --
	(157.46,-11.74) --
	(154.64,-11.74) --
	(151.82,-11.74) --
	(149.00,-11.74) --
	(146.18,-11.74) --
	(143.36,-11.74) --
	(140.54,-11.74) --
	(137.72,-11.74) --
	(134.90,-11.74) --
	(132.08,-11.74) --
	(129.26,-11.74) --
	(126.44,-11.74) --
	(123.62,-11.74) --
	(120.80,-11.74) --
	(117.98,-11.74) --
	(115.16,-11.74) --
	(112.34,-11.74) --
	(109.52,-11.74) --
	(106.70,-11.74) --
	(103.88,-11.74) --
	(101.06,-11.74) --
	( 98.24,-11.74) --
	( 95.42,-11.74) --
	( 92.60,-11.74) --
	( 89.78,-11.74) --
	( 86.96,-11.74) --
	( 84.14,-11.74) --
	( 81.32,-11.74) --
	( 78.50,-11.74) --
	( 75.68,-11.74) --
	( 72.86,-11.74) --
	( 70.04,-11.74) --
	( 67.22,-11.74) --
	( 64.40,-11.74) --
	( 61.58,-11.74) --
	( 58.76,-11.74) --
	( 55.94,-11.74) --
	( 53.12,-11.74) --
	( 50.30,-11.74) --
	( 47.48,-11.74) --
	( 44.66,-11.74) --
	( 41.84,-11.74) --
	( 39.02,-11.74) --
	( 36.20,-11.74) --
	( 33.38,-11.74) --
	( 30.56,-11.74) --
	( 27.74,-11.74) --
	( 24.92,-11.74) --
	( 22.10,-11.74) --
	( 19.28,-11.74) --
	( 16.46,-11.74) --
	( 13.64,-11.74) --
	( 10.82,-11.74) --
	(  8.00,-11.74) --
	(  5.18,-11.74) --
	(  2.36,-11.74) --
	( -0.46,-11.74) --
	( -3.28,-11.74) --
	( -6.10,-11.74) --
	( -8.92,-11.74) --
	(-11.74,-11.74);
\node[text=fillColor,inner sep=0pt,anchor=west, outer sep=0pt, scale=  2.5] at ( 80.00,215.00) {$U_1=0$};
 \node[text=MyOrange,inner sep=0pt,anchor=west, outer sep=0pt, scale=  2.5] at ( 80.00,175.00) {$U_2=5$};
\end{scope}
\end{tikzpicture}}
\caption{Excursion set of the constraint $g \leq 0$ for $U_1 \in \{-5,-2.5,0,2.5,5\}$ and $U_2=0$ (first row) and for $U_1 =0$ and $U_2 \in \{-5,-2.5,0,2.5,5\}$. (second row)}
\label{visualtestreda}
\end{figure}%
Visual conjecture (Figure \ref{visualtestreda}) shows that $u_2$ seems to induce more changes in $\Gamma_g$ than $u_1$. To verify this assumption, we compute the p-values and the indices $\hat{\hat{S}}^{\operatorname{H}_{set}}_{i}$, $\hat{\hat{S}}^{\operatorname{H}_{set}}_{T_i}$ for each input $U_i \sim \Udom([-5,5])$ for $i \in \{1,2,3\}$, where $U_3$ is an additional dummy input that does not appear in the function $g$. We first use $n=m=100$ and then $n=m=1000$. The hyperparameter $\sigma^2$ of the kernel $k_{set}$ is chosen to be equal to the empirical mean of $\lambda(\Gamma_g^{(i)}\Delta \Gamma_g^{(^j)})$ with $i>j$. We compute the indices for five characteristic ANOVA kernels. The first is the Sobolev kernel $K_{Sob}$ of order $1$, defined by:
\begin{equation*}
\label{Sobolevkernel}
    K_{Sob}(x,y) = 1+(x-\frac{1}{2})(y-\frac{1}{2})+\frac{1}{2}[(x-y)^2-|x-y| +\frac{1}{6}]. 
\end{equation*}
The four others are obtained using the transformation given in (\ref{kernelanovatransfo}) based on four classic kernels: the Gaussian kernel, the Laplace kernel and the Matérn $3/2$ and $5/2$.
For these kernels, the previous transformation is known analytically and can be found in the Appendix of \textcite{anova_kernel}).
The hyper-parameter is taken to be equal to the empirical standard deviation of the inputs. Since the inputs are not uniformly distributed on $[0,1]$, we apply the inverse of the cumulative distribution function, as suggested in \textcite{daveiga:hal-03108628}. We repeat the estimation 20 times over 20 independent samples. The results are shown  in Figure \ref{fig:reda100} for $n=m=100$ and in Figure \ref{fig:reda1000} for $n=m=1000$. The acceptance rates of the independence hypothesis are given in the table on the left side of the figure, and the variability of the indices is visualized in the boxplots on the right side.

\begin{figure}[ht]
  \begin{minipage}{0.3\textwidth} 
    \centering
    \begin{tabular}{cccc}
\toprule
Kernel & $U_1$ & $U_2$ & $U_3$ \\
\midrule
$K_{Sob}$ & $0$ & $0$ & $95$ \\
$K_{gauss}$ & $0$ & $0$ & $95$ \\
$K_{exp}$ & $0$ & $0$ & $95$ \\
$K_{3/2}$ & $0$ & $0$ & $95$ \\
$K_{5/2}$ & $0$ & $0$ & $90$ \\
\bottomrule
\end{tabular}
    \subcaption{Acceptance rates ($\%$) over $20$ independence tests with a risk of $5\%$}
    \label{tab:acceptance-rate-reda-100}
  \end{minipage}%
  \begin{minipage}{0.7\textwidth} 
    \centering
    \scalebox{0.9}{\input{Images/reda/hsic_anova_3_100_no_pvalue}}
    \subcaption{Estimations of $\hat{\hat{S}}^{\operatorname{H}_{set}}_{i}$ and $\hat{\hat{S}}^{\operatorname{H}_{set}}_{T_i}$}
    \label{fig:indices:reda:100}
  \end{minipage}
  \caption{Acceptance rates (\subref{tab:acceptance-rate-reda-100}) and estimation of $\hat{\hat{S}}^{\operatorname{H}_{set}}_{i}$ and $\hat{\hat{S}}^{\operatorname{H}_{set}}_{T_i}$ (\subref{fig:indices:reda:100})  for the excursion set $\Gamma_g$ computed for 5 kernels with $n=100$, $m=100$ and repeated 20 times}
  \label{fig:reda100}
\end{figure}

\begin{figure}[ht]
  \begin{minipage}{0.3\textwidth} 
    \centering
    \begin{tabular}{cccc}
\toprule
Kernel & $U_1$ & $U_2$ & $U_3$ \\
\midrule
$K_{Sob}$ & $0$ & $0$ & $95$ \\
$K_{gauss}$ & $0$ & $0$ & $95$ \\
$K_{exp}$ & $0$ & $0$ & $100$ \\
$K_{3/2}$ & $0$ & $0$ & $95$ \\
$K_{5/2}$ & $0$ & $0$ & $95$ \\
\bottomrule
\end{tabular}
    \subcaption{Acceptance rates ($\%$) over $20$ independence tests with a risk of $5\%$}
    \label{tab:acceptance-rate-reda-1000}
  \end{minipage}%
  \begin{minipage}{0.7\textwidth} 
    \centering
    \scalebox{0.9}{\input{Images/reda/hsic_anova_3_1000_no_pvalue}}
    \subcaption{Estimations of $\hat{\hat{S}}^{\operatorname{H}_{set}}_{i}$ and $\hat{\hat{S}}^{\operatorname{H}_{set}}_{T_i}$}
    \label{fig:indices:reda:1000}
  \end{minipage}
  \caption{Acceptance rates (\subref{tab:acceptance-rate-reda-1000}) and estimation of $\hat{\hat{S}}^{\operatorname{H}_{set}}_{i}$ and $\hat{\hat{S}}^{\operatorname{H}_{set}}_{T_i}$ (\subref{fig:indices:reda:1000})  for the excursion set $\Gamma_g$ computed for 5 kernels with $n=1000$, $m=1000$ and repeated 20 times}
  \label{fig:reda1000}
\end{figure}

With these results, we classify the inputs as influential or negligible (screening) and rank them by influence (ranking).
$100 \%$ of the p-values of $U_1$ and $U_2$ are below the threshold of $0.05$, so they are classified as influential. $U_3$ has a p-value greater than $0.05$ around $95\%$ of the time, so it is correctly identified as negligible by the independence test. The first-order and total-order indices give the expected results that $U_2$ has a greater influence than $U_1$ on the excursion sets. More precisely, it shows that $U_2$ alone explains about $70\%$ of $\operatorname{HSIC}(\bm U,\Gamma_g)$ (depending on the input kernel), while $U_1$ explains about $25\%$. The artificial input $U_3$ is responsible for $0\%$, as expected. The rest is some kind of interaction between $U_1$ and $U_2$. This is confirmed by looking at the total-order indices, whose values are a bit greater than the first-order. However, this part of the interaction is not to be interpreted as in the case of the Sobol' indices. It is actually an open question to interpret what kind of interactions are detected by the HSIC-ANOVA indices. This is studied in \textcite{sarazin:cea-04320711} in the case of Sobolev kernels.
The effect of the input kernel on the HSIC-ANOVA indices is also an open question, and in this toy case we can say that the input kernel has a non-negligible effect on the estimates of the p-values and the indices. However, in this case it does not change the results of the screening and the ranking. It also seems that the Sobolev kernel behaves differently from the other four kernels. 

Taking $n=m=1000$ in Figure \ref{fig:reda1000} reduces the variance of the indices, but was not necessary to screen and rank the inputs. To support this comment, we compute the relative quadratic risk of the estimator $\widehat{\widehat{\operatorname{HSIC}}}(U_1,\Gamma_g)$ defined by:
$$\mathcal R(\widehat{\widehat{\operatorname{HSIC}}}(U_1,\Gamma_g)) = \Esp \left(\frac{\widehat{\widehat{\operatorname{HSIC}}}\left(U_1, \Gamma_g\right) -  \operatorname{HSIC}(U_1, \Gamma_g)}{\operatorname{HSIC}(U_1, \Gamma_g)} \right)^2.  $$
 We also compute the associated upper bound given in Proposition \ref{prop nested} and the upper bound for a classic NMC estimator (with independent $\bm X^{(k)}_{i,j}$ for all $1 \leq i,j \leq n$, see equation \ref{real_nested_estimator}). The "true" value of HSIC, used to compute the quadratic risk, and the constants $\sigma_1$, $\sigma_2$, and $\sigma_3$ are computed for $n=m=3000$. The Sobolev kernel is used as the input kernel. We plot the risk and the two bounds in the Figure \ref{quadratic_risk} from $n=m=30$ to $n=m=500$. 
\begin{figure}
    \centering
    \scalebox{0.83}{\input{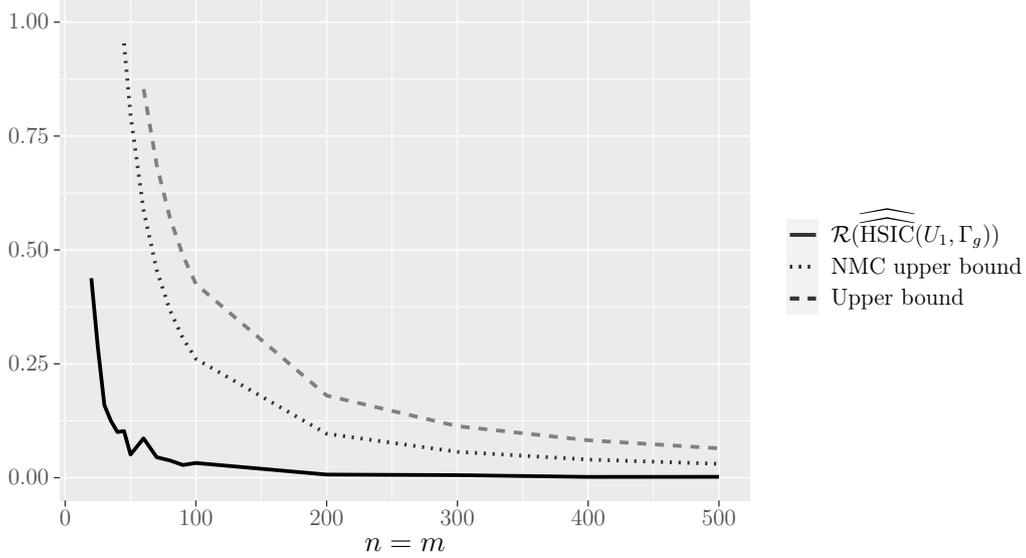}}
    \caption{Evolution of $\mathcal R(\widehat{\widehat{\operatorname{HSIC}}}(U_1,\Gamma_g))$ and of the associated upper bounds for the excursion set $\Gamma_g$}
    \label{quadratic_risk}
\end{figure}%
We observe, as expected, that the quadratic risk of the estimator is below the upper bound of Proposition \ref{prop nested}, but also below the NMC upper bound we would have by taking independent $\bm X^{(k)}_{i,j}$. This example confirms that we have not lost too much in terms of variance of the estimator by reusing the same $\bm X^{(k)}$. It also seems that we do not necessarily need to use high values of $n=m$. Therefore, in the next example, we will only compute the indices for $n=m=100$.
\subsection{Excursion sets of the optimization of an oscillator}
\label{oscill}
In \textcite{cousin_two-step_2022}, an optimization is performed with three probabilistic constraints. We consider the first two constraints, which are defined by the following functions :
\begin{equation*}
g_1(x_1,x_2,u_{1},u_{2},u_{p},u_{r_1},u_{r_2})=
u_{r_1} - \max_{t \in [0,T]} \Ydom'(x_1+u_{1},x_2+u_{2},u_p;t),
\end{equation*}
\begin{equation*}
g_2(x_1,x_2,u_1,u_2,u_{p},u_{r_1},u_{r_2})=
u_{r_2} - \max_{t \in [0,T]} \Ydom^{\prime \prime}(x_1+u_{1},x_2+u_{2},u_p;t),
\end{equation*}
where $\Ydom(x_1+u_{1},x_2+u_{2},u_p;t)$ is the solution of the harmonic oscillator defined by:
\begin{equation*}
(x_1+u_{1})\Ydom^{\prime \prime}(t)+u_p\Ydom'(t)+(x_2+u_{2})\Ydom(t)=\eta(t).
\end{equation*} 
The deterministic input domain is $\Xdom=[1,5] \times [20,50]$. The uncertain input probability distributions are given in Table \ref{tab:uncertainties oscillator}. $U_{r_3}$ is initially a random input associated with a third constraint. In our case it will play the role of a dummy input to check if it is recognized as negligible.
\begin{table}
\centering
\caption{Definition of the uncertain inputs}
\label{tab:uncertainties oscillator}
\begin{tabular}{cccc}
\toprule
Uncertainty & Distribution & Uncertainty & Distribution \\
\midrule
$U_{1}$ & $\mathcal{U}[-0.3, 0.3]$ & $U_{\mathrm{r}_1}$ & $\mathcal{N}(1,0.1^2)$ \\
$U_{2}$ & $\mathcal{U}[-1,1]$ & $U_{\mathrm{r}_2}$ & $\mathcal{N}(2.5,0.25^2)$ \\
$U_p$ & $\mathcal{U}[0.5,1.5]$ & $U_{\mathrm{r}_3}$ & $\mathcal{N}(15,3^2)$ \\
\bottomrule
\end{tabular}
\end{table}
We study the impact of the uncertain inputs on the excursion sets $\Gamma_{g_1}$ associated with the constraint $g_1 \leq 0$ and on the excursion sets $\Gamma_{g_2}$ associated with $g_2 \leq 0$. Since kernel-based methods are appropriate for vectorial outputs, we also consider the case where an output is the pair of the two excursion sets $(\Gamma_{g_1},\Gamma_{g_2})$, each associated with a constraint. For each case, we compute $\hat{\hat{S}}^{\operatorname{H}_{set}}_{i}$, $\hat{\hat{S}}^{\operatorname{H}_{set}}_{T_i}$ and the associated p-values for each uncertain input with $n=100$ and $m=100$. We again use the output kernel $k_{set}$ and $k_{set} \otimes k_{set}$ for the pairs of excursion sets. Note that the kernel $k_{set} \otimes k_{set}$ may not be characteristic. The same $5$ ANOVA kernels are used as in the previous example. We repeat the estimation $20$ times to again obtain the acceptance rates and boxplots of the first-order indices given in Figures \ref{hsic_anova_g1} to \ref{hsic_anova_g1g2couple}. The total-order indices are given in Figure \ref{fig:totalordercousin} in the Appendix \ref{annex:fig}.

\begin{figure}
  \begin{minipage}{0.52\textwidth} 
    \centering
    \begin{tabular}{lcccccc}
\toprule
Kernel& $U_1$ & $U_2$ & $U_p$ & $U_{r_1}$ & $U_{r_2}$ & $U_{r_3}$ \\
\midrule
$K_{Sob}$ & $15$ & $95$ & $55$ & $0$ & $100$ & $90$ \\
$K_{gauss}$ & $15$ & $100$ & $60$ & $0$ & $100$ & $95$ \\
$K_{exp}$ & $25$ & $90$ & $60$ & $0$ & $100$ & $95$ \\
$K_{3/2}$ & $15$ & $90$ & $60$ & $0$ & $100$ & $95$ \\
$K_{5/2}$ & $15$ & $100$ & $60$ & $0$ & $100$ & $95$ \\
\bottomrule
\end{tabular}
    \subcaption{Acceptance rates ($\%$) over $20$ independence tests with a risk of $5\%$}
    \label{tab:acceptance-rate-g1}
  \end{minipage}%
  \begin{minipage}{0.48\textwidth} 
    \centering
    \scalebox{0.9}{\input{Images/oscillator/g1/hsic_anova_3_no_pvalue_only_1_order}}
    \subcaption{Estimations of $\hat{\hat{S}}^{\operatorname{H}_{set}}_{i}$}
    \label{fig:indices:g1}
  \end{minipage}
  \caption{Acceptance rates (\subref{tab:acceptance-rate-g1}) and estimation of $\hat{\hat{S}}^{\operatorname{H}_{set}}_{i}$  (\subref{fig:indices:g1})  for the excursion set $\Gamma_{g_1}$ computed for 5 kernels with $n=100$, $m=100$ and repeated 20 times}
  \label{hsic_anova_g1}
\end{figure}

\begin{figure}
  \begin{minipage}{0.52\textwidth} 
    \centering
    \begin{tabular}{lcccccc}
\toprule
Kernel& $U_1$ & $U_2$ & $U_p$ & $U_{r_1}$ & $U_{r_2}$ & $U_{r_3}$ \\
\midrule
$K_{Sob}$ & $0$ & $100$ & $95$ & $95$ & $0$ & $95$ \\
$K_{gauss}$ & $0$ & $100$ & $95$ & $95$ & $0$ & $95$ \\
$K_{exp}$ & $0$ & $100$ & $95$ & $100$ & $0$ & $95$ \\
$K_{3/2}$ & $0$ & $100$ & $95$ & $100$ & $0$ & $95$ \\
$K_{5/2}$ & $0$ & $100$ & $95$ & $95$ & $0$ & $95$ \\
\bottomrule
\end{tabular}
    \subcaption{Acceptance rates ($\%$) over $20$ independence tests with a risk of $5\%$}
    \label{tab:acceptance-rate-g2}
  \end{minipage}%
  \begin{minipage}{0.48\textwidth} 
    \centering
    \scalebox{0.9}{\input{Images/oscillator/g2/hsic_anova_3_no_pvalue_only_1_order}}
    \subcaption{Estimations of $\hat{\hat{S}}^{\operatorname{H}_{set}}_{i}$}
    \label{fig:indices:g2}
  \end{minipage}
  \caption{Acceptance rates (\subref{tab:acceptance-rate-g2}) and estimation of $\hat{\hat{S}}^{\operatorname{H}_{set}}_{i}$  (\subref{fig:indices:g2})  for the excursion set $\Gamma_{g_2}$ computed for 5 kernels with $n=100$, $m=100$ and repeated 20 times}
  \label{hsic_anova_g2}
\end{figure}

\begin{figure}
  \begin{minipage}{0.52\textwidth} 
    \centering
    \begin{tabular}{ccccccc}
\toprule
Kernel& $U_1$ & $U_2$ & $U_p$ & $U_{r_1}$ & $U_{r_2}$ & $U_{r_3}$ \\
\midrule
$K_{Sob}$ & $0$ & $95$ & $60$ & $0$ & $0$ & $90$ \\
$K_{gauss}$ & $0$ & $100$ & $60$ & $0$ & $0$ & $90$ \\
$K_{exp}$ & $0$ & $95$ & $60$ & $0$ & $0$ & $90$ \\
$K_{3/2}$ & $0$ & $100$ & $60$ & $0$ & $0$ & $90$ \\
$K_{5/2}$ & $0$ & $100$ & $60$ & $0$ & $0$ & $90$ \\
\bottomrule
\end{tabular}
    \subcaption{Acceptance rates ($\%$) over $20$ independence tests with a risk of $5\%$}
    \label{tab:acceptance-rate-g1g2couple}
  \end{minipage}%
  \begin{minipage}{0.48\textwidth} 
    \centering
    \scalebox{0.9}{\input{Images/oscillator/g1_g2_couple/hsic_anova_3_no_pvalue_only_1_order}}
    \subcaption{Estimations of $\hat{\hat{S}}^{\operatorname{H}_{set}}_{i}$}
    \label{fig:indices:g1g2couple}
  \end{minipage}
  \caption{Acceptance rates (\subref{tab:acceptance-rate-g1g2couple}) and estimation of $\hat{\hat{S}}^{\operatorname{H}_{set}}_{i}$  (\subref{fig:indices:g1g2couple}) for the pair of excursion sets
$(\Gamma_{g_1} , \Gamma_{g_2} )$  computed for 5 kernels with $n=100$, $m=100$ and repeated 20 times}
\label{hsic_anova_g1g2couple}
\end{figure}

Acceptance rates tables can be used for screening and the boxplots for ranking. For the random set associated with the constraint $g_1 \leq 0$ (Figure \ref{hsic_anova_g1}), only $U_{r_1}$ is always detected as influential. $U_1$ is also detected as influential most of the time (with an acceptance rate of about $20\%$), and $U_p$ is detected as influential only $40 \%$ of the time. Between these three inputs, $U_{r_1}$ is much more influential, since $\hat{\hat{S}}^{\operatorname{H}_{set}}_{r_1} \approx 80 \%$. Then $U_1$ and $U_p$ are responsible for about $10\%$ and $3\%$ of $\operatorname{HSIC}(\bm U,\Gamma_{g_1})$. For the second constraint $g_2 \leq 0$ (Figure \ref{hsic_anova_g2}), only $U_1$ and $U_{r_2}$ are influential with their first-order index of about $10\%$ and $80\%$, respectively. If we consider the pair of the two random sets (Figure \ref{hsic_anova_g1g2couple}), we get a kind of compromise between the two previous cases: $U_1$, $U_{r_1}$ and $U_{r_2}$ are always recognized as influential. In terms of ranking, $U_{r_1}$ and $U_{r_2}$ have almost the same influence, with their first-order index around $40\%$ and $U_1$ remaining around $10\%$. It is important to note that considering pairs of random sets is completely different from considering the random set associated with the pair of constraints, denoted $\Gamma_{(g_1,g_2)}$, i.e., the intersection of the two sets. The results of the latter are shown in Figure \ref{hsic_anova_g1andg2} in Appendix \ref{annex:fig} and are indeed very different from the results in Figure \ref{hsic_anova_g1g2couple}. In this test case, we observe that the choice of the input kernel has a limited impact on the indices. Regardless of the stage of screening or ranking, the conclusions remain the same for the five kernels. 

\subsection{Sensitivity analysis for robust conception of an electrical machine}

In \autocite{reyes2023study}, a sensitivity analysis is performed in the context of robust conception of an electrical machine. The studied machine is a permanent magnet assisted synchronous reluctance motor which is one of the most used machines nowadays in electrical vehicles. The purpose of the authors is to take into account the uncertainties (manufactoring and assembly tolerances) on geometric and magnetic parameters of the machine components in the optimization of the mean torque (to be maximized) and the torque ripples (to be minimized). The resulting optimization problem is the following
$$
\min _{x \in \Xdom}\left(\Esp\left[f_1(x,\bm U) \right], \Esp\left[f_2(x,\bm U)\right] \right)
$$
where $f_1$ and $f_2$ are two real-valued objective functions defined on $\Xdom \subset \Real^{12}$ (respectively the opposite of the mean torque and the torque ripple) and $\Udom \subset \Real^{14}$. $12$ of the $14$ uncertain inputs are manufacturing tolerances on each $x$, summarized in \autoref{tab:uncertainties_geometric}, where, for example, $\pm 0.1^{\circ}$ means that the slot angle uncertainty $U_1$ follows a uniform law on $[-0.1, 0.1]$. The other two, $U_{13}$ and $U_{14}$, describe the magnetic material properties and follow uniform distributions on $[-1,1]$. 

\begin{figure}
  \begin{minipage}{0.52\textwidth} 
    \centering
    \includegraphics[scale=0.3]{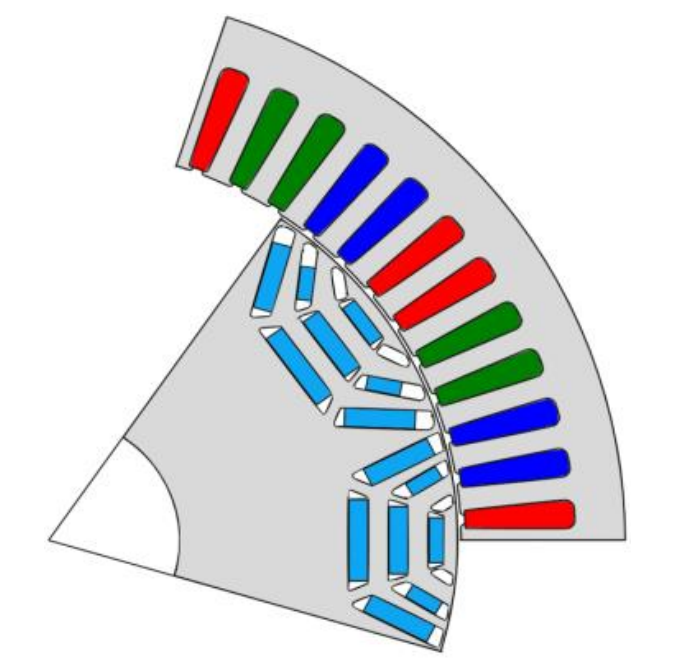}
    \subcaption{}
    \label{fig:geomtry}
  \end{minipage}%
  \begin{minipage}{0.48\textwidth} 
    \centering
    \includegraphics[scale=0.3]{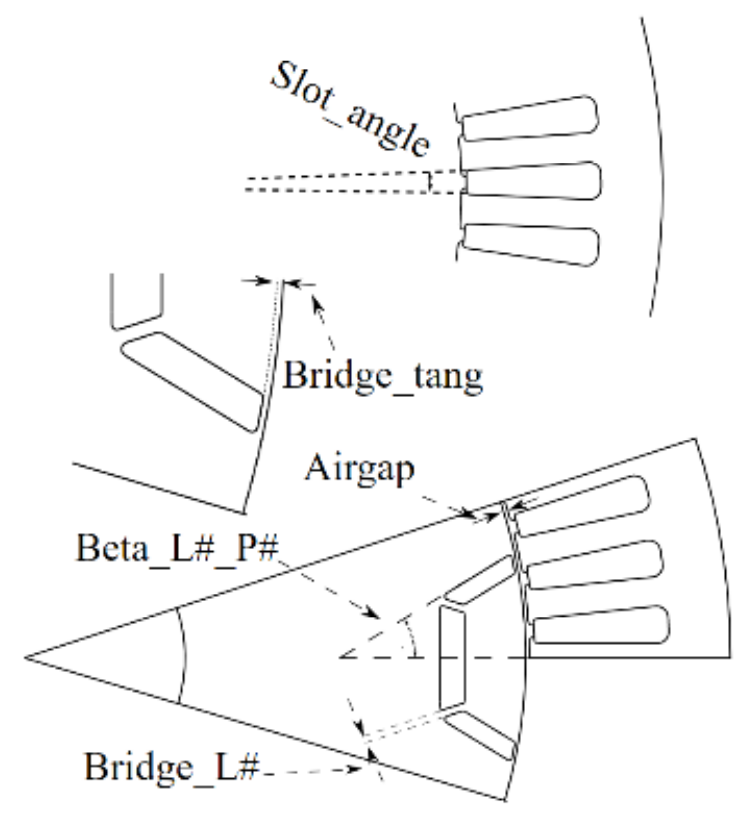}
    \subcaption{}
    \label{fig:design_parameters}
  \end{minipage}
  \caption{(a) Geometry of the permanent magnet assisted synchronous reluctance motor - (b) Design parameters for one layer (\# is the number of layer)}
  \label{fig:description_moteur}
\end{figure}

\begin{table}
\centering
\resizebox{0.7\textwidth}{!}{\begin{tabular}{cccc}
\toprule
\textbf{Input} & \textbf{Lower bound}  & \textbf{Upper bound} & \textbf{Manufacturing} \\
\textbf{parameters} & $x_{\text{min}}$ & $x_{\text{max}}$ & \textbf{Tolerance} $U$\\
\midrule
Slot angle & $2.47^\circ$ & $3.27^\circ$ & $\pm 0.1^\circ$ \\
$\beta_{L1P1}$ & $27.03^\circ$ & $29.66^\circ$ & $\pm 0.33^\circ$ \\
$\beta_{L1P2}$ & $37.03^\circ$ & $39.66^\circ$ & $\pm 0.33^\circ$ \\
$\beta_{L2P1}$ & $31.03^\circ$ & $33.66^\circ$ & $\pm 0.33^\circ$ \\
$\beta_{L2P2}$ & $47.03^\circ$ & $49.66^\circ$ & $\pm 0.33^\circ$ \\
$\beta_{L3P1}$ & $33.7^\circ$ & $37^\circ$ & $\pm 0.33^\circ$ \\
$\beta_{L3P2}$ & $59.7^\circ$ & $63^\circ$ & $\pm 0.33^\circ$ \\
Airgap & $0.55$ mm & $0.65$ mm & $\pm 0.03$ mm \\
Bridge$_{L1}$ & $2.6$ mm & $2.98$ mm & $\pm 0.05$ mm \\
Bridge$_{L2}$ & $0.9$ mm & $1.18$ mm & $\pm 0.05$ mm \\
Bridge$_{L3}$ & $0.5$ mm & $0.62$ mm & $\pm 0.03$ mm \\
Bridge$_{tang}$ & $0.4$ mm & $0.6$ mm & $\pm 0.05$ mm \\
\bottomrule
\end{tabular}}
\caption{Geometrical variables (see \autocite{reyes2023study} for more details)}
\label{tab:uncertainties_geometric}
\end{table}

The purpose of our study is to screen and rank the uncertain parameters according to their impact on the quantities of interest in this optimization.
Since we want to minimize an objective function, we are interested in quantifying the effect of $\bm U$ on sets where the objective function takes low values. This can be done by looking at the random set of the form $\Gamma_{f_1}= \{ x \in \Xdom, f_1(x,\bm U) \leq q_1 \}$, where $q_1$ is a threshold to be selected. It quantifies the effect of $\bm U$ on the sets where $f_1$ is below $q_1$, but does not take into account the variations of $f_1$ within this low-valued region. For this reason, we propose to look at the effect of the uncertain inputs on the set $\Gamma_{F_1}$ defined by $
\Gamma_{F_1} = \{ (\bm x,x_{13}) \in \Xdom \times [f_1^{min}, q_1],  x_{13} \leq f_1(\bm x,\bm U)  \}$
where $f_1^{min}$ is a lower bound  to only consider sets of low values of $f_1$. This also corresponds to the excursion set associated with $F_1: \Xdom \times [f_1^{min} , q_1]\times \Udom \rightarrow \Real$ defined by $F_1(\bm x,\bm u)=x_{13}-f_1(\bm x_{-13},\bm u)$. $\Gamma_{F_2}$ is defined similarly. 

We study three different cases~: sensitivity analysis on $\Gamma_{F_1}$, on $\Gamma_{F_2}$, and on the pair of excursion sets $(\Gamma_{F_1},\Gamma_{F_2})$. We use the thresholds $-q_1=420$ N.m and $q_2=7 \%$ where lies the Pareto front in Figure 9 of \autocite{reyes2023study}.
For these three cases, we take $n=m=100$ and use only the Sobolev kernel. With $20$ replicates, the acceptance rates with a risk of $0.05 \%$ are given in \autoref{tab:pval_adan} and the boxplots of the first-order indices are plotted in \autoref{fig:hsic_anova_adan}. 
\begin{table}
\centering
\resizebox{\textwidth}{!}{
\begin{tabular}{l *{14}{c}}
  \toprule
 & $U_1$ & $U_2$ & $U_3$ & $U_4$ & $U_5$ & $U_6$ & $U_7$ & $U_8$ & $U_9$ & $U_{10}$ & $U_{11}$ & $U_{12}$ & $U_{13}$ & $U_{14}$ \\ 
  \midrule
$\Gamma_{F_1}$ & 0.40 & 0.90 & 0.75 & 0.90 & 0.75 & 1.00 & 1.00 & 0.40 & 0.90 & 0.90 & 0.85 & 1.00 & 0.00 & 0.00 \\ 
$\Gamma_{F_2}$ & 0.70 & 0.00 & 0.00 & 0.75 & 0.00 & 1.00 & 0.65 & 0.15 & 0.95 & 0.70 & 0.90 & 0.60 & 0.00 & 0.40 \\ 
 $(\Gamma_{F_1},\Gamma_{F_2})$ & 0.25 & 0.00 & 0.05 & 0.90 & 0.05 & 1.00 & 0.95 & 0.10 & 0.85 & 0.75 & 0.80 & 0.75 & 0.00 & 0.00 \\ 
   \bottomrule
\end{tabular}}
\caption{Acceptance rates ($\%$) over $20$ independence tests with a risk of $5\%$ for the excursion sets $\Gamma_{F_1}$, $\Gamma_{F_2}$ and the pair $(\Gamma_{F_1},\Gamma_{F_2})$ with $-q_1=420$ N.m and $q_2=7\%$ computed with the Sobolev input kernel and with $n=100$, $m=100$.}
\label{tab:pval_adan}
\end{table}
\begin{figure}
\centering
 \resizebox{\textwidth}{!}{\input{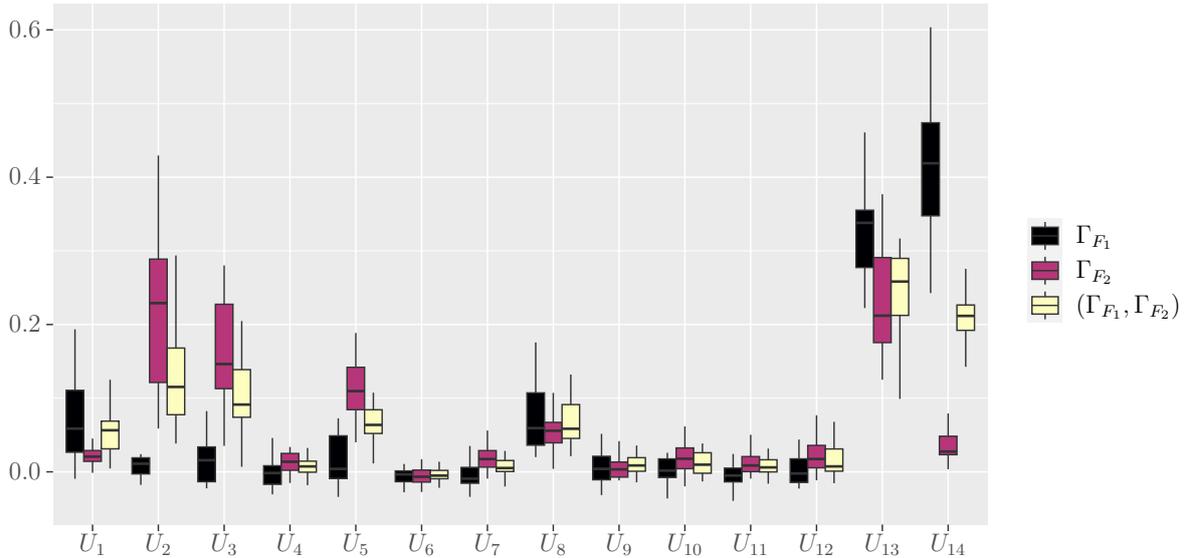}}
\caption{Estimation of $\hat{\hat{S}}^{\operatorname{H}_{set}}_{i}$ for the excursion sets $\Gamma_{F_1}$, $\Gamma_{F_2}$ and the pair $(\Gamma_{F_1},\Gamma_{F_2})$ with $-q_1=420$ N.m and $q_2=7\%$ computed with the Sobolev input kernel and with $n=100$, $m=100$. $20$ replicates.}
\label{fig:hsic_anova_adan}
\end{figure}

\autoref{tab:pval_adan} shows that the inputs influencing the excursion sets $\Gamma_{F_1}$ and $\Gamma_{F_2}$ are different. Only $U_{13}$ and $U_{14}$ are always tested as influential for $\Gamma_{F_1}$ and $U_2$, $U_3$, $U_5$ and $U_{13}$ for $\Gamma_{F_2}$. They correspond to the most influential inputs in \autoref{fig:hsic_anova_adan} and can be ranked. For example, $U_2$ has the most influence on $\Gamma_{F_2}$, followed by $U_{13}$, $U_3$ and $U_5$. Detected about $95\%$ of the time as independent of the output, $U_2$, $U_4$, $U_6$, $U_7$, $U_9$, $U_{10}$, $U_{12}$ for $\Gamma_{F_1}$, and $U_6$, $U_9$, $U_{11}$ for $\Gamma_{F_2}$ could be removed to simplify each model. The remaining inputs have an acceptance rate between $15\%$ and $85\%$, which means that they are sometimes classified as independent and sometimes not. They correspond to inputs that have a small effect on the output, as can be observed in \autoref{fig:hsic_anova_adan}. By taking the pair of the two sets, the influential inputs are the one that were influential on at least one of the sets. This results in five dominant inputs that can be ranked by their first order index: $ U_{13} \succeq U_{14} \succeq U_2  \succeq U_3 \succeq U_5$. The results associated to the excursion sets of the forms $\Gamma_{f_i}= \{ x \in \Xdom, f_i(x,\bm U) \leq q_i \}$ are similar in this case and given in Appendix \ref{annex:fig}.


\section{Conclusion}

In this paper, we propose a method to perform sensitivity analysis on set-valued outputs through kernel-based sensitivity analysis, which relies on the choice of a kernel between sets. We introduce the kernel $k_{set}$, which is based on the symmetric difference between two sets. We show that it is characteristic, which is an essential property for performing screening. We then adapt the recent HSIC-ANOVA index to set-valued outputs and introduce an efficient estimator. Finally, we compute the indices on three test cases including a real application for robust design of electrical motor. The proposed method allows to screen and rank the uncertain inputs according to their impact on the excursion sets. For future research, it could be interesting to find and study other set-valued output kernels. Other types of approaches to perform sensitivity analysis for sets could also be investigated, such as using universal indices from \textcite{GSA_Wass_universal_index}, or by using classic random set theory from \textcite{molchanov_random_2017}.

In the context of robust optimization, the presented method and especially the screening results could be used to reduce the dimension of the uncertain space by quantifying the impact of uncertain inputs on the optimization constraints. Reducing the dimension of the uncertain space can then be useful to reduce the computational cost of a joint space metamodel that could be used within a Bayesian optimization. 

More generally, sensitivity analysis for sets can also be used when dealing with numerical codes with set-valued outputs. This occurs in several areas: in the field of viability, where the outputs are sets called viability kernels, or, for example, in flood risk, where the output is the map of flooded areas.
\if0\blind{
\section{Aknowledgment}
The authors thank Gabriel Sarazin for his numerous fruitful discussions and comments. We are also grateful to the reviewers and the associate editor for their relevant and helpful comments. This research was conducted with the support of the consortium in Applied Mathematics
\href{https://doi.org/10.5281/zenodo.6581216}{\underline{CIROQUO}}, gathering partners in technological research and academia in the development of advanced methods for Computer Experiments.
}\fi
\setlength\bibitemsep{0.5\itemsep}
\printbibliography

\newpage
\section{Appendix}
\subsection{Proofs}
\subsubsection{Proof of Lemma \ref{lemme}}
\label{proof:lemma}
\lemma*
\begin{proof}
    \begin{align*}
        \delta(\gamma_1,\gamma_2)
        &=\lambda(\gamma_1 \Delta \gamma_2)\\
        &= \int_{\Xdom} \mathbbold 1 _{\gamma_1 \Delta \gamma_2} d\lambda \\
        &= \int_{\Xdom} (\mathbbold 1 _{\gamma_1 \cup \gamma_2} -\mathbbold 1 _{\gamma_1 \cap \gamma_2})d\lambda \\
        &= \int_{\Xdom} (\mathbbold 1 _{\gamma_1 } + \mathbbold 1 _{\gamma_2 } -2 \mathbbold 1 _{\gamma_1 \cap \gamma_2})d\lambda \\
        &= \int_{\Xdom} (\mathbbold 1 _{\gamma_1 } - \mathbbold 1 _{\gamma_2 } )^2d\lambda \\
        &= \left|\left|\mathbbold 1_{\gamma_1} - \mathbbold 1_{\gamma_2}\right| \right|^2_{2}
    \end{align*} 
    
\begin{align*}
    \delta(\gamma_1,\gamma_2)=0 
    &\Leftrightarrow \left|\left|\mathbbold 1_{\gamma_1} - \mathbbold 1_{\gamma_2}\right| \right|^2_{2}=0\\
    &\Leftrightarrow \mathbbold 1 _{\gamma_1}(x) = \mathbbold 1 _{\gamma_2}(x) \text{ for } \lambda \text{-almost every }x \in \Xdom \\
    &\Leftrightarrow \gamma_1 = \gamma_2 ~\lambda \text{-almost everywhere.}
\end{align*}

\end{proof}

\subsubsection{Proof of Proposition \ref{prop_k_def_pos}}
\label{proof:kdefpos}
\ksetkernel*
\begin{proof}
The proof is similar as the one of Lemma 2.1. in \textcite{set-index-process-herbin} but is recalled and adapted here.

Let $\alpha_1,...,\alpha_n \in \Real$ and $\gamma_1,...,\gamma_n \in \mathscr L ^*(\Xdom)$. Let's show that $$\sum_{kl}^n e^{- \frac{\lambda(\gamma_k\Delta \gamma_l)}{2\sigma^2}}\alpha_k \alpha_l \geq 0.$$
First let's use Lemma \ref{lemme} to write: 
$$\sum_{kl}^n e^{- \frac{\lambda(\gamma_k\Delta \gamma_l)}{2\sigma^2}}\alpha_k \alpha_l = \sum_{kl}^n e^{-\frac{\left|\left|\mathbbold 1_{\gamma_k} - \mathbbold 1_{\gamma_l}\right| \right|^2_{L^2(\Xdom)}}{2\sigma^2}}\alpha_k \alpha_l.$$
Then we use the Bochner-Milos theorem (Theorem 1.1 of \textcite{hida2013white}), which says that there exists a random variable $G \in L^2(\Xdom)$ such that
$$ \forall f \in L^2(\Xdom), ~ \Esp (e^{i\langle f,G \rangle_{L^2(\Xdom)}})=e^{-\frac{1}{2} ||f||^2_{L^2(\Xdom)}}.$$ This allows to derive : 
\begin{align*}
\sum_{kl}^n e^{- \frac{\lambda(\gamma_k\Delta \gamma_l)}{2\sigma^2}}\alpha_k \alpha_l
&= \sum_{kl}^n e^{-\frac{\left|\left|\mathbbold 1_{\gamma_k} - \mathbbold 1_{\gamma_l}\right| \right|^2_{L^2(\Xdom)}}{2\sigma^2}}\alpha_k \alpha_l\\
&= \sum_{kl}^n \Esp e^{i \langle \frac{\mathbbold 1_{\gamma_k} - \mathbbold 1_{\gamma_l}}{\sqrt 2 \sigma }, G \rangle_{L^2(\Xdom)}}\alpha_k \alpha_l \\
&= \Esp \sum_{kl}^n e^{\frac{i}{\sqrt 2 \sigma } \langle \mathbbold 1_{\gamma_k} , G \rangle_{L^2(\Xdom)}}\alpha_k e^{-\frac{i}{\sqrt 2 \sigma } \langle \mathbbold 1_{\gamma_l}, G \rangle_{L^2(\Xdom)}}\alpha_l \\
&= \Esp \left|\sum_{k}^n e^{\frac{i}{\sqrt 2 \sigma } \langle \mathbbold 1_{\gamma_k} ,G \rangle_{L^2(\Xdom)} }\alpha_k \right|^2 \geq 0.
\end{align*}
\end{proof}

\subsubsection{Proof of Proposition \ref{prop_k_bounded_mes}}
\label{proof:kboundedmes}
\kboundedmes*
\begin{proof}
$k_{set}$ is clearly bounded by $1$. In term of measurability, we are studying the measurability of $k_{set}$ with respect to the two Borels $\sigma$-algebras $\Bdom( \mathscr L_2^\star (\Xdom),\delta_2)$ and $\Bdom(\Real,|.|)$ with $\delta_2((\gamma_1^A,\gamma_1^A),(\gamma_2^B,\gamma_2^B))=\delta (\gamma_1^A,\gamma_1^B) + \delta (\gamma_2^A,\gamma_2^B)$. As $\delta : \mathscr L_2^\star (\Xdom) \rightarrow \Real^+$ is a distance on $\mathscr L^\star (\Xdom)$, it is a continuous function from $\mathscr L_2^\star (\Xdom)$ to $\Real$. $k_{set}$ is then continuous as $\exp$ is continuous.  Finally the continuity of $k_{set}$ implies measurability as we are working in two Borel $\sigma$-algebras. 
\end{proof}

\subsubsection{Proof of Proposition \ref{charac}}
\label{proof:kcharac}
\charackernel*
\begin{proof}
The proof is based on the Proposition 5.2. of \textcite{charac_kernel} which is recalled here with our notations.
\begin{prop}
Let $\mathcal{P}$ be a Polish space, $H$ a separable Hilbert space, $T$ a measurable and injective mapping from $\mathcal{P}$ to $H$, and $\varphi$ the Laplace transform of a finite Borel measure $\nu$ on $[0,+\infty)$ such that $\nu \ne 0$ and supp$\nu \ne \{0\}$. Then, the kernel $k$ on $\mathcal{P} \times \mathcal{P}$ defined by
$$
k\left(x, x^{\prime}\right):=\varphi\left(\left\|T(x)-T\left(x^{\prime}\right)\right\|^2_H\right), \quad (x, x^{\prime}) \in \mathcal{X}^2
$$
is integrally strictly positive definite with respect to $\mathcal{M}(\mathcal{X})$.
\end{prop}%
Using the previous Proposition, we will show that $k_{set}$ is integrally strictly positive which is a sufficient condition to be characteristic (see Theorem 7 in \textcite{sriperumbudur_universality_2010}).

To apply the Proposition, we first introduce the notation,
    $$\mathscr{F}_{\text {bin }}^{\star}=\{f \in L^2(\Xdom): \exists \gamma \in \mathscr{L}(\mathcal{X}) \text { such that } f=\mathbbold{1}_\gamma~ \lambda \text {-almost everywhere} \}.$$

We want to apply the proposition with $\mathcal{P}=\mathscr L^*(\Xdom)$, $H=L^2(\Xdom)$, $\varphi (\cdot) = \exp ( - \frac{\cdot}{2\sigma^2} )$ and $T$ defined by 
$$\begin{matrix}
 T : & \mathscr{L}^{\star}(\mathcal{X}) & \longrightarrow & \mathscr{F}_{\text {bin }}^{\star} \subset H\\ 
 & [\gamma]_\delta & \longmapsto & \left[\mathbbold{1}_\gamma\right]_{2},
\end{matrix}$$
where $[\cdot]_\delta$ and $[\cdot]_{2}$ denote equivalence classes in $\mathscr{L}(\mathcal{X})$ and $\mathscr{F}_{\text {bin }}$ respectively. We need to show that $\mathscr L^*(\Xdom)$ is Polish, that $L^2(\Xdom)$ is a separable Hilbert space, and that $T$ is measurable and injective.
\begin{itemize}
    \item Let's first show that is $T$ a well-defined measurable and injective mapping from $\mathcal{P}$ to $H$:
\begin{itemize}
    \item $T$ is well-defined and in the same time injective as for any $\gamma_1,\gamma_2 \in \mathscr{L}(\Xdom)$,
    \begin{align*}
        [\gamma_1]_{\delta}=[\gamma_2]_{\delta} &\Leftrightarrow \delta(\gamma_1,\gamma_2)=0 \\
        & \Leftrightarrow \lambda(\gamma_1 \Delta \gamma_2)=0 \\
        & \Leftrightarrow \left\|\mathbbold 1_{\gamma_1}-\mathbbold 1_{\gamma_2}\right\|_{2}=0 \\
        & \Leftrightarrow \left[\mathbbold{1}_{\gamma_1}\right]_{2}=\left[\mathbbold{1}_{\gamma_2}\right]_{2}.
    \end{align*}
    \item $T$ is measurable (with respect to the Borel $\sigma$-algebras $\mathcal{B}\left(\mathscr{L}^{\star}(\mathcal{X}), \delta\right)$ and $\mathcal{B}\left(\mathscr{F}_{\text {bin }}^{\star} , ||\cdot||_{2}\right)$) because it is continuous as it is an isometry from $\left(\mathscr{L}^{\star}(\mathcal{X}), \sqrt{\delta}\right)$ to $\left(\mathscr{F}_{\text {bin }}^{\star} , ||\cdot||_{2}\right)$.
    \item Let us show that $T$ is also surjective which will be useful later on. Let $[f]_{2} \in  \mathscr{F}_{\text {bin }}^{\star}$. There exist $\gamma \in \mathscr L (\Xdom)$ such that $f=\mathbbold 1 _{\gamma}$ $\lambda$-almost everywhere which implies that $||f-\mathbbold 1_\gamma ||_{2}=0$ i.e. $\left[f\right]_{2}=\left[\mathbbold{1}_{\gamma}\right]_{2}=T([\gamma]_{\delta})$.
\end{itemize}
\item $L^2(\Xdom)$ is a separable Hilbert space as $\Xdom$ is compact so separable.
\item Let us now show that $\mathscr L^*(\Xdom)$ is Polish i.e. a topological space homeomorphic to a separable complete metric space. Luckily, as $T$ is a surjective isometry, it is an homeorphism from $\mathscr{L}^{\star}(\mathcal{X})$ to $\mathscr{F}_{\text {bin }}^{\star}$ which is a metric space. Then, it only remains to prove that $\mathscr{F}_{\text {bin }}^{\star}$ is both complete and separable.
\begin{itemize}
    \item $\mathscr{F}_{\text {bin }}^{\star} \subset L^2(\Xdom)$ which is separable so $\mathscr{F}_{\text {bin }}^{\star}$ is separable.
    \item As $L^2(\Xdom)$ is complete, having  $\mathscr{F}_{\text {bin }}^{\star}$ closed is a sufficient condition for $\mathscr{F}_{\text {bin }}^{\star}$ to be complete. Let us show then that it is closed. 
     Let $\mathbbold 1 _{\gamma_n}  \overset{L_2}{\underset{n\rightarrow +\infty}{\rightarrow}}  f $
with $\gamma_n \in \mathscr L(\Xdom)~\forall n$. The $L_2$ convergence implies that there is a sub-sequence $(\mathbbold 1 _{\gamma_{\phi(n)}})_n$ that converges almost everywhere pointwise to $f$ (as stated in Theorem 3.13 of \textcite{book:functionalanalysis}). It means that there exists a $\lambda$-null set $\mathcal N$ s.t. $\forall x \notin \mathcal N, \mathbbold 1 _{\gamma_{\phi(n)}}(x) \underset{n\rightarrow+\infty}{\rightarrow} f(x)$. As $\mathbbold 1 _{\gamma_{\phi(n)}}(x)$ is a sequence of $0$ and $1$, we have that its limit, $f(x)$, belong to $\{0,1\}$. So, $f= \mathbbold 1_{ f^{-1}(\{1\})}~ \lambda$-almost everywhere and $f^{-1}(\{1\}) \in \mathscr L(\Xdom)$ as $f$ is measurable. So $f \in \mathscr{F}_{\text {bin }}$. $\mathscr{F}_{\text {bin }}^{\star}$ is thus closed.
\end{itemize}
\end{itemize}

Thus $k_{set}$ is integrally strictly positive definite with respect to $\mathcal M ( \mathscr L^*(\Xdom))$, the set of signed measure on $ \mathscr L^*(\Xdom)$, which implies that it is characteristic.
\end{proof}

\subsubsection{Proof of Proposition \ref{prop nested}}
\label{proof:knested}
\nestedestimator*
\begin{proof}
Let $f(\bm u,z)=g(\bm u)h(z)$ with $g(\bm u) = \left(K_{A}\left(\bm u_1, \bm u_2\right)-1\right)$, $h(z)=e^{-\frac{\lambda(\Xdom)}{2\sigma^2} \bm z}$, and let $\Phi (\bm x, \gamma_1,\gamma_2  )=\mathbbold 1 _{\gamma_1 \Delta \gamma_2} (x)$. Let us have iid samples $(\bm U_A^{(i)},\Gamma^{(i)})$,  $i=1,...,n$ of $(\bm U_A,\Gamma)$ and  $(\bm X^{(1)},. ..,\bm X^{(m)})$ of $\bm X \sim \Udom (\Xdom)$. We denote

\[\operatorname{H}= \operatorname{HSIC}(\bm U_A,\Gamma) = \Esp [ f(\bm U,\bm U',\Esp [ \Phi (\bm X, \Gamma,\Gamma') |(\Gamma,\Gamma')]],\]

\[
\operatorname{H}_{n,m}=\widehat{\operatorname{HSIC}}_{u}\left(\bm U_A, \Gamma\right)= \frac{2}{n(n-1)}\sum_{i < j}^n f (\bm U_A^{(i)},\bm U_A^{(j)},\frac{1}{m}\sum_{k=1}^{m} \Phi (\bm X^{(k)},\Gamma^{(i)},\Gamma^{(j)}),
\] and
\[
\operatorname{H}_{n} =\frac{2}{n(n-1)} \sum_{i < j}^n f(\bm U_A^{(i)},\bm U_A^{(j)},\Esp [ \Phi (\bm X,\Gamma^{(i)},\Gamma^{(j)}) | (\Gamma^{(i)},\Gamma^{(j)})]).
\]
First we split the risk into two terms:
\begin{equation*}
\Esp | \operatorname{H}_{n,m}- \operatorname{H} |^2 \leq 2 \Esp | \operatorname{H}_{n} - \operatorname{H} |^2 + 2 \Esp |\operatorname{H}_{n} - \operatorname{H}_{n,m}|^2
\end{equation*}
The first term is the variance of a classic U-statistic of order 2:
\begin{equation*}
 \Esp | \operatorname{H}_{n} - \operatorname{H} |^2 = \frac{2\sigma_1^2}{n(n-1)}+ \frac{4(n-2)\sigma_2^2}{n(n-1)}.
\end{equation*}

The second term can be developed:

\begin{equation*}
\Esp |\operatorname{H}_{n} - \operatorname{H}_{nm}|^2 = \frac{4}{n^2(n-1)^2} \left( \sum_{i< j}^n \sum_{p< l}^n  \Esp \left( E_{ij} - E_{ij,m} \right) \left( E_{pl} - E_{pl,m}\right) \right),
\end{equation*}
with
$$E_{ij}=f(\bm U_A^{(i)},\bm U_A^{(j)},\Esp [ \Phi (\bm X,\Gamma^{(i)},\Gamma^{(j)}) | (\Gamma^{(i)},\Gamma^{(j)})])$$ and $$E_{ij,m}=f (\bm U_A^{(i)},\bm U_A^{(j)},\frac{1}{m}\sum_{k=1}^{m} \Phi (\bm X^{(k)},\Gamma^{(i)},\Gamma^{(j)}).$$

As $\bm X^{(1)},...,\bm X^{(m)}$ are common to each $E_{ij,m}$, the terms $E_{ij,m}$ and $E_{pl,m}$ are not independent even if $i,j,p,l$ are pairwise distinct. We can still bound them but we will lose one order of convergence in $m$. We first have
\begin{align}
\label{node}
\left| \Esp \left( E_{ij} - E_{ij,m} \right) \left( E_{pl} - E_{pl,m}\right) \right | &\leq  (\Esp \left| E_{ij} - E_{ij,m} \right|^2 \Esp \left| E_{pl} - E_{pl,m}\right|^2)^{\frac{1}{2}} \\
&= \left(\Esp \left| E_{12} - E_{12,m} \right|^2\right)  \label{ineq2}\\
&\leq \frac{L^2}{m} \Esp \left[ \left(K_{A}\left(\bm U_A, {\bm U_A}'\right)-1\right)^2\operatorname{Var}\left(  \Phi(\bm X,\Gamma,\Gamma') | (\bm U_A,{\bm U_A}',\Gamma,\Gamma')\right)\right], \nonumber
\end{align}
using Cauchy-Schwarz inequality in (\ref{node}) and that $h$ is $L$-lipschitz.
Summing each term, we obtain
\begin{align*}
\Esp |\operatorname{H}_{n} - \operatorname{H}_{nm}|^2 \leq \frac{L^2}{m} \Esp \left[ \left(K_{A}\left(\bm U_A, {\bm U_A}'\right)-1\right)^2\operatorname{Var}\left(  \Phi(\bm X,\Gamma,\Gamma') | (\bm U_A,{\bm U_A}',\Gamma,\Gamma')\right)\right].
\end{align*}
Putting all results together, we get
\begin{align*}
\Esp | \operatorname{H}_{n,m}- \operatorname{H} |^2 \leq 2 \left( \frac{2\sigma_1^2}{n(n-1)}+ \frac{4(n-2)\sigma_2^2}{n(n-1)}+  \frac{L^2 \sigma_3^2}{m} \right).
\end{align*}
\end{proof}

\subsubsection{Convergence rate in the case of Independent $m$ sample}
\label{annex:indep}
If an $m$ sample $\bm X_{ij}^{(k)}$ is drawn independently for each $(i,j)$, we can obtain an asymptotic rate of $\mathcal O (\frac{1}{n}+\frac{1}{m^2})$. Indeed in (\ref{node}) we use the independence between $(\bm X_{ij}^{(k)})_k$ and $(\bm X_{pl}^{(k)})_k$ for $i,j,p,l$ pairwise distinct:

\begin{align*}
    \left| \Esp \left( E_{ij} - E_{ij,m} \right) \left( E_{pl} - E_{pl,m}\right) \right | &= \left|\Esp \left( E_{ij} - E_{ij,m} \right) \Esp \left( E_{pl} - E_{pl,m}\right)\right|\\
    &=\left|\Esp \left( E_{12} - E_{12,m} \right)\right|^2
\end{align*}
Then by applying Taylor Lagrange's formula to $h$ with $a=\Esp(\Phi (\bm X,\Gamma^{(1)},\Gamma^{(2)})|(\Gamma^{(1)},\Gamma^{(2)}))$ and $b=\frac{1}{m}\sum_{k=1}^{m} \Phi (\bm X^{(k)},\Gamma^{(1)},\Gamma^{(2)})$, we have the existence of $\theta \in \Real^+$ such that,
 \begin{align*}
 E_{12,m}-E_{12}&=\left(K_{A}\left(\bm U_A, {\bm U_A}'\right)-1\right)\left( h(b) - h(a)\right) \\
 &=\left(K_{A}\left(\bm U_A, {\bm U_A}'\right)-1\right) \left[ h'(a)(b-a) +\frac{h''(\theta)}{2}(b-a)^2\right].
 \end{align*}
Then we take the expectation and use the tower property and we use that Monte Carlo estimators are unbiased (i.e. $\Esp(b|(\Gamma^{(1)},\Gamma^{(2)})) = a$) which leads to
\begin{align*}
    \Esp( E_{12,m}-E_{12}) &= \Esp \left[ \Esp( E_{12,m}-E_{12} |\bm U_A,\bm U_A',\Gamma,\Gamma') \right]\\
    &=\Esp \left[ \left(K_{A}\left(\bm U_A, {\bm U_A}'\right)-1\right) \Esp \left( \frac{h''(\theta)}{2}(b-a)^2 | \bm U_A,{\bm U_A}',\Gamma,\Gamma'\right) \right].
\end{align*}
Then, as $h''$ is bounded by $L$, 
\begin{equation*}
    \left| \Esp( E_{12,m}-E_{12}) \right| \leq \frac{L^2}{2m}\Esp \left[ \left|K_{A}\left(\bm U_A, {\bm U_A}'\right)-1\right| \operatorname{Var}\left(  \Phi(\bm X,\Gamma,\Gamma') | (\bm U_A,{\bm U_A}',\Gamma,\Gamma')\right) \right],
\end{equation*}
which finally leads to
\begin{equation*}
     \left| \Esp \left( E_{ij} - E_{ij,m} \right) \left( E_{pl} - E_{pl,m}\right) \right | \leq \frac{L^4}{4m^2}\sigma_4^4,
\end{equation*}
where 
$$
\sigma_4^2=\Esp \left[ \left|K_{A}\left(\bm U_A, {\bm U_A}'\right)-1\right| \operatorname{Var}\left(  \Phi(\bm X,\Gamma,\Gamma') | (\bm U_A,{\bm U_A}',\Gamma,\Gamma')\right) \right].$$

If $i=p$ or $i=l$ or $j=p$ or $j=l$, we lost the independence so we use the previous result that 
\begin{equation*}
     \left| \Esp \left( E_{ij} - E_{ij,m} \right) \left( E_{pl} - E_{pl,m}\right) \right | \leq \frac{L^2}{m}\sigma_3^2.
\end{equation*}

There are $\frac{n(n-1)(n-2)(n-3)}{4}$ $i,j,p,l$ pairwise distinct with $i < j$ and $p < l$ so we finally obtain that,

\begin{align*}
\Esp | \operatorname{H}_{n,m}- \operatorname{H} |^2 &\leq 2 \left( \frac{2\sigma_1^2}{n(n-1)}+ \frac{4(n-2)\sigma_2^2}{n(n-1)}+  \frac{L^2 2(2n-3) \sigma_3^2}{n(n-1)m} +\frac{L^4(n-2)(n-3) \sigma_4^4}{4n(n-1)m^2}\right)\\
&= \mathcal O (\frac{1}{n} + \frac{1}{m^2}).
\end{align*}

\subsection{Figures}
\label{annex:fig}
\begin{figure}[h]
  \begin{minipage}{0.52\textwidth} 
    \centering
   \begin{tabular}{ccccccc}
\toprule
& $U_1$ & $U_2$ & $U_p$ & $U_{r_1}$ & $U_{r_2}$ & $U_{r_3}$ \\
\midrule
$K_{Sob}$ & $5$ & $95$ & $55$ & $0$ & $75$ & $90$ \\
$K_{Sob}$ & $5$ & $100$ & $60$ & $0$ & $75$ & $95$ \\
$K_{Sob}$ & $15$ & $95$ & $55$ & $0$ & $65$ & $95$ \\
$K_{Sob}$ & $10$ & $95$ & $55$ & $0$ & $65$ & $95$ \\
$K_{Sob}$ & $5$ & $100$ & $55$ & $0$ & $70$ & $95$ \\
\bottomrule
\end{tabular}
    \subcaption{Acceptance rate ($\%$) over $20$ independence tests with a risk of $5\%$}
    \label{tab:acceptance-rate-g1andg2}
  \end{minipage}%
  \begin{minipage}{0.48\textwidth} 
    \centering
    \scalebox{0.9}{\input{Images/oscillator/g1_and_g2/hsic_anova_3_no_pvalue_only_1sr_order}}
    \subcaption{Estimations of $\hat{\hat{S}}^{\operatorname{H}_{set}}_{i}$}
    \label{fig:indices:g1andg2}
  \end{minipage}
  \caption{Acceptance rate (\subref{tab:acceptance-rate-g1andg2}) and estimation of $\hat{\hat{S}}^{\operatorname{H}_{set}}_{i}$  (\subref{fig:indices:g1andg2}) for excursion set
$\Gamma_{(g_1,g_2)} $  computed for 5 kernels with $n=100$, $m=100$ and repeated 20 times}
\label{hsic_anova_g1andg2}
\end{figure}

\begin{figure}[h]
  \centering

  \begin{subfigure}{0.45\textwidth}
    \scalebox{0.9}{\input{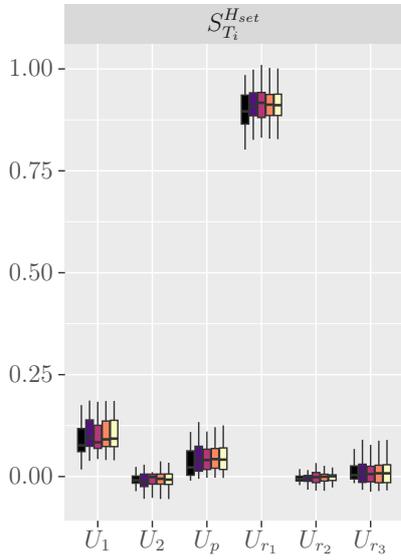}}
    \caption{$\Gamma_{g_1}$}
    \label{fig:sub1}
  \end{subfigure}
  \hfill
  \begin{subfigure}{0.45\textwidth}
    \scalebox{0.9}{\input{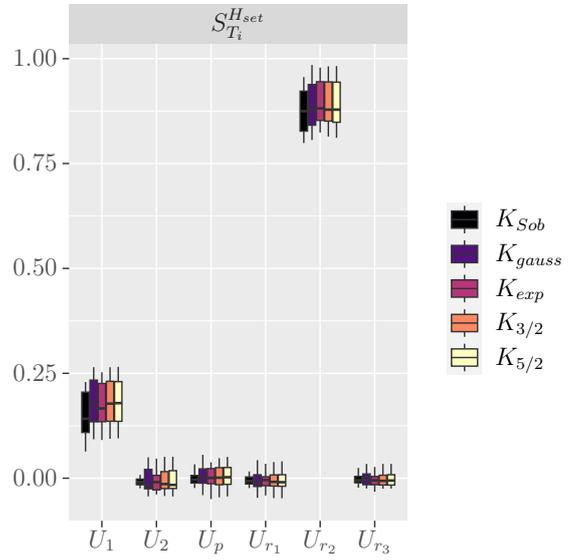}}
    \caption{$\Gamma_{g_2}$}
    \label{fig:sub2}
  \end{subfigure}

  \vspace{1em}  

  \begin{subfigure}{0.45\textwidth}
    \scalebox{0.9}{\input{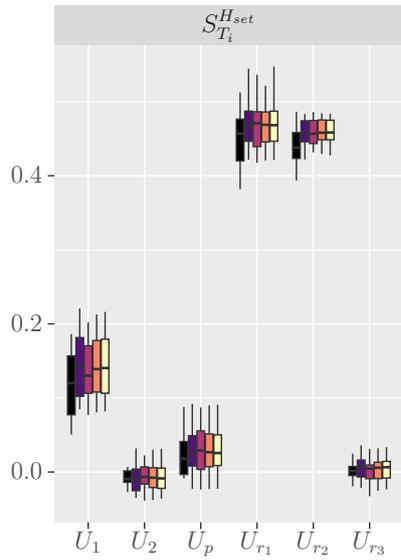}}
    \caption{$(\Gamma_{g_1},\Gamma_{g_2})$}
    \label{fig:sub3}
  \end{subfigure}
  \hfill
  \begin{subfigure}{0.45\textwidth}
    \scalebox{0.9}{\input{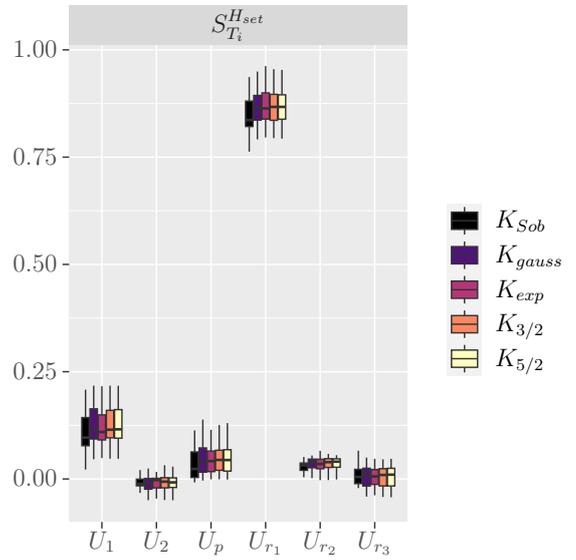}}
    \caption{$\Gamma_{(g_1,g_2)}$}
    \label{fig:sub4}
  \end{subfigure}

  \caption{Estimation of the total-order indices $\hat{\hat{S}}^{\operatorname{H}_{set}}_{T_i}$  for the oscillator case computed for 5 kernels with $n=100$, $m=100$ and repeated 20 times}
  \label{fig:totalordercousin}
\end{figure}

\begin{figure}
\centering
 \resizebox{\textwidth}{!}{\input{Images/Adan/hsic_anova_adan_100_ustat_420_7_2D}}
\caption{Estimation of $\hat{\hat{S}}^{\operatorname{H}_{set}}_{i}$ for the excursion sets $\Gamma_{f_1}$, $\Gamma_{f_2}$ and the pair $(\Gamma_{f_1},\Gamma_{f_2})$ with $-q_1=420$ N.m and $q_2=7\%$ computed with the Sobolev input kernel and with $n=100$, $m=100$. $20$ replicates.}
\label{fig:hsic_anova_adan_2d}
\end{figure}

\end{document}